\documentclass[11pt, A4,obeyspaces,hyphens,spaces]{article} 
\usepackage[left=2.5cm,right=2.5cm,top=2.5cm,bottom=2.5cm]{geometry} 

\usepackage{authblk}

\usepackage{graphicx}%
\usepackage{multirow}%
\usepackage{amsmath,amssymb,amsfonts}%
\usepackage{amsthm}%
\usepackage{mathrsfs}%
\usepackage[title]{appendix}%
\usepackage{xcolor}%
\usepackage{textcomp}%
\usepackage{manyfoot}%
\usepackage{booktabs}%
\usepackage{algorithm}%
\usepackage{algorithmicx}%
\usepackage{algpseudocode}%
\usepackage{listings}%

\usepackage{enumerate} 
\usepackage{subcaption}

\usepackage[dvipsnames]{xcolor}

\usepackage{indentfirst,stmaryrd,qsymbols,graphicx,amsfonts, amsthm,ulem,dsfont, footmisc, scalerel,stackengine,xspace,amsmath}
\usepackage{fourier}
\usepackage{tikz}
\usepackage{amsmath,hyperref,cleveref}

\usepackage{enumitem,enumerate} 

\usepackage{array}

\DeclareMathOperator{\Leb}{Leb}

\DeclareMathOperator{\argmin}{argmin}

\DeclareMathOperator{\cov}{cov}

\def \1{\mathds{1}}

\def \Cost{{\sf Cost}}

\def \LL{{\L}ukasiewicz }

\def \Mes{{\sf Mes}}

\def \P{\mathbb{P}}
\def \R{\mathbb{R}}

\def \UCost{{\sf UCost}}

\newcommand{\stkout}[1]{\ifmmode\text{\sout{\ensuremath{#1}}}\else\sout{#1}\fi}

\newcommand\ZnZ{\mathbb{Z}/n\mathbb{Z}}
\def \Z{\mathbb{Z}}

\def \app#1#2#3#4#5{\begin{array}{rccl} #1:&#2&\longrightarrow&#3\\ &#4&\longmapsto&#5\end{array}}
\def \as{\xrightarrow[n]{(as.)}}
\def \bA#1{\textbf{A}^{(#1)}}
\def \bB#1{\textbf{B}^{(#1)}}

\def \bD{{\bf D}}
\def \bD{{\bf D}} 
\newcommand\bF[1]{\textbf{F}^{(#1)}}
 
\def \bI#1{\textbf{I}^{(#1)}}
\def \bL#1{\textbf{L}^{(#1)}}
\def \bM{\begin{bmatrix}}
\def \bN{{\bf N}}
\newcommand\bO[1]{\textbf{O}^{(#1)}}

\def \bS{{\bf S}}

\def \bU{{\bf U}}

\def \bX{{\bf X}}

\def \bar{\overline}
\def \ba{\begin{align}}
\def \ea{\end{align}}
\def \ba{{\bf a}}
\def \ba{{\bf a}}

\def \bb{{\bf b}}
\def \bb{{\bf b}}
\def \bc{{\bf c}}
\def \ben{\begin{eqnarray}}
\def \beq{\begin{equation}}
\def \beqs{\begin{equation*}}
\def \be{\begin{eqnarray*}}

\def \bia{\begin{itemize}\compact \setcounter{d}{0}}

\def \bir{\begin{itemize}\compact \setcounter{c}{0}}
\def \bis{\begin{itemize}\compact }
\def \bi{\begin{itemize}\compact \setcounter{b}{0}}

\def \bls{{\tiny $\blacksquare$ }}

\def \bm{{\bf m}}
\def \bpar#1{\left\{\begin{array}{#1} }

\def \bs{{\bf s}}

\def \build#1#2#3{\mathrel{\mathop{\kern 0pt#1}\limits_{#2}^{#3}}}

\def \bu{{\bf u}}
\def \bv{{\bf v}}

\def \captionn#1{\begin{center}\sf\caption{\small #1}\end{center}}
\def \ceil#1{\lceil#1\rceil}

\def \da{\downarrow}
\def \dd#1{\frac{\partial}{\partial #1}}
\def \dd{\xrightarrow[n]{(d)}}

\def \dis{\displaystyle}

\def \eM{\end{bmatrix}}
\def \een{\end{eqnarray}}
\def \ee{\end{eqnarray*}}
\def \eia{\end{itemize}\vspace{-2em}~}
\def \eir{\end{itemize}\vspace{-2em}~}
\def \eis{\end{itemize}\vspace{-2em}~}
\def \ei{\end{itemize}\vspace{-2em}~}

\def \epar { \end{array}\right.}
\newcommand{\eqd}{\overset{(d)}{=}}

\def \eq{\end{equation}}
\def \eqs{\end{equation*}}
\def \eref#1{(\ref{#1})}

\def \floor#1{\lfloor#1\rfloor}

\def \ita{\addtocounter{d}{1}\item[(\alph{d})]}

\def \itr{\addtocounter{c}{1}\item[($\roman{c}$)]}

\def \l{\left}

\def \proba{\xrightarrow[n]{(proba.)}}
\def \probabis{\xrightarrow[]{(proba.)}}

\def \r{\right}

\def \sous#1#2{\mathrel{\mathop{\kern 0pt#1}\limits_{#2}}}

\def \under{{\sf under}}

\def\Id{{\sf Id}}
\def\bma{ \begin{bmatrix}}
\def\cro#1{\llbracket#1\rrbracket}

\def\ema{\end{bmatrix}}
\definecolor{amber}{rgb}{1.0, 0.75, 0.0}

\newcommand{\E}{\mathbb{E}}

\newcommand{\compact}{ \topsep0pt   \itemsep=0pt   \partopsep=0pt   \parsep=0pt}

\newcommand{\se}{{\sf e}}
\newcounter{b}
\newcounter{c}
\newcounter{d}
\newqsymbol{`e}{\varepsilon}
\renewcommand{\mod}{{~\sf mod~}}

\usepackage{soul}
\usepackage{xspace}

\newcommand{\ind}[1]{\mathds{1}_{#1}}

\newcommand{\Dirichlet}{{\sf Dirichlet}}

\definecolor{blue_drop}{RGB}{78,197,255}

\def\d{\mathrm{d}}

\newcommand\CDM{continuous dispersion model\xspace}
\newcommand\ACDM{{\sf CDM}\xspace}

\newcommand\CDMs{continuous dispersion models\xspace}
\newcommand \DDM{discrete dispersion model\xspace} 
\newcommand \ADDM{DDM\xspace} 
\def \fleche#1{{#1}_{\rightarrow}}
\def \flecheu#1{\overrightarrow{#1}}

\def \Var{\textrm{Var}}
\def \Comp{{\sf Comp}}

\def \IR{{\sf IR}}
\def \Law{{\sf Law}}

\counterwithin*{equation}{section}

\newcounter{aa}
\def \paraa#1{\noindent\bls\textbf{\addtocounter{aa}{1} \bgroup\bfseries (\Alph{aa})\egroup~~ #1}}

\newcommand{\modif}{}
\newcommand{\modifs}{}
\def \modd#1{#1}
\theoremstyle{thmstyleone}%
\newtheorem{theorem}{Theorem}[section]

\newtheorem{pro}[theorem]{Proposition}
\newtheorem{theo}[theorem]{Theorem}
\newtheorem{claim}[theorem]{Claim}
\newtheorem{lem}[theorem]{Lemma}

\theoremstyle{thmstyletwo}%
\newtheorem{rem}[theorem]{Remark}%
\newtheorem{note}{Note}
\theoremstyle{thmstylethree}%
\newtheorem{defi}[theorem]{Definition}%

\begin{document}
	
	\title{Dispersion models on a circle: universal properties and asymptotic results}

	\author[1]{{Jean-Fran\c{c}ois} {Marckert}}
	
	\author[1,2]{{Zoé} {Varin}}

	\affil[1]{Univ. Bordeaux, CNRS, Bordeaux INP, LaBRI, UMR 5800, F-33400 Talence, France}
	
	\affil[2]{Universit\'{e} Paris Cit\'{e}, CNRS, IRIF, F-75013, Paris, France}
	
	\maketitle

	\abstract{	Consider a sequence of masses $m_0,m_1,...$ arriving independently at uniform random  points $u_0,u_1,...$ on the unit circle $\mathbb{R}/\mathbb{Z}$ (or on $\mathbb{Z}/n\mathbb{Z}$, in the discrete version). Upon arrival, each mass undergoes a relaxation phase during which it is dispersed, possibly also at random. This process can model many physical phenomena, such as the diffusion of liquid in a porous medium. In the discrete case, it can model parking (related to additive coalescence and hashing with linear probing) in which the cars are permitted to follow random displacement policies.
		
		The dispersion policies considered in the paper ensure that at time $k$, after the successive dispersions of $m_0,\cdots,m_{k-1}$, the total covered region has Lebesgue measure $m_0+\cdots+m_{k-1}$. Furthermore, during the dispersion of a given mass, the covered domain increases continuously, except when it merges with another covered connected component (CC).
		
		We show a very general exchangeability property for the sequence of covered CC. Additionally, we demonstrate a universal spacing property between these CC, and a notable general result:  the number of free (not covered) CC follows a binomial distribution whose parameters depend solely on the number and total mass of arrived particles. Furthermore, conditional on the number of CC, the sizes of the free CC follow a simple Dirichlet distribution in the continuous case, regardless of the dispersion policy considered and the values of the masses. We also characterise the distribution of the occupied space.
		
		In the second part of the paper, we study the total cost associated with these models for various cost models, and establish connections with the additive coalescent. We also provide an asymptotic representation of the limiting covered space as the number of masses goes to infinity.}

%	\keywords{probability, limit theorems, parking models, universality}
%
%	
%	\pacs[MSC Classification]{Primary 60C05, 60F99, secondary 70L99}

	\section{Introduction}

	The dynamics we are interested in are quite general, and represent physical systems in which some matter arrives successively at random on a given surface $S$. Upon arrival, the matter undergoes a relaxation phase, during which it %this matter 
	is dispersed %on the surface 
	until it covers a new domain (that we will call occupied domain), whose  area is proportional to its initial mass. 
	These types of models are common in many fields of sciences, particularly physics, chemistry, biology, computer science, combinatorics, and probability theory. At this stage, we can use two main images to illustrate the phenomenon we want to study: \\
	-- We can imagine water droplets of various random sizes (as in the case of rain) falling  randomly on a structure, and moistening it, according to a  (possibly random) dispersion dynamic, depending on the external conditions (wind, temperature, etc.). In this case, the space is the structure, and the connected components (CC)\label{debut} of the moistened area constitute the ``occupied domain''. The moistened area can be imagined as having an area proportional to the total mass of water deposited (or equal to it, up to a change of unity).\\
	-- We can imagine some cars (all of the unit size), arriving one after the other in a parking lot. When the $i$th car arrives, it chooses a place $c_i$ uniformly at random, and  parks in the nearest available place. In this case, the problem is discrete: the space is the parking lot, and the occupied domain is the set of occupied places.  This is the classical parking model (see Section \ref{sect:related_work} for more details). \medskip
	
	In order to fully define our models, we need to define four main characteristics:\\
	$(a)$ the space/surface on which the matter arrives (in the paper it will be one-dimensional, discrete or continuous, and cyclic),\\
	$(b)$ the masses distribution; the masses $(m_i)$ can be random or not, of discrete nature or not,\\
	$(c)$ the arrival positions $(\bu_i)$ of the masses: they may be of discrete or continuous nature, but they will be random, and uniform on a well-chosen space, \\
	$(d)$ and the ``dispersion'' algorithm, which will be the dynamics, random or not, for a mass $m_i$ arriving at position $\bu_i$, according to which it will cover a part of the space with Lebesgue measure $m_i$, and remain  there eventually. \medskip{} 
	
	We put in bold the notation $\bu_i$ to distinguish its ``random variable'' type. The masses $(m_i,i\geq 0)$ are fixed simple non-negative real numbers at the beginning of the paper. They will be random in some future sections, and we will change the font to bold, when it will be the case.
	
	This paper is devoted to obtaining exact and asymptotic results regarding the distribution of the occupied and free components after dispersion of several masses. Among results, we present universality results concerning the distribution of the lengths of occupied (resp. free) components that are valid under very weak assumptions about the dispersion/covering policies. Roughly speaking, the distribution of the occupied connected  component  depends on the masses $m_0,\cdots,m_{k-1}$, while the distribution of the lengths of the free connected components depend only on $\l(k,\sum_{i=0}^{k-1} m_i\r)$. These distributions will be shown to be independent of the chosen dispersion policy.
	
	The next section provides an overview of the results obtained in the paper.

	\pgfdeclarelayer{background}
	\pgfsetlayers{background,main}
	\begin{figure}[h!] \centering
		\begin{subfigure}[b]{0.4\textwidth} \centering
			\begin{tikzpicture}[scale=1.3]
				\def\ray{1.25cm} \def\raysupvert{1.6cm} \def\raysuphoriz{1.5cm}
				\clip(-\raysuphoriz,-\raysupvert) rectangle (\raysuphoriz,1.9cm);
				\draw[black] (0,0) circle (\ray);
				
				\draw (-90:\ray-0.08cm) -- (-90:\ray+0.08cm);
				\node[below] at (-90:\ray) {{$0$}}; 
				
				\foreach \drop/\stepa/\stepb/\angledrop/\anglegauche/\angledroite/\decalx/\decaly in {0/2/3/50/30/60/0.5/0,1/4/5/230/190/230/0/-0.8, 2/6/7/110/80/120/0.5/0}{
					
					{ 
						\draw[color=black, thick ] (\angledrop:\ray-0.1cm) -- (\angledrop:\ray+0.1cm);
						\node[] at (\angledrop:\ray-0.3cm){$u_\drop$};
					}
					{\begin{pgfonlayer}{background}
							\draw[line width=5pt, color=blue_drop] (0,0)+(\anglegauche:\ray) arc (\anglegauche:\angledroite:\ray);
						\end{pgfonlayer}
					}
				}				
				\foreach \drop/\stepa/\stepb/\angledrop/\anglegauche/\angledroite/\decalx/\decaly in {3/8/9/90/20/160/0.5/0}{
					{ 
						\draw[color=black, thick ] (\angledrop:\ray-0.1cm) -- (\angledrop:\ray+0.1cm);
						\node[] at (\angledrop:\ray-0.3cm){$u_\drop$};
					}
					{
						\node[] (drop_a) at (\angledrop:\ray+0.4cm) {\includegraphics[scale=0.03]{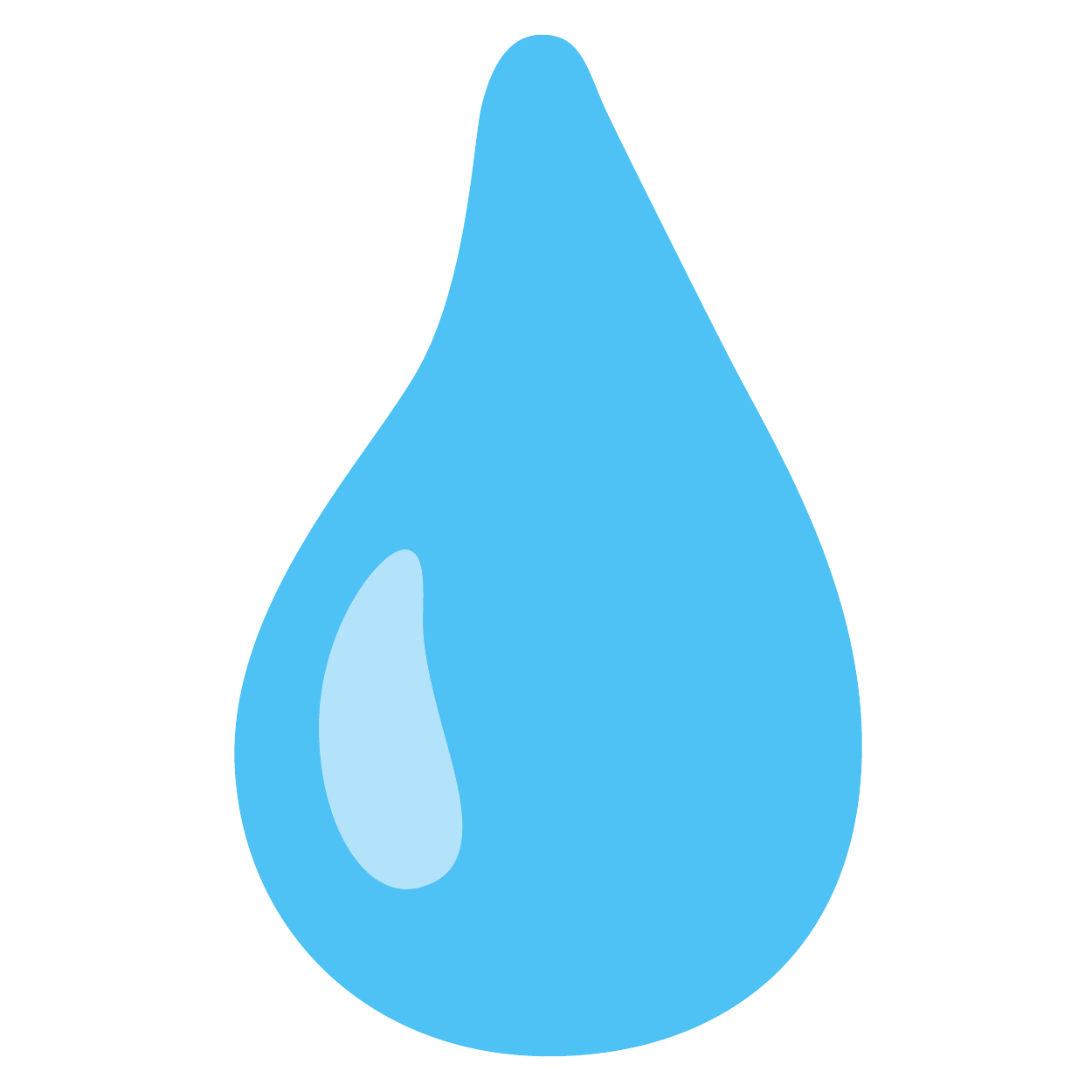}};
						\node[color = blue_drop, xshift=\decalx cm, yshift=\decaly cm] at (drop_a){$m_\drop$};
					}
				}				
			\end{tikzpicture}
			\subcaption{}\label{subfig:illu_modele_dispersion_continue1}
		\end{subfigure}
		\begin{subfigure}[b]{0.4\textwidth} \centering
			\begin{tikzpicture}[scale=1.3]
				\def\ray{1.25cm} \def\raysupvert{1.6cm} \def\raysuphoriz{1.5cm}
				\clip(-\raysuphoriz,-\raysupvert) rectangle (\raysuphoriz,1.9cm);
				\draw[black] (0,0) circle (\ray);
				
				\draw (-90:\ray-0.08cm) -- (-90:\ray+0.08cm);
				\node[below] at (-90:\ray) {{$0$}}; 
				
				\foreach \drop/\stepa/\stepb/\angledrop/\anglegauche/\angledroite/\decalx/\decaly in {0/2/3/50/30/60/0.5/0, 1/4/5/230/190/230/0/-0.8, 2/6/7/110/80/120/0.5/0, 3/8/9/90/20/160/0.5/0}{
					{ 
						\draw[color=black, thick ] (\angledrop:\ray-0.1cm) -- (\angledrop:\ray+0.1cm);
						\node[] at (\angledrop:\ray-0.3cm){$u_\drop$};
					}
					{\begin{pgfonlayer}{background}
							\draw[line width=5pt, color=blue_drop] (0,0)+(\anglegauche:\ray) arc (\anglegauche:\angledroite:\ray);
						\end{pgfonlayer}
					}
				}				
			\end{tikzpicture}
			\subcaption{}\label{subfig:illu_modele_dispersion_continue2}
		\end{subfigure}
		\caption{Illustration of a {Continuous dispersion model}. On the left, masses $m_0$, $m_1$ and $m_2$ arrived respectively at $\bu_0$, $\bu_1$ and $\bu_2$ and have been spread continuously. Mass $m_1$ has been spread on the right, while a proportion $2/3$ of mass $m_0$ has been spread on the right (and the rest was spread on the left). A drop of mass $m_3$ arrives at $\bu_3$. On the right, mass $m_3$ has been spread. The component that contains $\bu_0$, $\bu_2$ and $\bu_3$ thus has length $m_0 + m_2 + m_3$. After the $4$th mass has been spread, there are two intervals in $O^{(4)}$ and $F^{(4)}$.}
		\label{fig:illu_modele_dispersion_continue}
	\end{figure}
	
	\subsection{Content of the paper and organisation}
	
	In Section \ref{sec:VCDM} we define precisely the class of valid continuous dispersion models (\ACDM). The term ``continuous'' refers to the space which is $\R/\Z$, while the masses can be discrete or continuous, and random or not. A model will be said to be valid if it satisfies several conditions, including being invariant under rotation, and the fact that during the relaxation phase, the occupied CC which contains the point where the last mass arrival occurs,  evolves independently of the rest of the configuration, in a continuous manner (except at the collision times with other occupied components, when it merges with them).

	The main quantities of interest, which are the occupied space $\bO{k}$ after dispersion of $k$ masses, and the free space $\bF{k}$ are introduced. Each of them is a union of intervals. \par
	A list of examples of valid \ACDM is provided in Section \ref{sec:LE}. The ``right diffusion at constant speed'' is an important model that will be referred to throughout the paper. In words, when a mass $m$ arrives at $\bu$, it is pushed to the right, and is eroded while passing through a free space at constant speed (but not eroded in occupied space). Once the amount of free space $f$ on which it has been pushed reaches the Lebesgue measure $m$, no mass remains to be eroded, and $f$ becomes occupied. Other models such as the infinitesimal particle-like diffusion, the range of the Brownian path, and the short-sighted jam-spreader model, demonstrate the extent of the class of valid \ACDM.
	
	In Section \ref{sec:MUR}, we present the main universal results on \ACDM. These  principles are easier to understand when the masses $m_0,\cdots,m_{k-1}$ are not random, but not necessarily equal. These  principles will be used later for random masses.
	
	Arguably, the main theorems of the paper show that after the $k$ masses $m_0,\cdots,m_{k-1}$ have been dispersed:  
	\begin{itemize}
		\item[\bls] the joint law of  $\l(\bO{k},\bF{k}\r)$ does not depend on the mass arrival order (it has the same distribution for $m_{\sigma_0},\cdots,m_{\sigma_{k-1}}$ for any permutation $\sigma$) (Theorem \ref{theo:excha}), 
		\item[\bls]  the law of the occupied/free connected components, that is, of the pair $\l(\bO{k},\bF{k}\r)$, is universal for all valid \ACDM chosen (Theorem \ref{theo:excha}). Hence,  it suffices to analyse one of them to analyse all of them (and the right diffusion with constant speed is probably the simplest one), 
		\item[\bls] The law of the number $\bN_k$ of CC of $\bF{k}$ (or of $\bO{k}$) depends only on $k$ and on the sum of masses $\sum_{j=0}^{k-1}m_j$. Hence  it does not depend on the chosen \ACDM, nor on the individual masses $(m_j)$, nor on their order: $\bN_k-1$ follows the binomial distribution with parameters $(k-1,\sum_{j=0}^{k-1}m_j)$ (Theorem \ref{theo:excha}). Given $\bN_k=n$, the joint distribution of the free interval length is also explicit: it is a Dirichlet distribution. 
		\item[\bls] The distribution of the lengths of the CC of $\bO{k}$ still does not depend on the valid \ACDM considered, but depends on the masses, and the formula remains explicit (Theorem \ref{theo:full_dist}). 
		\item[\bls] The law of the process $(S(t),0\leq t\leq k)$ where $S(t)$ is the multiset of lengths of the elements of $\bO{t}$ at time $t$, is the same for all valid \ACDM (Theorem \ref{theo:proce}). \end{itemize}
	
	In Section \ref{sec:DC}, we present analogous results for valid \DDM (\ADDM), which are diffusion models on the discrete circle $\Z/n\Z$. In these models, the masses are integers and arrive at integer positions (and are diffused in a discrete manner). To make the discrete case fit into the continuous settings  and  allow for comparisons, we prefer to work on an isomorphic version of $\Z/n\Z$ embedded in the initial circle ${\cal C}=\R/\Z$, so we work on the set ${\cal C}_n:=\{k/n, 0\leq k\leq n-1\}$ (seen as a subgroup of $\R/\Z$), and we consider that the masses arrive uniformly on ${\cal C}_n$. The results obtained are similar to the \CDM cases (concerning the universality of the law of  $\l(\bO{k},\bF{k}\r)$ for all valid \ADDM considered, see Theorem \ref{theo:excha-D}, or the fact that we can permute the masses without changing the law of $\l(\bO{k},\bF{k}\r)$) and the distribution of $\l(\bO{k},\bF{k}\r)$ can be described (see Section \ref{sec:MUR} for more details). The distribution of the number of CC is no longer binomial, but remains explicit and quite simple (see Section \ref{sec:DNC}).
	
	In Section \ref{sec:LOT}, we present the first asymptotic results, when the number of masses goes to +$\infty$ (but the individual masses go to zero, and also in the discrete settings as $n\to+\infty$). These problems present difficulties that have only been partially overcome in this paper. A few open questions remain for the interested reader.
	
	In Section \ref{sec:park}, we investigate the difference between \ADDM and \ACDM, by defining models involving masses that are multiples of $1/n$ which can be defined in both discrete and continuous settings, and for which comparison is possible. The parking case in which all masses are multiples of $1/n$ can be very precisely studied. In this case, Chassaing \& Louchard \cite[Theorem 1.5]{chassaing2002phase} have established a phase transition when $t_n(\lambda):=n-\lambda \sqrt{n}$ cars have been parked  in the discrete parking (at this time, the largest CC contains a linear number of cars, and then, a macroscopic proportion in our settings). Our analysis shows that the total numbers of CC of $\bO{t_n(\lambda)}$, $\bN_{t_n(\lambda)}^{(n)}/\sqrt{n}$ and $\bN_{t_n(\lambda)}/\sqrt{n}$ in the discrete and continuous parking, respectively, have different limits in probability. Thus, they are very different, while large occupied CC are essentially the same in distribution (see Proposition \ref{pro:disc}).\begin{figure}[h!]
		\centerline{\includegraphics{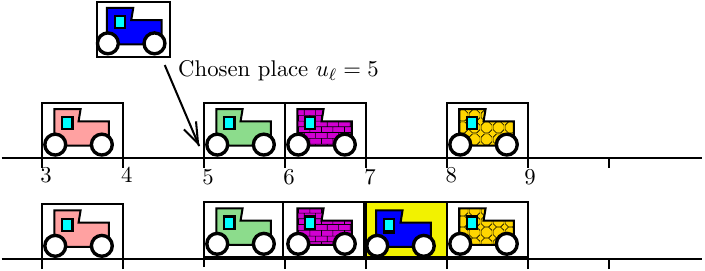}}
		\captionn{\label{fig:parking}In the standard parking problem, a car chooses a place, here $5$ and parks on the first available parking slot, which is an interval of the type $[k,k+1]$, on the right of the chosen place. Here, [7,8].}
	\end{figure}
	A generalisation of this comparison is provided in Theorem \ref{theo:genen}, which treats the case of random discrete masses. In Section \ref{sec:RWRL}, more general random masses are treated (some results from the parking problem for caravans of Bertoin and Miermont \cite{bm2006} are recalled).
	
	Finally, in Section \ref{sec:ED}, a rather general question is investigated for which only preliminary answers are obtained. One of the main questions of interest in some dispersion models may depend on its entire history. For example, the cost of constructing a table in the data structure called ``hashing with linear probing'' depends on its entire history (Knuth \cite[Sec. 6.4]{Knuth3}). This cost is defined as to the total displacement of the cars in the standard parking problem, which  corresponds to the sum of $\floor{ U_i b_i}$, where the $(U_i)$ are independent and identically distributed (i.i.d.) uniform on $[0,1]$ and $b_i$ is the length of the occupied component that received the $i$th car.
	In general, we may imagine that the cost corresponds to some energy dissipation of the model (or ``work'', in physical terminology, to disperse matter). This cost appears to be a kind of integral of a functional of the sequence, indexed by $k$,  of the length of the CC that received the $k$th mass. \par
	In the case of the parking construction, we provide limit results describing ``all large block lengths'' appearing during construction (Theorem \ref{theo:vague}). In Theorem \ref{theo:dqgreht}, we prove that if the unitary cost of a single mass insertion into a block of length $k$ satisfies certain assumptions, then, the limiting total cost can be expressed as a functional of the Brownian excursions.
	The general question is discussed in Section \ref{sec:qegrhtyu}, but only preliminary results are stated.
	
	{\modif	Finally, Section \ref{sect:table_of_notations} provides a summary table gathering the major notation occurring throughout the paper.}

	\subsection{The nature of general dispersion models, and restriction to ``valid ones''}
	\label{sec:srgethtu}
	It is difficult to define a general framework for dispersion models because they take very different forms in nature and in combinatorial or technical applications. The matter arriving on the medium can be discrete or continuous. The medium itself can be homogeneous or heterogeneous. The matter can be discrete or continuous and can flow continuously on the medium. Dispersion can be continuous over time. After deposition and stabilisation, the layer of deposited matter can be non-constant and non-homogeneous, as with sand dunes.
	
	To account for the wide range of  mass arrival models, we could use a process $(A_t,t\geq 0)$, where for each $t$, $A_t$ is a random Borel measure on ${\cal C}$, and for a Borel set $b$, $A_t(b)$ would be the total mass that has arrived in $b$ before time $t$.  
	
	The state of the system at time $t$, could then be described by a second measure $L_t$ which would give the level of matter at time $t$. Again, for a Borel set $b$, $L_t(b)$ would be the quantity of matters, deposited on $b$, at time $t$. However, this would not be sufficient in many cases, since the simplest mechanical model would require additional data, such as local speed and forces.
	
	In some cases, the physical or technical phenomenon governing the relaxation phase can be well described by partial differential equations acting on a function $f$, where $f(t,x)$ represents the amount of matter at $x$ at time $t$. However, in other cases, a stochastic process is needed to account for the stochastic nature of the studied problem, and perhaps some control is necessary in other cases.
	The variety of  models is vast. In many cases, an important enrichment of the probability space would be needed in order to accommodate the random physical phenomenon under consideration.\medskip 
	
	However, \textbf{the main conceptual result} of this paper is as follows. There is a huge class of dispersion models, that we will call
	\textbf{Valid continuous dispersion models}, 
	that have the following characteristics:
	\bia
	\ita they are ``discrete time processes'' -- the time arrivals of masses occur each slot of time, 
	\ita they produce a fixed ``level of deposition'', in words, an abscissa $x$ in the medium is either free or occupied/covered (and once occupied it stays so forever), 
	\ita the masses are treated sequentially, in the sense that  stabilisation/relaxation of mass $m_k$ is done before the arrival of mass $m_{k+1}$, 
	\ita during the dispersion of the $k+1$th mass $m_k$ that arrives at $\bu_{k}$, the process $(\IR_k(t),k\leq t \leq k+1)$ that describes the connected occupied  component containing $\bu_k$ at time $t$ grows continuously (except when it merges with an already existing occupied CC of $O^{(k)}$ it hits) and the evolution of $\IR_k$ on $[t,t+\d t]$ is independent of the complement of $\IR_k(t)$. In other words, there exists a ``continuous time'' Interval Relaxation process $\IR_k$ for $t\in[k,k+1]$ which describes the evolution of this CC, such that $\Leb\l(\IR_k(t)\backslash O^{(k)}\r)$ is continuous, and starts at zero at time $t=k$, grows at speed 1 until time $k+m_k$, and then stays unchanged until time $k+1$.
	\ita the model is invariant under rotation (in particular, the arrival positions $\bu_k$ are uniform),
	\eia~\\
	and all these models, after stabilisation of $(m_0,\cdots,m_{k-1})$, induce the same (tractable) distribution on the finally occupied domain $\bO{k}$.\medskip
	
	These properties, that are needed for the universal results described above, \textbf{are not properties of the complete dispersion models (which as already said, may demand a much richer probability space), but only induced properties on occupied CC processes}.

	Moreover in the analysis $(a)$, $(b)$, $(d)$ and $(e)$ are crucial, while $(c)$ is not, but simplifies greatly the exposition, and allows us to define analysable ``cost functions'' associated with the dispersion process.

	In Section \ref{sec:VCDM}, we provide   formal definitions. In Section \ref{sec:LE}, we present a list of examples (that can be understood without reading Section \ref{sec:VCDM}).

	\subsection{Valid \CDM{s} (valid \ACDM) : formal addendum}
	\label{sec:VCDM}
	Here, we provide more precise definitions of the  conditions $(a)-(e)$ that were discussed  informally in the previous subsection.\\

	\textbf{Notation.} For any $a,b$ in $\R/\Z$ denote by $\bar{a}$ and $\bar{b}$ the representatives of $a$ and $b$ in $[0,1)$, respectively. Also, denote the distance from $a$ to $b$,  in the canonical positive direction by $\flecheu{d}(a,b)=\inf\{x\geq 0:\bar{a}+x-\bar{b}=0 \mod 1\}$.
	
	We denote by $\fleche{[a,b]}$ the following circular interval: for $a,b$ in $\R/\Z$, $\fleche{[a,b]} = \l\{ \bar{a+x}, x \in \l[0, \flecheu{d}(a,b)\r]\r\}$.
	
	\medskip

	Following the content of Section \ref{sec:srgethtu}, we assume that ``the physical dispersion phenomenon'' takes place on a probability space $\l(\Omega,{\cal A},\mathbb{P}\r)$,

	\noindent\bls \textbf{Continuous space.}\\ The medium on which the dispersion takes place is the (continuous) unit circle ${\cal C}:=\R/\Z$.  
	
	\noindent\bls \textbf{Model of masses.}\\ The sequence of masses is $(m_i,0\leq i\leq n)$ for a finite $n$  (in the paper). The masses are random or not, each mass is non-negative, and the total mass $\sum_{i=0}^n m_i\leq 1$; observe that the mass $m=0$ is permitted.
	\medskip
	
	\noindent\bls \textbf{Uniform arrival positions, and independence between masses and positions.}\\ The mass $m_i$ arrives at position $\bu_i$: the $(u_i,i\geq 0)$ are i.i.d., independent of the masses $(m_i,i\geq 0)$, and taken according to ${\sf Uniform}({\cal C})$. 
	\medskip
	
	\noindent\bls\textbf{Valid Continuous Dispersion Models.} \\
	A valid dispersion model successively treats  a sequence of mass arrival events $(m_i,u_i)$ where $i$ goes from 0 to $n-1$, and $\bu_i$ is the place at which the mass $m_i$ arrives. The set $\bO{j}$ denotes the occupied set after the first $j$ masses $m_0,\cdots,m_{j-1}$ have been dispersed. The free space is denoted  $\bF{j}={\cal C}\setminus \bO{j}$.

	At time $0$, the occupied set $\bO{0}$ is the empty set. We will add hypotheses to the diffusion process so that -- among other things -- the set $\bO{k}$ is a finite set of  closed intervals (its CC), and the sequence $(\bO{k})$ satisfies
	\[\bO{k}\supseteq \bO{k-1},~~\textrm{ for all }k\geq 1.\]
	The event $\bO{k}= \bO{k-1}$ is possible only when a zero mass $m_{k-1}=0$ arrives in $\bO{k-1}$, in an already occupied CC. 
	Moreover, the Lebesgue measure of occupied space at time $k$ corresponds to the total mass that has arrived before time $k$:
	\beq\label{eq:mas}\Leb\l(\bO{k}\r)=\sum_{i=0}^{k-1} m_i.\eq
	{\modifs Notice also that when a zero mass arrives at a point $\bu_k$ which is free, then $\bO{k}=\bO{k-1}\cup\{\bu_k\}$: the "single point" $\bu_k$ does not modify the Lebesgue measure of the occupied set but adds a new connected component, and splits a free interval. 
	}
	
	As explained in the previous section, we will not define a ``general dispersion model'', but rather, we will fix some of the properties that it must satisfy, regarding the expansion of the occupied CC that receives the new mass $(m_k,u_k)$, or, in other words, regarding the interval relaxation processes $(\IR_k)$ that it induces.
	
	Assume that $\bO{k}$ has been constructed (using the $k$ elements $(u_0,m_0),\cdots,$ $(u_{k-1},m_{k-1})$) and has the aforementioned properties. The hypotheses we need to  specify concern the process $(\IR_k(t),k\leq t \leq k+1)$, which describes the evolution of the CC receiving the new mass $m_k$ at $\bu_k$, from the time $k$ it receives it until time $k+1$.\par
	
	A \underbar{valid} interval relaxation process 
	$\IR_k:=(\IR_k(t),k\leq t \leq k+1)$ satisfies the following properties: 
	\bir
	\itr It has to be well defined a.s. for all set $O^{(k)}$ of disjoint  closed intervals in the support of $\bO{k}$, and for all mass arrival events $(u_k,m_k)$ in the support of the considered mass arrival model.
	\itr  Given $(\bO{k},u_k,m_k)$, the process $\IR_k$ can be deterministic, or random. For each $t\in[k,k+1]$, $\IR_k(t)=\fleche{[a(t),b(t)]}$ is an interval of ${\cal C}$, and we demand that $t\mapsto \IR_k$ is non decreasing, and that this process takes its values a.s. in $D([k,k+1],{\cal C}^2)$ (the space of c\`adl\`ag processes indexed by $[k,k+1]$ with values in ${\cal C}^2$, equipped with the Skorokhod topology), where we identified the set of directed intervals on ${\cal C}$ with ${\cal C}^2$.
	\itr  Define  $\IR_k(k)$ as the CC of $O^{(k)}$ containing $\bu_{k}$, and if there are no such component, set $\IR_k(k)=\{\bu_k\}$.  We require that during the relaxation time $[k,k+m_k]$, $\Leb(\IR_k(t))$ grows outside $\bO{k}$, at unit speed, that is
	\beq\label{eq:LebEv}\Leb\l(\IR_k(k+t)\backslash \bO{k}\r)=t, \textrm{for }t\in [0,m_k]\eq
	and can jump (that is $\IR_k(k+t)\neq \IR_k(k+t-)$) only when the interval $\IR_k(k+t)$ coalesces with one of the intervals present in $\bO{k}$ (contrary to appearances, this condition is not really a condition on the speed of dispersion! See Remark \ref{rem:etjy}).  From time $k+m_k$ to time $k+1$, $\IR_k(t)$ remains constant, and equal to $\IR_k(k+m_k)$. 
	\itr For all $s\leq k+ m_k$, $(\IR_k(t),t\in[k,s])$ is independent from $\bO{k}\backslash \IR_k(s)$: in words, $\bO{k}\backslash \IR_k(s)$ gives the occupied domain at time $s$ outside $\IR_k(s)$, and as such, determines the complete configuration outside the relaxation interval. The process  $(\IR_k(t),t\in[k,s])$ encodes the history of the relaxation interval of the $k+1$th mass in the time interval $[k,s]$. This condition imposes that the interval relaxation is independent from what lies outside this interval.
	%is the history of the interval of $r
	%the evolution  of the CC of $\IR_k$ containing $u_{k}$ between time $t$ and $t+\d t$ must be independent of the rest of the state of the medium, that is, from $\bO{k}\setminus \IR_k(t)$.
	This property can be view as a kind of spatial Markovian property.
	\itr The model is invariant under rotation, meaning that the (distribution of the) diffusion processes $\IR_k(t)$ defined from time $k$ to time $k+m_k$ does not depend on its position on the circle. First, $\bu_k$ is uniform (and independent from the other random variables). Denotes by $r_{\alpha}$ the rotation by $\alpha\in \R/\Z$ in $\R/\Z$. 
	For all $k$, the diffusion process $\IR_k(t)$ defined on the time interval $[k,k+m_k]$ has a distribution which depends uniquely on $(\bO{k},u_k,m_k)$. For all $\alpha\in\R/\Z$,  
	\beq\label{eq:dqgeg}\Law\l( \l(\IR_k\r)~|~(\bO{k},u_k,m_k)\r)=\Law\l(r_{-\alpha}\l(\IR_k\r)~|~(r_{\alpha}(\bO{k}),r_\alpha(u_k),m_k)\r). \eq
	\eir~\\
	
	We define $\bO{k+1}$ as the closed set $\IR_k(k+1) \cup \bO{k}$ (by construction $\IR_k(k+1)$ may contain zero, one or several CC of $\bO{k}$).
	
	\begin{defi}
		A continuous dispersion model  (\ACDM)  is said to be valid, if all the interval relaxation  processes $(\IR_k)$ are valid.
	\end{defi}
	
	\begin{rem}[About the relaxation time of the $k$th mass]\label{rem:etjy}
		A good way to understand the extent of the class of valid continuous processes is to consider that the relaxation phase of the $k$th mass occurs within a time frame which is independent of the discrete times enumerating the mass arrival events $(m_k,u_k)$.\par
		When one wants to encode an actual physical system (e.g.\ the diffusion of a droplet in a random medium), it is often natural to encode the relaxation time of the $k$th mass by a \underbar{continuous time interval}:\\
		(a)  on $[k,k+1)$ (so that, everything is done for the arrival of the next $k+1$th mass ,\\
		(b) or on $[0,+\infty)$ to let as many time as needed for the dispersion time,\\
		(c) or on $[k,k+m_{k}]$, and then assume that the mass deposition is done at unit speed.\par
		
		Since we observe the states after the relaxation times (to define $O^{(k)},m_k,u_k...$), and since one can pass from any of the models $(a),\cdots,(c)$ to the others by a change of time, these points of view are equivalent for our purposes. We then choose the model $(c)$ to make easier the description of everything on the same clock.\par
		When needed, we will refer to this property as ``the constant-speed deposition property''.
	\end{rem}
	
	\begin{rem} [About the invariance under rotation] Without the invariance under rotation hypothesis $(v)$, the diffusion process could depend on the arrival points $\bu_i$. We could design some diffusion processes that would stop their diffusion at some special places (e.g. so that to produce occupied sets $\bO{k}$ such that each CC has a (random) extremity with a rational coordinate on ${\cal C}$). All universal results presented in Section \ref{sec:MUR} would not hold without condition $(v)$.
	\end{rem}
	
	Figure \ref{fig:illu_modele_dispersion_continue} illustrates the definition of the continuous dispersion model.

	\subsection{List of examples of valid continuous dispersion processes}
	\label{sec:LE}
	
	In this section we provide a list of 8 valid continuous dispersion models.
	
	\subsubsection{Right Diffusion at Constant Speed (RDCS): an important cornerstone.}
	\label{sec:RDCS}
	
	The Right Diffusion at Constant Speed is an important diffusion process that will serve as a cornerstone for the rest of the paper. It is defined for any masses $(m_0,\cdots,m_{k})$ as long as their sum is in $[0,1)$. We exclude the trivial case $\sum m_i=1$, in which all the space is occupied, to avoid treating this special case which would add unnecessary complications.\par
	Assume that there is a mass arrival $(m,\bu)$ at time $k$, on a configuration characterised  by $(\bO{k},\bF{k})$. Initially,	$\IR_k(k)=\fleche{[a,b]}$ is the occupied CC of $\bO{k}$ receiving $\bu$ if it exists (it may be created if $\bu$ was not in $\bO{k}$, in which case $\IR_k(k)=[a,b]:=[\bu,\bu]$). Now, the mass $(m,\bu)$ is pushed to the right on $\R/\Z$ from $\bu$ and covers the space from there. At time $t\in[k,k+m_k]$, the CC is the smallest oriented  interval $\IR(t) = \fleche{[a,b(t)]}$ on $\R/\Z$ (that grows on its right extremity) and satisfies \eref{eq:LebEv} (that is $\Leb\l(\IR_k(k+t)\backslash \bO{k}\r)=t, \textrm{for }t\in [0,m_k]$).  See Figure \ref{fig:RDCS2} for an illustration. 
	\begin{figure}[h!]
		\centerline{\includegraphics[scale=0.9]{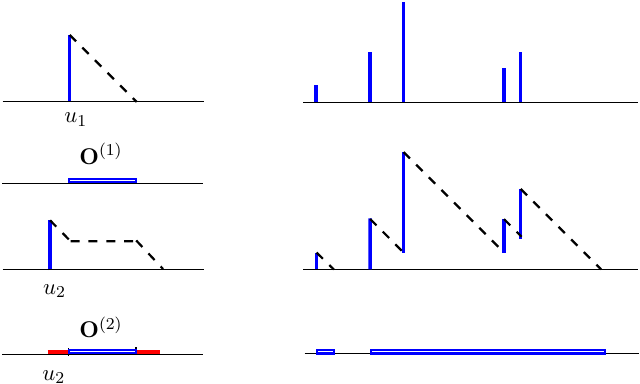}}
		\captionn{\label{fig:RDCS2} The cycle $\R/\Z$ is represented as the segment $[0,1)$. In the RDCS, the masses are pushed to the right. \underbar{In the first column}, two mass arrival events and their relaxations. A new mass is pushed to the right and covers the medium at the same speed as  if it was eroded. The third picture shows that the erosion occurs only in free spaces, while the mass is pushed with no loss in occupied domains.\\ \underbar{The second column illustrates a phenomenon that will be proven later:} the final  occupied domain does not depend on the order of the sequence $(u_i,m_i)$. This is clear in the RDCS model, because, the quantity of ``non-eroded mass'' that passes above a given abscissa $x$ does not depend on ``the identity of the mass'' which is eroded at a given location.}
	\end{figure}
	
	\begin{rem}\label{rem:caravan}
		The  representation in Figure \ref{fig:RDCS2} appears in many papers, typically to encode paths with a deterministic slope -1. The randomness appears as positive jumps (notably in queuing theory with one server). \LL paths  are used to encode the remaining charge as a function of time. Customers arrive at some random times, and demand a random quantity of service. The server  delivers service at a constant rate, so that the total amount of service remaining to be done at time $t$, seen as a process of $t$, can be represented by a path with a slope $-1$ and vertical jumps at the arrival times of customers. The value of each jump corresponds to the customer's demand.
		In Bertoin \& Miermont \cite{bm2006} (see Section \ref{sec:RWRL}), the RDCS is called the ``model of caravan'': ``masses are drops of paint brushed to the right''. There is however a subtle difference: they encode the covered space $\bO{k}$ as a union as open sets, while we take them as closed. It allows us to consider covered set reduced to a point, which will be shown to significantly simplify our analysis.
	\end{rem}

	Of course the Left Diffusion at Constant Speed (LDCS) can be defined analogously. It will be only used in the next three examples.
	
	\subsubsection{$(p,1-p)$ proportion of the mass is diffused to the right/to the left}
	\label{sec:p1mp}
	When a mass arrival event $(m,u)$ occurs, first, proceed to the RDCS of $(pm,u)$ (a proportion $p$ of the mass is diffused to the right), and then proceed to the LDCS of $((1-p)m,u)$. If we proceed to the RDCS and LDCS one after the other, the total time of diffusion is $m$. If we perform both tasks simultaneously, the final result is the same but occurs earlier, at time $\max\{p,1-p\}m$. This is another instance in which we can observe that many details of the models are irrelevant when we are only interested in the final configuration.
	
	\subsubsection{With probability $p$ perform a right diffusion, with probability $1-p$ a left one}
	
	When a mass arrival event $(m,u)$ occurs, toss a Bernoulli$(p)$ coin. If the result is 1, then proceed to the RDCS of $(m,u)$; otherwise proceed to the LDCS of $(m,u)$.
	
	\subsubsection{Diffusion to the closest side}
	
	When a mass arrival event $(m,u)$ occurs: If $u$ does not belong to the current occupied domain $\bO{k}$ then do a $(1/2,1/2)$ proportion to the right/left diffusion. Otherwise, if the arrival position $u$ belongs to $[a,b]$, a CC of $\bO{k}$, then compare the distance $d_1:=\flecheu{d}(a, u)$ with $d_2:=\flecheu{d}(u,b)$. If $d_1<d_2$, then do an LDCS of $(m,u)$, otherwise, do an RDCS.
	
	\subsubsection{Diffusion to the closest side with constant reevaluation}
	
	Do the same thing as in ``diffusion to the closest side'', but {after each infinitesimal time $\d t$}, recompute the distance to the sides of the relaxation interval. When the sides are at an equal distance, grow both sides at the same speed.
	
	\subsubsection{Infinitesimal particle-like diffusion}
	\label{sec:IPLD} 
	This model describes the asymptotic diffusion behaviour of masses $m_i$ which are composed of ``infinitesimal'' particles of mass $1/M$ (and the limit   regime is taken for $M\to+\infty$), where these small particles perform independent random walks successively until the moment they exit the occupied interval. In doing so, they increase the length of the occupied domain by $1/M$  for subsequent infinitesimal particles composing $m_i$.
	
	In order to describe these dynamics let us assume that a mass $m$ arrives at 0, in a block $\fleche{[-a,b]}$ (with $a,b>0$). As usual, if the arrival location 0 is in the free space, then $\fleche{[a(0),b(0)]}=\{0\}$. Since the dynamics are invariant under rotation, it suffices to provide a full description of this case (as we will see, we specify the speeds at which $a$ and $b$ move, so that what we will say remains valid when CC merge: it suffices to modify $a$ and $b$ accordingly). 
	For $B$, a Brownian motion starting at 0, 
	\[\P\big(\inf\big\{t:B(t)=-a\big\}<\inf\big\{t:B(t)=b\big\}\big)=b/(a+b).\] 
	At the limit over $M$, the speed at which a boundary moves satisfies a simple differential equation which is the immediate limit of the preceding considerations; to understand how the interval changes over time, let us work on $\R$ for a moment (the transfer to $\R/\Z$ is just a formality).\par
	If $(m, 0)$ is a mass arrival event, and $\IR_k(k+t)=[-a(t),b(t)]$ is the occupied CC (with $-a[t]<0<b[t]$), then the relative speed of the sides should be proportional to $b(t)$ and $a(t)$, respectively\footnote{ that is for a function $C(t)$ we should have
		$ b'(t)= C(t). a(t)$ and $a'(t)= C(t).b(t)$,
		and since we opted for the unit-speed deposition, we choose $C(t)$ so that $a'(t)+b'(t)=1$. We then have
		\[ b'(t)= \frac{a(t)}{a(t)+b(t) }, ~~  a'(t)= \frac{ b(t)}{a(t)+b(t)},\]}
	and then, we get
	\beq\label{eq:a-b} b(t)=   \frac{t+a(0)+b(0)}{2}+ \frac{b(0)^2-a(0)^2}{2(a(0)+b(0)+t)},~~ a(t)= \frac{t+a(0)+b(0)}{2}+ \frac{a(0)^2-b(0)^2}{2(a(0)+b(0)+t)}.\eq

	Of course, upon collision with another occupied block, $\IR_k(t)$ jumps, as usual.
	
	We justify the approximation  of the interval evolution by this fluid limit in Section \ref{sec:annex_FA}.

	\subsubsection{Range of the Brownian path}
	Let us provide an example given by a random process (see Figure \ref{fig:RBP}).
	
	Let $R(u,t):=\{B_s^{(u)} \mod 1, 0\leq s\leq t\}$ denote the range of a Brownian motion $B^{(u)}$ on $\R/\Z$, starting at $u$. Of course $t\mapsto R(u,t)$ is non-decreasing with respect to the inclusion partial order.
	
	To define $\bO{k+1}$ in terms of $\bO{k}$, start a Brownian motion at $\bu_k$, and stop it at the (first random) time $\tau$ for which the newly visited space has the right Lebesgue measure, that is, satisfies $\Leb(R(\bu_k,\tau)\setminus \bO{k})=m_k$. Then, set $\bO{k+1}=\bO{k}\cup  R(\bu_k,\tau)$ as usual. 
	
	\begin{figure}[htbp]
		\centerline{\includegraphics[scale = 0.9]{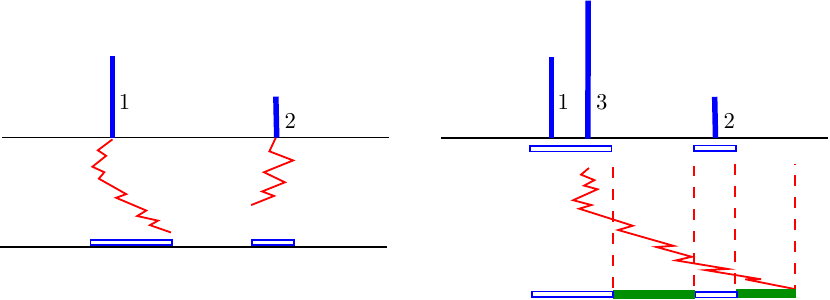}}
		\captionn{\label{fig:RBP} The range of a Brownian motion model: the first Brownian motion $B^{(\bu_0)}$ starts at $\bu_0$ and is killed at the first $t$ for which $\Leb\l(\l\{B_s^{(\bu_0)},s\leq t\r\}\r)=m_0$. Its track is $\bO{1}=\l\{B_s^{(\bu_0)},s\leq t\r\}$. The Brownian motion $B^{(\bu_{i})}$ starts at $\bu_{i}$ and is killed when  $\Leb\l(\l\{B_s^{(\bu_i)},s\leq t\r\} \setminus \bO{i-1}\r)=m_i$, that is when the set of points visited by $B^{(\bu_{i})}$, but that have never been visited by the previous Brownian motions, has Lebesgue measure $m_i$.}
	\end{figure} 
	In this special case, the $k+1$th relaxation interval process has also a simple representation, since one can understand the evolution of the connected component that received the new point: for any $x\in[0,m_k]$, consider $\tau(x)=\inf\{s\geq 0~:\Leb(R(\bu_k,s)\setminus \bO{k})\geq x\}$.
	For all $0\leq x\leq m_k$, define $\IR(k+x)$ as the CC of $\bO{k}\cup R(\bu_k,\tau(x))$ containing $\bu_k$ (and let $\IR_k$ be constant, as usual, between time $k+m_k$ and time $k+1$).

	\subsubsection{The short-sighted jam-spreader model}
	
	A short-sighted individual wants to spread some drops of jam that falls on a donut (identified with ${\cal C}$) according to a classical mass arrival events $((m_k,u_k),k\geq 0)$.
	During the spreading of  drop $(m_k,u_k)$, the spreader cannot see where the jam is lacking on ${\cal C}$. He therefore acts randomly, progressively transforming the pile of non-deposited jam, that we could represent by a Borel measure  $(M_{k+t},0\leq t\leq 1)$ on $\IR_k(k+t)$ whose total mass at time $k+t$ is the mass $m_k-t$ of undeposited jam.
	We may assume that the spreader proceeds randomly, and that his actions continuously expand $\IR_k(k+t)$, until all the jam that has arrived is finally spread, and a new portion of ${\cal C}$ with Lebesgue measure equal to $m_k$ is covered.  \par
	We assume that any individual's ``strategy'' is invariant under rotation, and that the $k$th policy he adopts for the  drop $(u_k,m_k)$ is independent of the current sets $\bO{k}$ and $\bF{k}$, and  depends only on the occupied block in which this drop falls (and possibly on the current measure $M_{k+t}$ {(encoding the deposition process of $m_k$)}  he is spreading). 
	%The policies used for different $i$ may differ.
	{\modifs The policy may depend on the arrival index $i$.}

	\subsection{Main universality results for valid \ACDM: deterministic masses case}
	\label{sec:MUR}
	This section is devoted to the study of the statistical properties of a \underbar{valid} \CDM, when the sequence of masses $m[k]:=(m_i,i=0,\cdots,k-1)$ that arrive in the system   is formed by   \underbar{deterministic}, non-negative masses, and satisfies only the total weight condition 
	\beq W(m[k]):= \sum_{i=0}^{k-1} m_i <1. \label{eqn:defWmk}\eq
	We will turn to random masses in Section \ref{sec:RMR}.
	
	We insist on the fact that we allow the $m_i$ to be 0 too. Again, a zero mass arriving in an occupied CC has no effect. However, when it arrives in a free CC, it creates a new occupied CC reduced to the point $\{u\}$ at which the arrival took place.
	
	The first aim of the paper is to describe the distribution of the occupied and free spaces $\bO{k}$ and $\bF{k}$ at time $k$ when $k$ masses $((m_i,\bu_i),0\leq i \leq k-1)$ have been dispersed according to a valid \CDM.

	We let
	\beq \bN_k=\#\bF{k} =\#\bO{k}+\1_{k=0}, \label{eqn:def_Nk}
	\eq
	be the number of free CC when $k$ masses have been dispersed in the system.
	For $k\geq 1$, label the CC  ${\bf O}_0^{(k)},\cdots,{\bf O}_{\bN_k-1}^{(k)}$  of $\bO{k}$, and $\bF{k}_0,\cdots,\bF{k}_{\bN_k-1}$ those of $\bF{k}$, such that	by turning around the circle $\R/\Z$, one finds successively (they are adjacent in this order), 
	\beq\bL{k}:={\bf O}_0^{(k)},{\bf F}_0^{(k)},\cdots, {\bf O}_{\bN_k-1}^{(k)},{\bf F}_{\bN_k-1}^{(k)},\label{eqn:def_ordre_OetF}\eq 
	and moreover, the point  
	\beq\label{eq:qfgezfqd} 0 \in {\bf O}_0^{(k)}\cup{\bf F}_0^{(k)}.\eq
	The point 0 of the circle is then used to determine which CC of $\bO{k}$ and $\bF{k}$ will be named $\bO{k}_0$ and $\bF{k}_0$, respectively: this introduces a bias which complicates the exposition slightly.
	The sequence of block lengths
	\beq|\bL{k}|:=\l(|\bO{k}_0|,|\bF{k}_0|,\cdots, |\bO{k}_{\bN_k-1}|,|{\bf F}_{\bN_k-1}^{(k)}|\r)\eq allows to characterise the occupied and free sets up to a rotation, and the shift information, given by the number $s_0 \in \l[0,|\bO{k}_0\cup \bF{k}_0|\r]$, allows to place $0$ in the pair of consecutive intervals with lengths $|\bO{k}_0|$ and $|\bF{k}_0|$ ($0$ is at distance  $s_0$ from the left side of this pair, for a coherent notion of ``left''). \par
	Hence, each of $\bL{k}$ and $(|\bL{k}|,s_0)$ allows to characterise the other. %:to ,
	% \[s_0 = - \min {\bf O}_0^{(k)}, \textrm{ and then } s_0  \in \l[0,|\bO{k}_0\cup \bF{k}_0|\r]\] allows to reconstitute $\bL{k}$ (each of them characterizes the other one).

	\subsubsection{Universal properties}
	
	The next theorem is one of the paper's main results and gathers the properties that are universal for all valid \ACDM. In short, for a fixed $k$, the distribution of $\l(\bO{k},\bF{k}\r)$ is an invariant of all valid \ACDM.
	\begin{theo}\label{theo:excha}
		Consider a valid \ACDM $A$, and some deterministic masses $m[k]=(m_0,\cdots,m_{k-1})$, which are either positive or zero and satisfy the total weight condition  $W(m[k])  <1$. Let $(\IR_j^A,0\leq j \leq k-1)$ denote the interval relaxation processes associated with $A$. We have  
		\begin{enumerate}[label=(\roman*)]
			\item\label{item:i_in_theo_excha} For a fixed $k$, the law of $\l(\bO{k},\bF{k}\r)$ does not depend on $A$ (it is the same for all valid \ACDM that are designed to treat the masses $(m_0,\cdots,m_{k-1})$.
			\item The law of $\l(\bO{k},\bF{k}\r)$ is invariant under the permutation of the masses:
			For any permutation $r\in {\bf S}_k$ (the symmetric group on $\{0,\cdots,k-1\}$),
			\[\Law\l (\l(\bO{k},\bF{k}\r)~|~ (m_0,\cdots,m_{k-1})\r)=\Law\l (\l(\bO{k},\bF{k}\r)~|~ (m_{r(0)},\cdots,m_{r(k-1)})\r)\]
			(where we have written, as a condition, the successive masses that are used).
			\item For $k\geq 1$, the number of free CC,  $\bN_k$, has (almost) a binomial distribution:
			\beq
			\label{eq:binom} \bN_k \eqd 1+ B(m[k]), \textrm {~~where~~} B(m[k])\sim {\sf Binomial}(k-1,R_k) 
			\eq
			where $R_k=1-W(m[k])$ is the final free space available after the deposition of the first $k$ masses.
			\item\label{item:iv_in_theo_excha} Take a uniform element $\theta$ in $\{0,\cdots,b-1\}$ (independently of everything). Conditional on $\bN_k=b$, 
			\beq\l(\l|\bF{k}_{i+\theta \mod b}\r|,0\leq i \leq b-1\r)\eqd R_k\, (\bD_0,\cdots,\bD_{b-1})\eq  where $(\bD_0,\cdots,\bD_{b-1})\sim{\sf Dirichlet}(b;1,\cdots,1)$. 
			Hence, the sequence $\l(|\bF{k}_{i}|,0\leq i \leq b-1\r)$ taken according to a uniform independent rotation, is unbiased. 
			\item\label{theo:free_space_nb_blocks_as_a_process_CV} 
			The process $(\bN_k,k\geq 0)$ is a Markov chain \modif{ whose distribution does not depend on $A$}. Its transition matrix can be computed explicitly (see Section \ref{sec:qefgrethry}). 
		\end{enumerate}
	\end{theo}

	Removing the rotation in \ref{item:iv_in_theo_excha} comes at a cost, due to the bias concerning the block containing 0 (see Remark \ref{rem:bias}). The free set has an extra universal property that we prefer to state separately for a technical reason.%, appearing already in Theorem \ref{theo:excha} $(i)$. 
	
	\modd{The proof of this Theorem relies on several principles, one of which stated in Lemma  \ref{lem:infinitesimal}, providing an "infinitesimal" approach: even during the relaxation time, during the evolution of the occupied/free spaces, some symmetries are preserved, and the free spaces are Dirichlet distributed; in words the present theorem can be used during the relaxation time! Of course, the situation after the relaxation of a proportion $x$ of a mass $m$ corresponds to the complete relaxation of a mass $xm$, so that this property may seem natural. But beyond that, there is some invariants preserved even when the dynamics are perturbed, and the mass to be deposited is moved from the evolving connected component to another one. This very strong invariance results will yield to the following theorem which ``trivialises'' the free space distribution.}

	Let $\overrightarrow{\bF{k}}$ denote the free spaces random variable under the right diffusion with constant speed model (RDCS).
	\begin{theo}\label{theo:length2}Assuming the hypothesis of Theorem \ref{theo:excha}, for a fixed $k\geq 1$,
		\[\Law\l (|\bF{k}|~|~ (m_0,\cdots,m_{k-1})\r)=\Law\l (\l|\overrightarrow{\bF{k}}\r|~|~ (W(m[k]),\underbrace{0,\cdots,0}_{k-1\textrm{~zeroes}})\r);\]
		that is, for all valid \CDM{s}, the lengths of the free spaces are distributed as if the masses $(m_0,\cdots,m_{k-1})$ were replaced by a ``global mass'' $W(m[k])$, and $k-1$ zero masses, and dispersed by the right diffusion with constant speed model. 
	\end{theo}
	Hence, even if a \ACDM is not made to deal with all masses, the distribution of the free spaces is the same as those for the RDCS, for which the masses $m_i$ can be permuted, or even rearranged the way we want as long as we keep their number and sum. {\modif We illustrate this result on Figure \ref{fig:illutheo:length2}.} This theorem is a Corollary of Theorem \ref{theo:excha}, in which by $(iii)-(v)$, we see that the total free space distribution depends   only on $(k,\sum_{i=0}^{k-1}m_i)$. 
	
	\begin{figure}[h!] \centering
		\begin{subfigure}[b]{0.4\textwidth} \centering
			\begin{tikzpicture}[scale=1.3]
				\def\ray{1.25cm} \def\raysupvert{1.6cm} \def\raysuphoriz{1.5cm}
				\clip(-\raysuphoriz,-\raysupvert) rectangle (\raysuphoriz,1.9cm);
				\draw[black] (0,0) circle (\ray);
				
				\draw (-90:\ray-0.08cm) -- (-90:\ray+0.08cm);
				\node[below] at (-90:\ray) {{$0$}}; 
				
				\foreach \drop/\stepa/\stepb/\angledrop/\anglegauche/\angledroite/\decalx/\decaly in {0/2/3/50/30/60/0.5/0,1/4/5/230/190/230/0/-0.8, 2/6/7/110/80/120/0.5/0}{
					
					{ 
						\draw[color=black, thick ] (\angledrop:\ray-0.1cm) -- (\angledrop:\ray+0.1cm);
						\node[] at (\angledrop:\ray-0.3cm){$u_\drop$};
					}
					{\begin{pgfonlayer}{background}
							\draw[line width=5pt, color=blue_drop] (0,0)+(\anglegauche:\ray) arc (\anglegauche:\angledroite:\ray);
						\end{pgfonlayer}
					}
				}								
			\end{tikzpicture}
			\subcaption{}\label{subfig:illutheo:length21}
		\end{subfigure}
		\begin{subfigure}[b]{0.4\textwidth} \centering
			\begin{tikzpicture}[scale=1.3]
				\def\ray{1.25cm} \def\raysupvert{1.6cm} \def\raysuphoriz{1.5cm}
				\clip(-\raysuphoriz,-\raysupvert) rectangle (\raysuphoriz,1.9cm);
				\draw[black] (0,0) circle (\ray);
				
				\draw (-90:\ray-0.08cm) -- (-90:\ray+0.08cm);
				\node[below] at (-90:\ray) {{$0$}}; 
				
				\foreach \drop/\stepa/\stepb/\angledrop/\anglegauche/\angledroite/\decalx/\decaly in {0/2/3/50/30/60/0.5/0, 1/4/5/140/190/230/0/-0.8, 2/6/7/70/80/120/0.5/0}{
					{ 
						\draw[color=black, thick ] (\angledrop:\ray-0.1cm) -- (\angledrop:\ray+0.1cm);
						\node[] at (\angledrop:\ray-0.3cm){$u'_\drop$};
					}
				}	
				{\begin{pgfonlayer}{background}
						\draw[line width=5pt, color=blue_drop] (0,0)+(50:\ray) arc (50:-60:\ray);
					\end{pgfonlayer}
				}			
			\end{tikzpicture}
			\subcaption{}\label{subfig:illutheo:length22}
		\end{subfigure}
		\caption{\modif Illustration of Theorem \ref{theo:length2}. On the left, some CDM with 3 masses $m_0$, $m_1$ and $m_2$ arrived respectively at $\bu_0$, $\bu_1$ and $\bu_2$, and are dispersed according to some CDM. On the right, the same total mass arrives, but two masses are zero: that is $m_0' = m_0+m_1+m_2$, $m_1'=0$ and $m_2'=0$; the arrival location are respectively at $u'_0$, $u_1'$ and $u_2'$. The CDM used in this case is the RDCS. In both cases, the free spaces $\bF{3}$ and ${\bF{3}}'$ have the same distribution. }
		\label{fig:illutheo:length2}
	\end{figure}
	
	\subsubsection{On the occupied blocks distribution} 
	
	According to Theorem \ref{theo:excha}, to understand the distribution of $\l(\bO{k},\bF{k}\r)$ at time $k$ for any given \ACDM, it suffices to describe the distribution of $\l(\bO{k},\bF{k}\r)$ for the RDCS, also at time $k$.
	As we will see the distribution of $\bO{k}$ is more complex than that of $\bF{k}$.
	
	A technical tool comes into play. 
	\begin{defi} For some masses $m[k]:=(m_0,\cdots,m_{k-1})$ with sum in $[0,1)$, the piling propensity %\footnote{This piling propensity is ``close to'' the probability that all the masses arrive within an interval of length $W(m[k])$ which is $W(m[k])^{k}$.} 
		of $m[k]$ is defined by 
		\beq\label{eq:Q} Q(m[k]):=W(m[k])^{k-1}.\eq
	\end{defi}
	
	By Theorem \ref{theo:excha},
	\beq Q(m[k]) = \P(\bN_k=1),\eq so that $Q(m[k])$ measures the propensity of  $m[k]$ to form a single occupied CC, under the action of any valid \ACDM.
	The notation $Q(m[k])$ perhaps hides the fact that $Q$ depends on $k$ and of $W(m[k])$ only.\par
	Given that $\bN_k=1$,  $\bO{k}$ is reduced to a single (random) interval $\fleche{[A,A+W(m[k])]}$ and by invariance under rotation $A$ is uniform on $\R/\Z$:
	\[  \P(A \in \d a ~|~ \bN_k=1)=\d a\]
	so that
	\beq  \label{eq:NP} \P(A \in \d a,  \bN_k=1)=  W(m[k])^{k-1}\,\d a.\eq 
	This single-block consideration accounts for the fact that all $m_i$ fall within $[a,a+W(m[k])]$, with one of the $m_i$ arriving at $\d a$, and the others combining to cover exactly $[a,a+W(m[k])]$. This mundane observation allows to produce a multi-block formula at the double cost of fixing the identity of the masses participating to each block, and placing the blocks precisely on the circle so that they do not intersect  (this suffices, since for any valid \ACDM the dispersion policy of a mass is independent of what is outside of the connected component containing that mass). 
	
	For a set of indices $J$, set
	\[W(m(J)):=\sum_{j\in J} m_j.\]
	Let $\bI{k}_j$ denote the indices of the masses among $m[k]$ that have been dispersed to form the occupied CC $\bO{k}_j=\fleche{\l[\bA{k}_j,\bB{k}_j\r]}$. 
	We have, mainly as a consequence of the ``single block formula'' \eref{eq:NP},
	\begin{theo}\label{theo:full_dist}Assume the same hypothesis as in Theorem \ref{theo:excha}. For any $b\in\{1,\cdots,k\}$, any partition $(J_0,\cdots,J_{b-1})$ of the set $\{0,\cdots,k-1\}$ with non-empty parts, any sequence $(a_0, \cdots,a_{b-1})$ in ${\cal C}^b$ such that the $a_0, \cdots,a_{b-1}$ are cyclically ordered around ${\cal C}$ (i.e.\ such that turning around the circle we get $a_0\preceq 0\preceq a_1\preceq \cdots \preceq a_{b-1}\preceq a_0 \preceq 0 \cdots$) and such that
		\[\flecheu{d}(a_i,a_{i+1 \mod b})> W(m(J_i))~~\textrm{ for }i\in\{0,\cdots,b-1\},\]  
		we have
		\beq
		\P\l( \bA{k}_j \in \d a_j, \bI{k}_j =J_j,0\leq j \leq b-1\r)= \prod_{j=0}^{b-1} Q(m(J_j))\, \d a_j.
		\eq
	\end{theo}
	\noindent $\bullet$ The condition $\flecheu{d}(a_i,a_{i+1 \mod b})> W(m(J_i))$ is needed because the interval $\bO{k}_i$ has length  $W(m(J_i))$, starts at $a_i$ and must end before the start of the next occupied CC, \\ 
	$\bullet$ In this formula, a bias appears under the condition $a_0\preceq 0\preceq a_1$. \\
	$\bullet$ Notice that the event $\l\{ \bA{k}_j \in \d a_j,\bI{k}_j =J_j,0\leq j \leq b-1\r\}$ characterises entirely $\bO{k}$ as well as $\bF{k}$.\\
	To extract the occupied block lengths, we need to ``count'' the number of ways to produce some given block lengths.
	\begin{theo}\label{theo:dqgsrt}Take the same hypothesis as Theorem \ref{theo:excha}. Let $M_0,\cdots,M_{b-1}$ be some non-negative masses. For all $a_0\preceq 0\preceq a_1\preceq \cdots \preceq a_{b-1}\preceq a_0  $ and such that
		$\Leb\l( \fleche{[a_i,a_{i+1 \mod b}]}\r)> M_i$ for all $i\in \{0,\cdots,b-1\}$, 
		\beq\label{eq:qdfe}
		\P\l( \bA{k}_j \in \d a_j,  \l|\bO{k}_j\r|=M_j, 0\leq j \leq b-1\r)= \sum_{(J_0,\cdots,J_{b-1}) \in P(k,b)} \prod_{j=0}^{b-1} Q(m(J_j)) \ind{W(m(J_j)) = M_j}\, \d a_j
		\eq where $P(k,b)$ is the set of partitions of $\{0,\ldots, k-1\}$ into $b$ parts. 	
	\end{theo} 
	In some cases, Formula \eref{eq:qdfe} can be simplified. For instance, when the cardinalities of the parts $J_j$ of the partitions $(J_0,\cdots,J_{b-1})$ compatible with $(M_0,\cdots,M_{b-1})$, have a fixed number of elements. For example, if all the masses $m_j$ are equal to some constant $w$, then the possible lengths $M_j$ are multiples of $w$ (there exists $k_j\in \mathbb{N}$ such that $M_j=k_j w$), and then in this special case $Q(m(J_j))= ( M_j)^{k_j-1}$ and  the number of partitions $(J_0,\cdots,J_{b-1})$ in non-empty parts, such that, for all $i$, $W(m(J_i))=M_i$ is $\binom{k}{k_0,\ldots, k_{b-1}}$ (where the total mass $\sum_{i=0}^{b-1} M_i=kw$). In this special case, we then have $$
	\P\l( \bA{k}_j \in \d a_j,  \l|\bO{k}_j\r|=M_j, 0\leq j \leq b-1\r)=  \binom{k}{k_0,\ldots, k_{b-1}}  \prod_{j=0}^{b-1} (k_jw)^{k_j-1} \, \d a_j.$$

	Again, Theorem \ref{theo:dqgsrt} characterises the distribution of $(\bO{k},\bF{k})$.
	Computing  the ``marginal'' $|\bO{k}|$, which records only the block lengths is also possible. To do so, we need to compute ${\sf T}(M[b])$ the volume of the ``translation set'' $(a_0,\cdots,a_{b-1})$ described in Theorem \ref{theo:dqgsrt}:  \beq\label{eq:OlengthMarginal}
	\P\l(  \l|\bO{k}_j\r|=M_j, 0\leq j \leq b-1\r)= {\sf T}(M[b]) \sum_{(J_0,\cdots,J_{b-1}) \in P(k,b)} \prod_{j=0}^{b-1} Q(m(J_j)) \ind{W(m(J_j)) = M_j} 
	\eq
	where
	\beq\label{eq:TsetMeasure} {\sf T}(M[b])= M_0 \frac{(1-W(m[k]))^{b-1}}{(b-1)!}+\frac{(1-W(m[k]))^{b}}{b!}\eq
	(see Section \ref{sec:JTset} for a proof which relies on the fact that this set has same measure as $S:=\{ (u,s_0,\cdots,s_{b-2}), u \in [0,M_0+s_0], 0\leq s_0 \leq s_1\leq \cdots s_{b-2}\leq 1-W(m[k])\}$ where the variable $s_j$ accumulates the free length spaces between the $\bO{k}_0$ and $\bO{k}_{j+1}$, and the variable $u$ is used to place zero in the interval $[0,M_0+s_0]$)).  
	
	The measure $T(M[b])$ depends only on $(M_0,W(m[k]))$. The presence of $M_0$ is due to the size bias.

	\begin{rem}[bias of $\bO{k}_{0}$ and $\bF{k}_{0}$]\label{rem:bias}
		Equations \eqref{eq:OlengthMarginal} and \eqref{eq:TsetMeasure} show that $\bO{k}$ is biased, and the global bias of a configuration is proportional to  $|\bO{k}_0|+|\bF{k}_{0}|$. 
		
		The block $\bF{k}_0$ is also biased, and while the rotation of the blocks in Theorem \ref{theo:excha}.\ref{item:iv_in_theo_excha} allowed to avoid dealing with it, we can also characterise the distribution of $\bF{k}_0$: 
		
		We condition on $\bN_k=b$. Recall that the entire process is invariant under rotation and that the occupied blocks are exchangeable:\\
		-- either $0$ belongs to an occupied block (with probability $1-R_k$), and $\bF{k}_0$ is the block that follows. Its length is $R_k \times \ell'$, where $\ell'$ is a $\beta(1, b{-1})$ distributed random variable (thus $|\bF{k}_0{}'|$ has density $ g_2(x):=  (b{-1})(1-x/R_k)^{  b{-2}}/R_k\1_{x\in [0,R_k]}$, this is the distribution of the first marginal in a $\Dirichlet(b;1,\cdots,1)$ random variable),\\ 
		-- or $0$ belongs to a free block (which is thus $\bF{k}_0$), in this case $|\bF{k}_0| = R_k \times \ell$ where $\ell$ a $\beta(2,b{-1})$ random variable (and thus $|\bF{k}_0|$ has density $g_1(x)=b(b {-1})(x/R_k)(1-x/R_k)^{b-2}/ R_k \1_{x\in [0,R_k]}$).
		
		Finally, conditional on $\bN_k=b$, the density of $\bF{k}_0$ is $ 
		\big(  R_kg_1(x) + (1-R_k)g_2(x)\big)$.
	\end{rem}

	Informally, a multiset $\{\{ x_1, \cdots,x_m\}\}$ is  a set in which elements may have an arbitrary multiplicity\footnote{Define an equivalence relation between sequences: two sequences are equivalent if they have the same length, and if they are equal up to a permutation of their terms. A multiset is an equivalence class for this relation. }.

	The final universal result we would like to present concerns ``the coalescence'' process induced by the diffusion models we study.
	\begin{theo}\label{theo:proce} Let $(m_0,m_1,\cdots,m_{k-1})$ be some masses with sum in $[0,1)$. Consider, for $j\in\{0,\cdots,k\}$, the multiset 
		\[S(j):=\l\{\l\{\l|\bO{j}_i\r|, 0\leq i\leq \bN_j-1\r\}\r\}\]
		which provides the lengths (with multiplicity) of the occupied CC at time $j$.
		
		The distribution of the process $(S(j),0\leq j \leq k-1)$ is the same for all valid \ACDM.
	\end{theo}
	\begin{rem}\label{rem:qdsezgeh}
		Free spaces do not have the same universal property: \modifs{the distribution of the process of free space lengths (indexed by the time $j$) depends on the policy. For example,} there is a  positive probability that using the RDCS from time $t$ to $t+1$ with a non zero-mass, a single free CC is reduced. However, when using the ``{$(p,1-p)$ proportion to the right/left diffusion}'' of Section \ref{sec:p1mp}, this  never happens.   
	\end{rem}

	\begin{rem}
		Chassaing \& Louchard \cite{chassaing2002phase}, considered $S(t)$ the multiset of occupied block lengths at time $t$ in the parking process. The blocks correspond to a (maximal) set of consecutive occupied places along with the first free place to the right of them. 
		They proved that the process $(S(t),t\geq 0)$ has the same law as in the additive coalescent process starting with $n$ particles of mass 1 observed at successive coalescence times. This process can be encoded by coalescent forests (Pitman \cite{MR1673928}), in which tree sizes correspond to block lengths. The joint distribution of the tree sizes is known: suitably ordered they have the same distribution as i.i.d. Borel variables\footnote{The Borel law ${\sf Borel}(\lambda)$, is the discrete distribution $p_\lambda$ with support $\{1,2,3,\cdots\}$ defined by $p_\lambda(k)= (\lambda k )^{k-1}\exp(-k\lambda)/k!$.} conditioned to have a fixed sum (see e.g. Pitman \cite[Prop.6]{MR1673928}, Bertoin \cite[Cor. 5.8]{MR2253162}, \cite[Proof of Prop. 5.1]{chassaing2002phase}, Marckert \& Wang \cite{MW2019}... ).   
	\end{rem}

	\subsection{Reduction to the Right Diffusion with Constant Speed and excursion sizes} 
	\label{sec:collecting_paths}
	Theorem \ref{theo:excha} tells us that for any valid \ACDM, given $m[k]$, the distribution of $(\bO{k},\bF{k})$ can be studied using the RDCS.
	We recommend looking at Figure \ref{fig:RDCS3} before reading the following text.
	\begin{figure}[h!]
		\centerline{\includegraphics[scale=0.75]{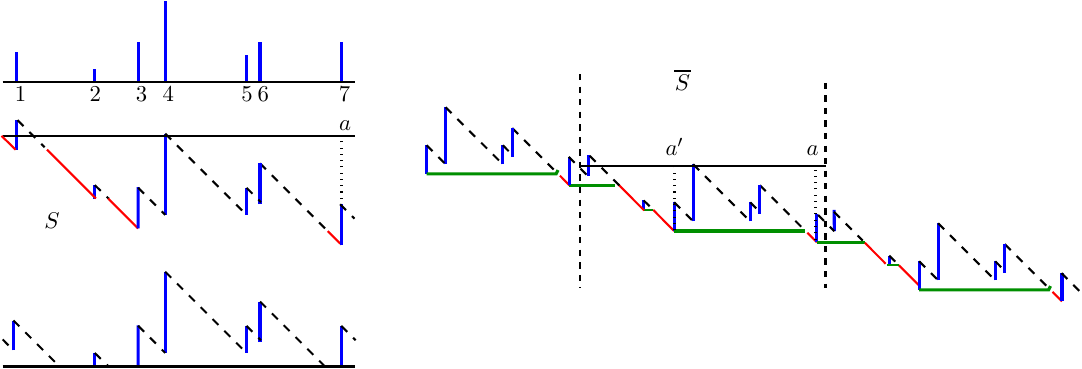}}  
		\captionn{\label{fig:RDCS3}\underbar{On the first column:} By Theorem \ref{theo:excha}, the masses can be treated in our preferred order. If we use the order 7,1,2,3,4,5,6 with the RDCS, we get the third picture. Since we are on the cycle $\R/\Z$, there are three occupied CC. The second picture represents the process $S$. To recover the excursions, there is a boundary effect. For example, the first excursion above the ``minimum process'' does not correspond to an occupied CC.\\ \underbar{In the second column}, one can see that  working with $\bar{S}$ which is obtained by pasting several trajectories of $S$ head to tail, then, starting from the abscissa $a$ of the minimum of argmin $S$ on $[0,1]$, one recovers the occupied block lengths as the excursion lengths above the minimum process of $\bar{S}$   on $[a,a+1]$.\\
			The same property holds if one considers instead $a' = \argmin(t\mapsto S_t + S_1 t)$ on $[0,1]$\\
			Notice also that this encoding is not sufficient to encode zero masses, since they leave unchanged the collecting path, while possibly creating CC. The red portions of the paths correspond to the points where the left to right minimum decreases. Thus, excursions correspond to the portions of $\bar{S}$ between red points. The red portions of the paths do not  correspond to red parts on $S$ that are represented on the left column because of a border effect at zero, to take into account to recover the excursion that straddles zero.   }   \end{figure}
	\begin{defi}
		Define the ``collecting path'' process $S:=(S_x,0\leq x \leq 1)$ by 
		\beq\label{eq:fegee} S_x =  -x + \sum_{j=0}^{k-1} m_j \1_{\bu_j\leq x},~~~\textrm{for }x \in[0,1].\eq
		The extended collecting path $\bar{S}=(\bar{S}_x,x\in \R)$ is a process indexed by $\R$, defined as on Figure \ref{fig:RDCS3} by concatenating copies of the collecting paths head to tail,
		\[\bar{S_x}= S_{\{x\}}+ \floor{x} S_1,~~~~\textrm{for }t\in\R\]
		where $\{x\}$ is the fractional part of $x$.
	\end{defi} The collecting path is named ``profile'' in Bertoin \& Miermont \cite{bm2006} (but the term ``profile'' is used in many other contexts, with   few commonalities to the present situation. The term  ``collecting path'' seems more appropriate for its function of collecting and updating  server tasks over time).
	For all levels $y\in \R$ set $\tau_y=\inf\l\{t: \bar{S}_t=y\r\}$ as the first hitting time of $y$. To define the excursion of $\bar{S}$ at level $y$, set $\tau^y=\sup\l\{x: \min\{ \bar{S}_u, u \in[\tau_y,x]\}=y\r\}$.
	Set $\ell_x:=\tau^y-\tau_y$ the length of the excursion above level $y$.
	Let $a$ be the first time that $S$ reaches its minimum on $[0,1]$:
	\[a = \min \argmin S.\]
	\begin{lem} \label{lem:fgth}
		The multiset of positive occupied CC lengths in the RDCS with mass arrival events  $[(m_i,\bu_i),0\leq i \leq k-1]$, corresponds to the non-zero excursion lengths of the collecting path $\bar{S}$ for $x\in [a,a+1)$ above the current minimum process.
	\end{lem}
	\begin{rem}\label{rem:fgth2}	Zero masses leave unchanged the collecting paths and can create zero-length CC, which cannot be recovered in $\bar{S}$.   
		Lemma \ref{lem:fgth} is proven in \cite[Lemma 3]{bm2006} (and similar discrete results are present in \cite{Aldous1997}, \cite{chassaing2002phase}, ...). Studying the sizes of excursions  in collecting paths is an important tool for  analysing  cluster sizes in multiplicative and additive coalescence processes. The link between them is given by this Lemma.
	\end{rem}
	\begin{rem}[\ACDM with random masses] Hence, for some deterministic masses $m[k]$ with sum $W(m[k])<1$, the 8 dispersion models introduced in Section \ref{sec:LE} induce the same distribution for $(\bO{k},\bF{k})$. This is also true, as a simple corollary, if they use the same random masses. The number of CC of $\bF{k}$ is then always one plus a random variable with law ${\sf Binomial}(k-1,W(m[k]))$. The law of their lengths, and other characteristics are the same in all cases. In particular, they can be studied using the RDCS model, and correspond then (up to the boundary effect discussed in Figure \ref{fig:RDCS3}), to the excursion lengths of the collecting paths above the current minimum.
	\end{rem}

	\subsection{Related works / Motivation}\label{sect:related_work}

	Here, we would like to discuss some related work and provide  perspective on the phenomena at play in the paper, especially regarding the universality results obtained. Although diffusion processes form a large family of models encompassing many existing models, their global universality properties can be viewed as generalisation, or infinitesimal generalisation, of phenomena observed and utilised  in various domains of combinatorics, and probability theory. A large part of combinatorics aims to understand macroscopic objects made of small  components that respects certain rules.  The set of   objects of size $n$ is sometimes equipped with a probability distribution induced by a given construction, sometimes corresponding to the
	uniform distribution. The same distribution can often be obtained by several different constructions. Most of the time, these  \textit{identities in distribution} between constructions cannot be proven directly. For example, Aldous \& Broder construction of the uniform spanning tree is very different from Wilson's algorithm. They use different {sources of randomness}, and each of them cannot be reduced to the other, even if they ultimately generate the  same distribution on the set of spanning trees.\par

	In the present paper, the main reason for the universality is given in Section \ref{sec:ryjqdqsd}. Roughly speaking, it relies on the proof that infinitesimal deposition at a boundary of the relaxation interval containing the currently treated mass can be modified without distorting the configuration distribution after this infinitesimal time.  Infinitesimal modifications to the dispersion policies alter the occupation seen as a time-indexed process, but do not change its distribution after $k$ relaxations. \par
	This type of ``modification/displacement'' of elementary constituents in a structure, while still controlling the resulting distribution,  is central to many significant advances in research devoted to understanding random combinatorial models. In the rest of this section, we will give some key elements, growing in complexity, that exemplify this statement.\par
	The most ``elementary'' result in this field consist in bijections made of step rearrangements.
	
	For example, as illustrated in Figure \ref{fig:bri-mean}, when $N=2n$ is even, the number of meanders with $N$ steps (the set of paths starting at 0, with steps $\pm1$ and staying non-negative) coincides with the number of size $N$ bridges (paths starting and ending at 0 with steps $\pm1$).
	\begin{figure}[h!]
		\centering
		\includegraphics[height=3cm]{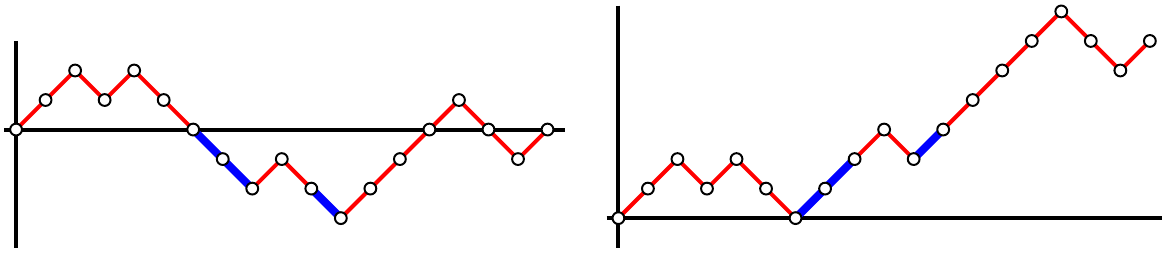}
		\captionn{\label{fig:bri-mean}{Illustration of a bijection between the set of paths with steps $(+1,-1)$ ending at 0 with $2n$ steps (called bridges), and the set of meanders with same size (paths that stay non-negative all along). To convert a bridge into a meander, replace each down-step that leads to a new ``left-to-right minimum'' with an upstep. To convert a meander into a bridge, replace each upstep corresponding to the last passage of an edge of this level with a downstep. Do this for all $k+1$ smaller to half the final meander value. }} 
	\end{figure}  A bijection between these sets consists in turning over, in the bridge, the steps at which a new record minimum is reached. A more involved rearrangement allows to represents paths on a graph as pairs $(SAW,EC)$ where $SAW$ is a self-avoiding walk, and $EC$ a heap of cycles supported by $SAW$ (see Viennot \cite{VX}, Zeilberger \cite{DZ},  Marchal \cite{Marchal}, Marckert \& Fredes \cite{MR4545865}).

	Rearranging the steps of structures formed by several paths is at the core of important results obtained during the last decades in probability theory: we can cite in that respect \\
	-- Propp and Wilson \cite{PW98} sampling method of the uniform spanning tree,\\
	-- the result of Diaconis and Fulton \cite{diaconis1991growth} for a certain commutation property regarding the set of final position of random walks, when particles start successively from a sequence of vertices $s_1,\cdots,s_j$ (with possible multiplicity) and perform random walks stopped at a vertex visited for the first time (by one of the random walks): the final set of visited vertices has same distribution if the set $(s_1,\cdots,s_j)$ is permuted. This result is of the prime importance in the study of internal DLA (see e.g.  Lawler \& al. \cite{MR1188055}) and in the study of loop-erased random walks (see e.g. Lawler \cite{lawler1992internal}).
	
	A direct motivation of the present paper are the results obtained in Varin \cite{varin2024golf} devoted to studying of the golf process (defined by Fredes \& Marckert \cite{MR4545865} as a tool in a new proof of Aldous--Broder theorem). 
	The golf process is defined on a graph in which some vertices initially contain  a ball, some vertices contain a hole, and some vertices are neutral. Then, each ball moves according to a random walk and stops at the first hole it encounters. This hole then  becomes a neutral vertex for subsequent balls, .
	
	As for Diaconis and Fulton commutation property, the order in which the balls are activated does not affect the distribution of the remaining holes at the end of the process. Even better, on $\Z/n\Z$, if $b$ balls and $h$ holes are placed (according to a uniform distribution on the set of possible choices), then the  final  set of free holes $H$ is the same for a large class of ball displacement policies that largely encompass random walks model. As explained in  \cite{varin2024golf} (and a similar remark has been done by Nadeau \cite{nadeau_bilateral_parking_procedures}), the distribution of $H^1$ is the same for all displacement policies under which a ball at a given position determines its  trajectory  distribution, using (at most) the data of its current position in the interval formed by the first hole to its left and right.  This includes random walks that do not use this information, and many policies unrelated to random walks.
	
	Hence, on $\Z/n\Z$, this universal distribution property is stronger than the commutation property of  Diaconis and Fulton. This result suggested that a strong property of exchangeability should hold for general diffusion policies, including continuous ones, on spaces possessing some additional symmetries. This leads us to the present work.
	
	In the standard parking model defined on $\Z/n\Z$, the cars move in the positive direction (to the right). The occupation distribution, when $k$ cars have been parked is known (see e.g.  \cite{Konheim1966AnOD,pittel1987linear,chassaing2002phase,MR2253162}), and is invariant by the space symmetry $S:i\mapsto n-i$. This means that if, after a car arrival, we perform the symmetry $S$ with probability $p$, then the distribution of the parking is unchanged\footnote{{The reasoning can be extended: if after each unit displacement of a car, the symmetry is made with probability $p$, or even with a probability depending on this $k$ previous moves and previous symmetries done, then finally, as long as the policy chosen ensures that the car a.s. parks, the distribution of the parking occupation after the parking of the $k$th car has same distribution as in the standard parking model. To prove that, by recurrence, it suffices to  start with the parking already occupied  by $k-1$  cars, placed using  the standard model). When the $k$th car moves according the any local policy which is independent from the parking outside the current occupied block it is moving in, it induces a distribution for the $k$-cars parking occupation which is a mixture on ``standard right directed parking'' and ``standard left directed parking'', the two being equal at time $k$.}
	}\footnote{ On a segment too, if the cars decide to ``go to the right'' with probability $p$ and to the left with probability $1-p$, independently, as explained by Durmi\'c \& al. \cite{MR4626349}, the probability for $n-1$ cars to park on the segment (that is, no car gets out of it), is independent of $p$. As we will discuss here, these phenomena are much more general.}. However, as a process indexed by $k$, this is not longer true, since in the standard parking model, large occupied blocks essentially grow  from the right. 
	For example, it implies that  if the cars are directed either to the right, or to the left, with probabilities $p$ and $1-p$, then, finally,  the occupation statistics is the same for each value of $p$ at any fixed time $k$ (this is false for the temporal process). As we will prove, the fact is that even if several cars are allowed to arrive simultaneously (say $n_i$ cars arrive at position $\bu_i$, uniform on $\Z/n\Z$, at time $i$), even if their policies may be complex and depend on  the previous cars' choices, as long as during the evolution of the car positions, each displacement depends at most on their joint positions on the current CC of parked cars that contains them (or equivalently, to the distance to the interval formed by the first free place on the left and on the right), then, the distribution of the geometry of the final occupied places depends only on the vector $(n_1,\cdots,n_k)$, and not on the particular dispersion policy considered (and it is the same for all permutation of the $n_i$). 
	This remark is important because  the asymptotic behaviour of ``these general parking'' (that will be considered as discrete dispersion models here) coincides with parking in which the cars (or caravans) are right directed. Results  for these cases are already known (Chassaing \& Louchard \cite{chassaing2002phase}, Bertoin \&  Miermont \cite{bm2006}).  
	
	The evolution of occupied blocks consists of the creation of CC (when a mass arrives in a free CC), the growth of blocks (during the mass dispersion), and the coalescence of existing blocks. For this reason, and because the (discrete) additive coalescence process is related to parking models (\cite{chassaing2002phase,MR1673928,MR3547745}, coalescence considerations and representation will be important throughout the paper.

	\subsection{Proof of Theorem \ref{theo:excha}, Theorem \ref{theo:length2} and Theorem \ref{theo:proce}} 
	\label{sec:ryjqdqsd}

	The main argument is the proof of existence of a very robust invariant for all dispersion models. It is robust in the sense that it holds when  $k$ masses have been dispersed, but also during the dispersion of the $k$th mass: conditional on $\bN_{k} = b$, the lengths $(\ell_0,\cdots,\ell_{b-1})$ of the free blocks are Dirichlet distributed (if $R =1-\sum m_i$, their distribution is $R\l[\bD_0,\cdots,\bD_{b-1}\r]$ where $\l[\bD_0,\cdots,\bD_{b-1}\r]\sim \Dirichlet(b;1,\cdots,1)$); again, this is true at the first order, because a bias comes from the possible dependence between the indices of the $\ell_i$ and the position of zero on ${\cal C}$, but mainly, the proof consists in exhibiting this invariant, and establishing its conservation during the "execution" of any valid dispersion policy. \medskip

	The proofs rely on some elementary principles. In this section, we introduce them step by step, describing what happens when we progressively add some mass arrival events on the circle.\medskip	
	
	Consider $k$ i.i.d. uniform points $\bu_0,\dots,\bu_{k-1}$ taken on the circle ${\cal C}=\R/\Z$.
	
	\paraa{A static principle.}~\\
	Take $I$ to be a measurable subset of ${\cal C}$ with positive Lebesgue measure, that can be random or not, but is chosen independently of the $\bu_i$.
	
	If we condition the set $S:=\{\bu_i, 0\leq i \leq k-1\}$ to satisfy $|S\cap I|=m$ (i.e. to contain $m$ elements of $S$), then $S \setminus I$, the elements of $S$ that are not in $I$, have same law as a sample of $k-m$ uniform random variables on ${\cal C}\backslash I$.
	Of course, the cardinality of $|S\cap I|$ has law ${\sf Binomial}(k,\Leb(I))$.

	\paraa{A space eating principle.}~\\ 
	We will ``eat'' the space around $\bu_0$ at constant speed. We introduce two continuous non-decreasing, non-negative functions $\alpha$ and $\beta$, such that $\alpha(0)=\beta(0)=0$, and
	\beq\label{eq:qgs} \beta(t)+\alpha(t)=t,\textrm{ for } t\in [0,1).\eq
	Again, we allow $((\alpha(t),\beta(t)),t\geq 0)$ to be a random process, independent of the $\bu_i$.
	For this, define $s[t]$, ``the swallowed space at time $t$ around $\bu_0$'', as the length $t$ interval on the circle,  
	\[s[t]:= \fleche{[ \bu_0 -\alpha(t), \bu_0 + \beta[t]]}.\] 
	Now, let us consider the non-eaten variables $\bu_j$ at time $t$:
	\[ X(t):= \{\bu_0,\cdots, \bu_{k-1}\}   \backslash s[t] .\] 
	At any time $t\in[0,1)$, the (remaining) non-eaten space $R(t) := {\cal C} \setminus s[t]$, satisfies
	\[\Leb( R(t) )= 1-t;\]
	each of the $\bu_i$ is in $s[t]$ with probability $t$ (except $\bu_0$ which is surely in $s[t]$), so that
	\[|X(t)| \sim {\sf Binomial} (k-1, 1-t)\]
	and conditional on $|X(t)|=j$, the non-eaten variables $\bu_i$ are i.i.d. uniform in $R(t)$.\\
	Notice that the origin 0 of ${\cal C}$ is in $s[t]$ with probability $|s[t]|=t$, by independence of $(\alpha(t),\beta(t))$ to the $\bu_j$.

	\paraa{The Markovian dynamics of $|X(t)|$.}
	\begin{lem}\label{lem:infinitesimal}Provided that $\alpha(t)+\beta(t)=t$ for all $t$, we claim that the process $(|X(t)|,t\in[0,1])$ is a Markov process independent of $(\alpha(t),\beta(t))$, and then,  the law of $(|X(t)|,t\in[0,1])$ is the same regardless of the process $(\alpha(t),\beta(t))$.\end{lem}
	Indeed, from time $t$ to time $t+\d t$, the swallowed space passes from a Lebesgue measure $t$ to $t+\d t$, and this is true whatever are the relative speeds of $\alpha(t)$ and of $\beta(t)$ (since this speed must be independent of the uneaten points).
	
	Each element of $X(t)$ survives up to time $t+\d t$ with probability ${\Leb}(R(t+\d t))/{\Leb}(R(t))$, and at time $t+\d t$, each surviving random variable (included in ${X(t+\d t)}$) will be uniform in $R(t+\d t)$.  
	\\
	
	It is worth noticing that $|X(t)|$ is always a Markov process even if $(\alpha(t),\beta(t))$ is not! 
	
	More than that, we may allow $(\alpha(t+\d t),\beta(t+\d t))$ to depend on the number/identity of the swallowed $\bu_j$ before time $t$ (and also, of the absorption times): as long as at any time $t$, $(\alpha(t+\d t),\beta(t+\d t))$ is defined independently from the non-swallowed uniform points, the probability of absorption of a new point between time $t$ and time $t+\d t$ does not depend on the choice of the distribution of $(\alpha,\beta)$, which implies that $|X(t)|$ is Markovian: the distribution of $(|X(t)|,t\geq 0)$ is universal for this class of processes $(\alpha,\beta)$.
	
	We add that in the preceding discussion we preferred to use the infinitesimal notation $dt$ instead of taking a small $\epsilon>0$, because the dispersion policy can change when $s[t]$ meets with another occupied component. We could also have worked with stopping times, but this would complexify the exposition with no gain.

	\paraa{Relation between \CDM and space-eating configurations}~\\
	In the previous paragraphs we discussed the space-eating model, and we let an interval $s[t]$ grow around $\bu_0$, as in the continuous diffusion process. In a diffusion process where the points $\bu_1$ to $\bu_{k-1}$ have zero masses while  $\bu_0$ is the only point having a positive mass, this mass would diffuse around, and cover a space $s[m_0]$ (with Lebesgue measure $m_0$), and cover/eat the other $\bu_j$ in $s[m_0]$. The conditions we took for $(\alpha(t),\beta(t))$, namely unit-speed deposition and independence with respect to $\bu_1,\cdots,\bu_{k-1}$, are the conditions for a valid \CDM. Once again, the  second  item of \textbf{(C)} applies: as long as the  eaten-space  dynamics depend only on the eaten points and on the present eaten space $s[t])$, the distribution of the final configuration say, at time $m_0$, does not depend on the details of the definition of $(\alpha(t),\beta(t))$. 
	
	\begin{figure} \centering
		\begin{subfigure}[b]{0.45\textwidth} \centering
			\begin{tikzpicture}[scale=1.3]
				\def\ray{1.25cm} \def\raysupvert{1.6cm} \def\raysuphoriz{1.5cm}
				\clip(-\raysuphoriz,-\raysupvert) rectangle (\raysuphoriz,1.9cm);
				\draw[black] (0,0) circle (\ray);
				
				\draw (-90:\ray-0.08cm) -- (-90:\ray+0.08cm);
				\node[below] at (-90:\ray) {{$0$}}; 
				
				\foreach \drop/\stepa/\stepb/\angledrop/\anglegauche/\angledroite/\decalx/\decaly in {0/2/3/50/30/60/0.5/0,1/4/5/280/190/230/0/-0.8, 2/6/7/130/80/120/0.5/0}{
					
					{ 
						\draw[color=black, thick ] (\angledrop:\ray-0.1cm) -- (\angledrop:\ray+0.1cm);
						\node[] at (\angledrop:\ray-0.3cm){$u_\drop$};
					}
				}
				\node[color = blue_drop, xshift=0.5cm, yshift=0 cm] at (90:\ray+0.3cm) {$m_0$};	
				\begin{pgfonlayer}{background}
					\draw[line width=5pt, color=blue_drop] (0,0)+(20:\ray) arc (20:160:\ray);
				\end{pgfonlayer}			
			\end{tikzpicture}
			\subcaption{CDM on $\mathcal{C}$ after mass $m_0$ has been spread}\label{subfig:illu_corres_space_eaten_and_CDM1}
		\end{subfigure}\hfill
		\begin{subfigure}[b]{0.53\textwidth} \centering
			\begin{tikzpicture}[scale=1.3]
				\def\ray{.76cm} \def\raysupvert{1.6cm} \def\raysuphoriz{1.5cm}
				\clip(-\raysuphoriz,-\raysupvert) rectangle (\raysuphoriz,1.9cm);
				\draw[black] (0,0) circle (\ray);
				
				\draw (-90:\ray-0.08cm) -- (-90:\ray+0.08cm);
				\node[below] at (-90:\ray) {{$0$}}; 
				
				\node[color = blue_drop, xshift=0.5cm, yshift=0 cm] at (90:\ray+0.3cm) {$u_0,u_2 \in s[m_0]$};	
				\draw[color=blue_drop, very thick ] (90.3:\ray-0.1cm) -- (90.3:\ray+0.1cm);
				\node[] at (90.3:\ray-0.3cm){$v_0$};
				
				\draw[color=black, thick ] (286.4:\ray-0.1cm) -- (286.4:\ray+0.1cm);
				\node[] at (286.4:\ray-0.3cm){$v_1$};
				
			\end{tikzpicture}
			\subcaption{Eaten-space dynamics on $\mathcal{C}\backslash\sim_{s[m_0]} = \R/(1-m_0)\Z$. 
			}\label{subfig:illu_corres_space_eaten_and_CDM_2}
		\end{subfigure}
		\caption{Illustration of the correspondence between the continuous dispersion model and the eaten-space dynamics. When mass $m_0$ has been spread in the CDM, a portion $m_0$ of $\mathcal{C}$ has been eaten in the eaten-space model. At time 1, there are two uniform points on $\mathcal{C}\backslash\sim_{s[m_0]} $: $v_0$ (that contains all the points in $s[m_0]$, i.e. $\bu_0$ and $\bu_2$) and $v_1$ (that only contains $\bu_1$).}
		\label{fig:illu_corres_space_eaten_and_CDM}
	\end{figure} 
	This eaten-space point of view is illustrated on Figure \ref{fig:illu_corres_space_eaten_and_CDM}.	 
	
	\paraa{On the disappearance of the  eaten space.}~\\
	A topological solution for formalising the disappearance of space on ${\cal C}$ consists in identifying all the points of $s[t]$.
	This amounts to designing an equivalence relation $\sim_{s[t]}$ on ${\cal C}$ where $x\sim_{s[t]} y$ if $x=y$ or if $x$ and $y$ are both in $s[t]$. The quotient space ${\cal C}/\sim_{s[t]}$ is isomorphic to $\R/(1-t)\Z$ via the mapping that sends any $v \in [0,1)$ to $|[0,v]\backslash s[t]|$; this mapping induces an isomorphism from ${\cal C}/\sim_{s[t]}$ to $\R/(1-t)\Z$ since for every $v,v'\in s[t]$, $|[0,v]\backslash s[t]| = |[0,v']\backslash s[t]|$, and since for every $v$, $|[0,v]\backslash s[t]|\in[0,1-t)$. This isomorphism preserves the position of the origin, and it is then somehow canonical. In this latter space, the non-eaten random points $\bu_i$ are, conditionally on their number $b-1$, uniform on  $\R/(1-t)\Z$, while the eaten interval $s[t]$ corresponds to a point $\bu_0(t)$, independent of the others, but not distributed uniformly: $\P(0 \in s[t]) = |s[t]| = \P(\bu_0 = 0)$ and, conditional on $\bu_0 \neq 0$, $\bu_0$ is uniform on $\R/(1-t)\Z$.%  {\color{red} like a scar. Par contre si on prend cet isomorphisme, $\bu_0(t)$ est pas uniforme (mais $=0$ avec proba $t$ et uniforme avec proba $1-t$ (par contre, que l'on prenne cet isomorphisme ou un autre isomorphisme aléatoire, il n'y a pas de souci pour les "non-eaten random variables")) }

	We can relabel by $\bv_0(t),\cdots, \bv_{b-1}(t)$ these $b$ variables (with $\bv_0(t)=\bu_0(t)$, using the initial order of the labels between the surviving $\bu_j$ to determine the identities $1$ to $b-1$ of the vertices $\bv_1(t),\cdots,\bv_{b-1}(t)$). The vertices $\bv_1(t),\cdots, \bv_{b-1}(t)$ are distributed as $b-1$ uniform points on $\R/ (1-t)\Z$. As explained above, there is a bias concerning the position of $\bv_0(t)$, but we can still analyse the relative distances between them, as, for $\bu$ uniform on $\R/ (1-t)\Z$, $\bv_0(t)+\bu,\cdots, \bv_{b-1}(t)+\bu$ are distributed as $b$ uniform points on $\R/ (1-t)\Z$. They thus verify the following claim:  
	%\sout{If one is only interested in the joint position of the $v_j(t)$  on $\R/ (1-t)\Z$, one can see that, given that they are $b$ of them in total, they are distributed as $b$ uniform points on $\R/ (1-t)\Z$}
	\par
	
	\begin{claim}\label{claim:Dirichlet}
		Let $b\geq0$ be a (random) integer. Let $\bu_0,\ldots,\bu_{b-1}$ be $b$ uniform random points on $\R/ (1-t)\Z$, and let $\ell_0,\cdots,\ell_{b-1}$ be the sequence of lengths of the intervals of $\l(\R/ (1-t)\Z\r)\setminus\{\bu_i,0\leq i\leq b-1\}$  formed by the $\bu_i$, where these intervals are taken cyclically around $\R/ (1-t)\Z $, starting from the intervals at the right \footnote{we can start the labelling from any $\bu_i$ as long as the labeling is chosen independently from the interval lengths} of $\bu_0$. Then,
		\begin{equation}
			\big(\ell_0,\cdots,\ell_{b-1}\big)\eqd (1-t)\l[\bD_0,\cdots,\bD_{b-1}\r] \label{eqn:Dirichlet}
		\end{equation} where $\l[\bD_0,\cdots,\bD_{b-1}\r]\sim \Dirichlet(b;1,\cdots,1)$.
	\end{claim}
	
	Noticing that the relative distances between $\bv_0(t),\cdots, \bv_{b-1}(t)$ are invariant under rotation, it is then immediate that Equation (\ref{eqn:Dirichlet}) is also true for $\bv_0(t),\cdots, \bv_{b-1}(t)$.%, defining the $\ell_i$ similarly.

	%Letting $\ell_0,\cdots,\ell_{b-1}$ be the sequence of lengths of the intervals of $\l(\R/ (1-t)\Z\r)\setminus\{v_i(t),0\leq i\leq b-1\}$  formed by the $v_i(t)$, where these intervals are taken cyclically around $\R/ (1-t)\Z $, starting from the intervals at the right \footnote{we can start the labeling from any $v_i(t)$ as long as the labeling is chosen independently from the interval lengths} of $v_0(t)$, (conditionally on the fact that there are $b$ points), then
	%\beq\big(\ell_0,\cdots,\ell_{b-1}\big)\eqd (1-t)\l[\bD_0,\cdots,\bD_{b-1}\r]\eq where $\l[\bD_0,\cdots,\bD_{b-1}\r]\sim \Dirichlet(b;1,\cdots,1)$.

	On  ${\cal C}$, if one interprets now $s[t]$, as well as all the other $\{\bu_j\}$, as occupied intervals that are not covered by $s[t]$, then, these $b$ elements form $\bO{k}$, and the intervals in between (i.e.\ the free space lengths $|\bF{k}|$) are isomorphic to those in  $\R/(1-t)\Z$, and thus are distributed as  $(1-t)\l[\bD_0,\cdots,\bD_{b-1}\r]$.  		
	\medskip
	
	We now make some connections with our dispersion models.
	
	\paraa{Is it allowed in the \ACDM to place the points $(\bu_i,0\leq i\leq k-1)$ beforehand?}~\\
	As long as the dispersion of the masses $(m_i,\bu_i)$ does not depend on the presence of the points with a higher index, the masses can be placed beforehand. Their presence makes it clear that, when a single mass is dispersed while the other ones are zeroes,  the dispersion of this mass and the arrival of the others commute, if we are only interested in the distribution of the final configuration.\par
	This allows us to disconnect somehow ``the diffusion date'' $i$, from the arrival place $\bu_i$. The presence of points all together in the system allows to compare the effect of  diffusing  a small quantity of matter $\d m$ from one of these points or from another.

	\paraa{Space eaten around two points / \ACDM diffusion.}~\\
	Now assume that $\bu_0$ and $\bu_1$ have positive masses $m_0$ and $m_1$ while the other ones, $\bu_2$ to $\bu_{k-1}$, still have mass 0. Assume that $\bu_0$ is again equipped with two functions $(\alpha_0(t),\beta_0(t))$ satisfying $\alpha_0(t)+\beta_0(t)=t$ for all $t$, which defines an interval $s_0[t]=[\bu_0-\alpha_0(t), \bu_0+\beta_0(t)]$. In the \ACDM perspective  $\bu_0$ has invaded $s_0[m_0]$ at time $m_0$ on ${\cal C}$, while on the eaten-space perspective, $\bu_0(m_0)$ is now a point on $\R/((1-m_0)\Z)$ and this point can now have been identified with other $\bu_i$, including the other massive point $\bu_1$.

	From the \ACDM perspective, when $\bu_1$ starts, its environment has changed. There are two cases: either $\bu_1$ is in $s_0[m_0]$, or it is not.
	
	In both cases, it will use some functions $\alpha_1(t),\beta_1(t)$, from its activation time, $t=1$, to eat the space  or to diffuse. If $\bu_1$ is in $s_0[m_0]$ (which occurs with probability $m_0$ by the previous discussion), $\alpha_1(t)$ and $\beta_1(t)$ may depend on $(\bu_0,s_0[m])$, while if it is not in $s_0[m_0]$ (with probability $1-m_0$), $\alpha_1(t)$ and $\beta_1(t)$ must be independent from $s_0[m_0]$, at least before a possible collision.
	
	Define $s_1[t]= \fleche{[\bu_1-\alpha_1(t),\bu_1+\beta_1(t)]}$ and we assume now that $\Leb(s_1[t] \backslash s_0[t])=(t-1)$ for $t \in[1,1+m_1]$ so that $s_1$ grows at unit speed (out of $s_0[m_0])$.
	
	Again, there are two points of view on the evolution of $s_1[t]$ as $t$ grows: one is as a \ACDM on ${\cal C}$ where $s_0[m_0]$ is already covered, and the other one is in $\R/((1-m_0)\Z)$ where $\bu_0(m_0)$, the special point representing $s_0[m_0]$, behaves as the others. 
	
	Let us now discuss the \underbar{free spaces} evolution as $s_1[t]$ evolves.\\
	\noindent \bls Eaten-space perspective on $\R/((1-m_0)\Z)$,\\
	-- if $\bu_1$ is in $s_0[m_0]$, then the evolution of the length of $s_1[t]$ after time $1$, as well as the lengths of the free spaces, is indistinguishable \underbar{in distribution} from the evolution of $s_0$ after time 1 (if we would let $s_0[m_0]$ resume its eating activity after time 1).\\
	-- if $\bu_1$ is not in $s_0[m_0]$, then, in terms of the lengths of the free spaces,  it is the same! That is, eating the space around $\bu_0(m_0)$ or around $\bu_1$ is the same in law, since, up to a uniform rotation, the surviving points $\bv_0,\ldots \bv_{b-1}$ are distributed as i.i.d. uniform points on $\R/((1-m_0)\Z)$. In particular, $\bv_0 + \bu$ and $\bv_1 + \bu$ are indistinguishable (with $\bu$ uniform on $\R/((1-m_0)\Z)$). 
	
	\noindent \bls \ACDM perspective on $\mathcal{C}$: The idea is the same: when $\bu_1$ becomes active, either it is in $s_0[t]$ or it is not. The free spaces around ${\cal C}$ are distributed as $(1-t) \l[\bD_0,\cdots,\bD_{b-1}\r]$ (under the condition that there are  {$b$} intervals if $\bu_0$ has covered all but $b-1$ other points). If $\bu_1$ is in $s_0[1]$, then in terms of free space, as we said before, only the length of $s_0[1+t]$ matters, not how it grows. Therefore, diffusing around $s_0$ or around $s_1$ is the same, in distribution for the final free space (at time $1+m_1$).
	\medskip
	
	\noindent\textbf{Conclusion }~\\
	Assume that we are back to the initial models of several masses $(\bu_0,m_0),\cdots,(\bu_{k-1},m_{k-1})$  and that a valid \ACDM is given. Assume for a moment that we are only interested in the free spaces $\bF{k}$.
	
	Then, from the  eaten-space  point of view, during the relaxation phase of the $i$th point, the new mass extends around one of the surviving points, that are, conditional on their numbers and up to a uniform rotation, independent uniform points on $\R/(1-(m_0+\cdots+m_{i-2}))\Z$ (and in particular Claim  \ref{claim:Dirichlet} applies to them). The point around which the extension is done is not important, by invariance by relabelling at time $i$: everything would have been the same if all the space was eaten around the same point! Hence \textbf{Theorem \ref{theo:length2}} is a consequence of this fact: if one  is  only interested in free space, we can relocate the location where the relaxation is done, without consequence.
	
	\begin{rem}\label{rem:trans}
		A Corollary of this analysis is that if we   allow the relaxation phases to occur simultaneously, possibly with different speeds, as long as the procedure is invariant under rotation and independent between different occupied CC (before coalescence), then, the distribution of the free and occupied spaces would be the same (for any valid \ACDM).\par
		The universality of the free space distribution would resist to even more involved generalisations, as for example, transfer of non-deposited masses -- between current occupied CC.
	\end{rem}

	As \textbf{Theorem \ref{theo:excha}} also deals with occupied components, we can add some details to conclude. 
	
	Points $(iii)$, $(iv)$ and $(v)$ can be stated as properties of the free CC only. They are consequences of the previous discussion (details for  \textbf{Theorem \ref{theo:excha}$(v)$} are postponed in Section \ref{sec:qefgrethry} below). For example for point $(iii)$, the fact that $\Law(\bN_k-1)={\sf Binomial}(k-1,R_k)$ comes from the fact that we can rearrange the masses as we want when we are only interested in the free CC lengths and number, and in the case in which $m_0=W$ while $m_1=\cdots=m_{k-1}=0$, the result is obvious.

	Clearly, the distribution of the CC of the occupied set $\bO{k}$ depends  on the masses. For example, two masses 0.8 and 0 differ from  two masses 0.6 and 0.2. 
	
	To prove $(i)$, we can add a recursive argument to the previous discussion. First, the position of the CC $s_0[m_0]$ has a distribution independent from $(\alpha_0(t),\beta_0(t))$, and of $(\bu_1,\cdots,\bu_{k-1})$ (its position is uniform on ${\cal C}$). Now, assume that at time $i^-$ (just before the arrival of the $i+1$th mass $(\bu_i,m_i)$), the current state $(\bO{i},\bF{i})$ has a distribution that does not depend on the choice of the valid \ACDM considered. There are two cases: either $\bu_i$ is in one the CC $O\in \bO{i}$, or it is in a free space $F\in \bF{i}$. Then the relaxation process $\IR(t)=(\alpha_i(t),\beta_i(t))$ (with dependence already discussed) will  govern  the growing of $s_i[t]$ around $\bu_i$.
	
	The construction of the occupied and free spaces at time $i$ is independent from the position of the origin 0 on the circle, and this is due to the invariance  of valid models under rotation (condition $(v)$). A dependence arises when one labels  the occupied components (as explained in \eref{eq:qfgezfqd} and in Remark \ref{rem:bias}), but the occupied and free sets at time $i$, $(\bO{i},\bF{i})$, are invariant under rotation by any angle (chosen independently from the current state). The conjunction of events that lead to the position of occupied blocks at time $i$  has a local nature. Each occupied block has essentially grown its own territory. By symmetry, the probability of the conjunction of events that leads to the presence of a unique occupied CC in each of some fixed (deterministic) intervals with respective length $a_1,\cdots,a_b$ separated by some free spaces with lengths $f_1,\cdots,f_b$, is the same regardless of the order of these intervals or the order of the $f_i$. Moreover, as long as the $f_i$ are non-zero, this probability does not depend neither on the precise value of the $(f_i)$, because the valid dispersion algorithms are independent of  free spaces. 
	Thus, the amount of space that $s_i[t]$ will cover before colliding with another CC does not depend on  $(\alpha_i(t),\beta_i(t))$. In particular, the probability that a collision occurs before the end of the diffusion (at time $i+m_i$) is independent of  $(\alpha_i(t),\beta_i(t))$. If a collision occurs, the time of the first collision has a distribution independent of $(\alpha_i(t),\beta_i(t))$. In this case, $s_i[t]$ will coalesce with a CC whose length has a law which does not depend on $(\alpha_i(t),\beta_i(t))$ (it has the same distribution as if it was chosen uniformly among the other CC lengths!). 
	
	The arguments for the proof of $(ii)$ have also been given: the lengths of the occupied CC depend on the masses that are diffused from each vertex, but not on the order of diffusion (see Remark \ref{rem:trans}).
	
	Finally, Theorem \ref{theo:proce} is a consequence of the general exchangeability results obtained (universality of the probability of collision of the connected component that undergoes the relaxation phase during a time interval $\d t$, and in case of collision, universality of the distribution of the length of the occupied CC hit).

	\subsection{Proof of Theorem \ref{theo:excha}$(v)$}
	\label{sec:qefgrethry}

	\noindent\underbar{Markovian property of $(\bN_j)$:} We have 
	\begin{equation}
		\Law\l( \bF{k+1}~|~(\bN_j,0\leq j\leq k)=(n_j,0\leq j \leq k)\r)=\Law\l( \bF{k+1}~|~\bN_{ k}=n_{k}\r). \label{eqn:Markov_Ni}
	\end{equation}	 
	The reason is that conditional on $\bN_{k}$, and also conditional on $(\bN_j,j\leq k)$, the uniformly rotated free spaces $(|\bF{k}_{i+\theta \mod b}|,0\leq i \leq b-1)$ (where $\theta$ is uniform on $\{0,\cdots,b-1\}$ and independent of the $|\bF{k}_i|$) are distributed as $R_k\times \l({\bf D}_1,\cdots,{\bf D}_{b}\r)$ (by Theorem \ref{theo:excha}$(iv)$), 	where the $\l({\bf D}_1,\cdots,{\bf D}_{b}\r)$ are  $\Dirichlet(b;1,\cdots,1)$ distributed. This ensures that (\ref{eqn:Markov_Ni}) is true.  
	Let us now describe its transition matrix of $(\bN_i,i\geq 0)$.
	
	Now, assume that $\bN_k=n$, and $m_0,\cdots, m_{k}$ are given, with $W=W(m[k])=m_0+\cdots+m_{k-1}$, $W'=W(m[k+1])=m_0+\cdots+m_{k}$, and set $R=R_k=1-W$, $R'=R_{k+1}=1-W'$.

	For a Dirichlet vector ${\bD}^{(b)}=(\bD_{1}^{(b)},\cdots,\bD_{b}^{(b)})\sim \Dirichlet(b;1,\cdots,1)$, denote by $\bS^{(b)}_j=\bD_{1}^{(b)}+\cdots+\bD_j^{(b)}$ (so that $\bS^{(b)}_0=0$ and $\bS^{(b)}_b=1$). 
	Observe that, for $x\in [0,1]$,
	\[\P\bigl(\bS^{(b)}_j\leq x\leq \bS^{(b)}_{j+1}\bigr)=\binom{b-1}j x^{j}(1-x)^{b-1-j},\]
	because the Dirichlet random variable with these parameters, is the joint distribution of  $b$ intervals between the order statistics of $\hat{u}_1,\cdots,\hat{u}_{b-1}$ uniform random variables on $[0,1]$. 
	\be\P(\bN_{k+1}= n+1~|~\bN_k=n)&=& R_k\P\bigl( \bS^{(n+1)}_1\leq m_{ {k}}/R_k \bigr).
	\ee 
	Moreover, for $j\in \{0,\cdots,n-1\}$, 
	\be\P(\bN_{k+1}= n-j~|~\bN_k=n)&=& R_k\P\l( \bS^{(n+1)}_{j+1}\leq m_{ {k}}/R_k \leq \bS^{(n+1)}_{j+2} \r)+(1-R_k)\P\l( \bS^{(n)}_{j}\leq m_{ {k}}/R_k \leq \bS^{(n)}_{j+1}\r)\\
	&=& R_k \binom{n}{j+1}\l(\frac{m_k}{R_k}\r)^{j+1}\l(1-\frac{m_k}{R_k}\r)^{n-1-j}\\&&+ (1-R_k)\binom{n-1}{j}\l(\frac{m_k}{R_k}\r)^{j}\l(1-\frac{m_k}{R_k}\r)^{n-1-j}.
	\ee
	The term starting with $R_k$ corresponds to the case $\bu_{k+1} \in \bF{k}$. To have $\bN_{k+1}= n-j$, we need either $\bu_{k+1} \in \bF{k}$, and then that the component growing from $\bu_k$ absorbs $j+1$ blocks among the $n$ other blocks, or to have $\bu_{k+1} \in \bO{k}$, and then that the component growing from $\bu_k$ absorbs $j$ blocks among the $n-1$ other blocks. 
	\begin{rem}  
		We can also directly compute the distribution of $\bN_{k_2}$ conditional on $\bN_{k_1}$ (with $k_1 < k_2$). In fact, one can prove, using the same arguments as for Theorem \ref{theo:length2}, that assuming the hypothesis of Theorem \ref{theo:excha}, $\Law\l (|\bF{k}|~|~ (m_0,\cdots,m_{k_2-1})\r)$ equals 
		\[\Law\l (\l|\overrightarrow{\bF{k}}\r|~|~ (W(m[k_1]),\underbrace{0,\cdots,0}_{k_1-1\textrm{~zeroes}}, W(m[k_2])-W(m[k_1]),\underbrace{0,\cdots,0}_{k_2 - k_1 -1\textrm{~zeroes}})\r).\]
		We let $M_1=W(m[k_1])$, $M_2 = W(m[k_2]) - M_1$, $R_1 = 1 - M_1$ and $R_2 = 1 - M_1 - M_2$. Then, we consider the \ACDM with $m_0 = M_1$, $m_{k_1} = M_2$, and $m_j = 0$ for any $j\in\llbracket1, k_2-1\rrbracket\backslash\{k_1\}$.
		Then, 
		\begin{align*}
			\bN_{k_2} &\eqd 1 + {\sf bin}(\bN_{k_1} - 1 , R_2/R_1) + B(R_2) + {\sf bin}(k_2 - k_1 - 1, R_2)\\
			&\eqd 1 + {\sf bin}(\bN_{k_1} - 1 , R_2/R_1) + {\sf bin}(k_2 - k_1, R_2).
		\end{align*}
		In the first equality, ${\sf bin}(\bN_{k_1} - 1 , R_2/R_1)$ is a binomial r.v.\ that counts the occupied blocks of length 0 at time $k$ that are not absorbed into the newly covered area at time $k_1+1$, after the addition of mass $M_2$. The variable $B(R_2)$ is a Bernoulli r.v.\ that equals 1 if the masses $M_1$ and $M_2$ form a single occupied CC. Finally, ${\sf bin}(k_2 - k_1 - 1, R_2)$ counts the number of 0 masses that arrive between times $k_1 + 2$ and $k_2$ and do not fall into an already occupied CC.
	\end{rem}

	\section{Valid discrete dispersion models}
	\label{sec:DC}
	
	There are many similarities with  continuous dispersion models and we will mainly focus on the differences. 
	%		We willpoint out the differences with continuous dispersion models.
	Here, the masses arrive in the space ${\cal C}_n:=\{0/n,\cdots,(n-1)/n\}\subset {\cal C}$, which is considered as a subgroup of $\R/\Z$. 
	The masses are now positive multiples of $1/n$, and zero masses are still allowed.

	\noindent\textbf{To be valid, a discrete dispersion model}  must satisfy the same conditions as continuous models, except that:\\
	-- The invariance under rotation (condition \eref{eq:dqgeg}) must hold only for the rotation by $1/n$, (in particular the arrival positions are uniform on ${\cal C}_n$ now),\\ 
	-- after each interval relaxation, the occupied set $\bO{k}$ is still formed by closed intervals (some of them having possibly length zero), with extremities in ${\cal C}_n$. 
	
	The notion of occupation remains the same as in the continuous case, and the right diffusion with constant speed in this discrete setting, is also a valid discrete dispersion model.
	
	For example, if $n\geq 8$, if there are four mass arrival events $(m_0,\bu_0)=(3/n,0)$, $(m_1,\bu_1)=(1/n, 5/n), (m_2,\bu_2)=(0,3/n)$, $(m_3,\bu_3)=(0,7/n)$, then, using the RDCS:
	we have $\bO{1}= \big\{[0,3/n]\}$,  $\bO{2}= \{[0,3/n],[5/n,6/n]\}$, $\bO{3}= \bO{2}$, $\bO{4}=\{[0,3/n],[5/n,6/n],\{7/n\}\big\}$.
	
	Again, the addition of a zero mass can leave the occupied set unchanged, or  create the appearance of a point (as stated in Lemma \ref{lem:fgth} and Remark \ref{rem:fgth2}). The distance between two occupied CC is at least $1/n$.
	\begin{rem}
		The number of points in ${\cal C}_n$ in a given CC with extremities in ${\cal C}_n$ is not proportional to the Lebesgue measure of this CC seen as an interval: if it has length $k/n$, it contains $k+1$ points of ${\cal C}_n$.
	\end{rem}

	\subsection{Main universality result for valid discrete dispersion models}

	To state the  analogue of  Theorem \ref{theo:excha} for valid discrete dispersion models (\ADDM), we  define a discrete analogue to the Dirichlet distribution. If the  total free space is  $R/n$ for some positive integer $R$, then there can be at most $R$ free CC, each of them with a positive length multiple of $1/n$.
	
	Let $(k,R)$ be a pair of  positive integers  with $k\leq R$. We call composition of $R$ in $k$ parts a sequence $(s_1,\cdots,s_k)$ with positive integer coordinates, summing to $R$. Let ${\Comp}(R,k)$ be the set of these compositions. We have  
	\beq\label{eq:qsdzer} |\Comp(R,k)|= \binom{R-1}{k-1}\eq
	since we can associate bijectively with each increasing sequence  $(x_1,\cdots,x_{k-1})$ made of elements of $\{1,\cdots,R-1\}$, the sequence $(x_1,x_2-x_1,\cdots,x_{k-1}-x_{k-2}, R-x_{k-1})$. 
	\begin{defi} For all $(k,R)$ integers such that $1\leq k \leq R$, we call Discrete Dirichlet distribution of $R$ in $k$ parts (and write ${\sf DDirichlet}(R,k)$) the uniform distribution on $\Comp(R,k)$. 
	\end{defi}
	The following analogue of Theorem \ref{theo:excha} holds in the discrete setting:
	\begin{theo}\label{theo:excha-D} Let $n\geq 1$ be an integer. Consider a valid \ADDM $A$, and some deterministic masses $m[k]=(m_0,\cdots,m_{k-1})$ all non-negative multiples of $1/n$, satisfying the total weight condition  $W(m[k])<1$. Let $R_k=1- W(m[k])$.
		\bir
		\itr  The law of  $\l(\bO{k},\bF{k}\r)$ is the same for all valid \ADDM $A$.
		\itr  The law of $\l(\bO{k},\bF{k}\r)$ is invariant under the permutation of the masses.
		\itr The number of free CC $\bN_k^{(n)}$ has an explicit distribution $\nu(R_k,k)$ (see Section \ref{sec:DNC}, where an analogue of Theorem \ref{theo:excha} $(v)$ can also be found).
		\itr  Take a uniform element $\theta$ in $\{0,\cdots,b-1\}$ (independently of everything else). 
		Conditional on $\bN_k^{(n)}=b$, the (unbiased) free-space lengths 
		\beq\l(\l|\bF{k}_{i+\theta \mod b}\r|,0\leq i \leq b-1\r)\eqd C[b]:=(C_0,\cdots,C_{b-1})\eq
		where $nC[b]$ follows the ${{\sf DDirichlet}}(nR_k,b)$ distribution.
		\eir
	\end{theo}
	The proof  can be proven using the same  principles as those of Theorem \ref{theo:excha}.
	Once again, the following theorem is an important feature that allows one to perform precise computations: 
	\begin{theo}\label{thm:_freespaces_D}Assuming the hypothesis of Theorem \ref{theo:excha-D},
		the distribution of ${\bF{k}}$ only depends on $k$ and the total mass $W(m[k])$:
		\[\Law\l (\bF{k}~|~ (m_0,\cdots,m_{k-1})\r)=\Law \l(\overrightarrow{\bF{k}}~|~ \bigl(W(m[k],\underbrace{0,\cdots,0}_{k-1~\textrm{zeroes}})\bigr)\r).\]
	\end{theo} 
	
	\paragraph{Discrete piling propensity}
	
	\begin{lem} Let $A$ be a valid \ADDM process on ${\cal C}_n$. If the masses $(m_0,\cdots,m_{k-1})$ are non-negative multiple of $1/n$ and satisfy $W(m[k])<1$, then, the discrete piling propensity is
		\beq\label{eq:PilingD}
		Q^{(n)}(m[k])=\P(\bN_k^{(n)}=1) =
		\l(W(m[k])+\frac1n\r)^{k-1}.
		\eq Moreover, for any $\ell\in{\cal C}_n$, by letting $A_j$ be the starting point of the block $j$, we have,
		\beq\label{eq:shy}\P\l(\bN_k^{(n)}=1,A_0=\ell\r)= Q^{(n)}(m[k])/ n.  \eq
		
	\end{lem}
	\begin{proof} Using the rotation invariance, Formula \eref{eq:shy} and \eref{eq:PilingD} are equivalent. We will prove \eref{eq:shy} which is a bit easier to prove, from a combinatorial perspective. 
		
		We remove the normalisation by $n$, and consider the masses $nm[k]$ instead on $\Z/n\Z$ and also, first on $\Z/(M+1)\Z$ where $M=nW(m[k])$. The RDCS on $\Z/(M+1)\Z$ with these masses leaves a single vertex remains free at the end, and with probability $1/(M+1)$, the last vertex $M$ is free.
		This means that if mass $m_i$ is placed at $p_i$ for $i$ going from 0 to $k-1$ in $\Z/(M+1)\Z$, among the $(M+1)^k$ such map $p$, a fraction $1/(M+1)$ of them leaves the last place empty.
		Hence, the number of ways to arrange the $k$ masses on $\Z/n\Z$ such that, after relaxation (in the discrete RDCS model) the set (of classes) $\ell,\cdots,\ell +M-1$ is occupied in $\Z/n\Z$ is also $(M+1)^{k-1}$. It suffices to multiply by $1/n^k$ which is the probability of each map $p$.
	\end{proof}
	We could have also used Theorem \ref{thm:_freespaces_D} to compute $\P(\bN_k^{(n)}=1)$ in the case when there is  one  mass $W(m[k])$ (the $k-1$ others being zero): it covers a fraction $W(m[k])+1/n$ of $\mathcal{C}_n$, and $\bN_k^{(n)}=1$ if and only if every zero mass falls on one of these vertices. This happens with probability $(W(m[k])+1/n)^{k-1}$, since there are $k-1$ such masses.   
	
	There exists also an analogous of  Theorem \ref{theo:full_dist} and Theorem \ref{theo:dqgsrt}: 
	\begin{theo}\label{theo:full_dist-D} Consider some masses $m_0,\cdots,m_{k-1}$, non-negative multiple of $1/n$, such that $W(m[k])<1$.
		For any $b \in\{1,\cdots,k\}$, any partition $(J_0,\cdots,J_{b-1})$ of the set $\{0,\cdots,k-1\}$ with non-empty parts, any sequence $(a_0,a_1,\cdots,a_{b-1})$ in ${\cal C}_n^b$ such that the $a_0,a_1,\cdots,a_{b-1}$ are cyclically ordered around ${\cal C}_n$ (i.e. such that turning around the circle we get $a_0\preceq 0\preceq a_1\preceq \cdots \preceq a_{b-1}\preceq a_0 \preceq 0 \cdots$) and such that
		\[ \Leb\l(\fleche{[a_i,a_{i+1 \mod b}]}\r)> W(m(J_i))~~\textrm{ for }i\in\{0,\cdots,b-1\}.\]
		For a valid \ADDM,  we have
		\beq
		\P\l( \bA{k}_j= a_j, \bI{k}_j =J_j,0\leq j \leq b-1\r)=  \prod_{j=0}^{b-1} \frac{Q^{(n)}(m(J_j))}n.
		\eq
		Moreover, if we consider some non-negative masses $M_0,\cdots,M_{b-1}$,
		\beq
		\P\l( \bA{k}_j = a_j,  \l|\bO{k}_j\r|=M_j, 0\leq j \leq b-1\r)= \sum_{(J_0,\cdots,J_{b-1}) \in P(k,b)} \prod_{j=0}^{b-1} \frac{Q^{(n)}(m(J_j))}n \ind{W(m(J_j)) = M_j}
		\eq
		where, as in Theorem \ref{theo:dqgsrt}, $P(k,b)$ is the set of partitions of $\{0,\ldots, k-1\}$ into $b$ parts.
	\end{theo} 
	We have in the discrete case the analogue of Theorem \ref{theo:proce}:
	\begin{theo}\label{theo:proce-D} Let $(m_0,m_1,\cdots,m_{k-1})$ be masses multiple of $1/n$ with sum in $[0,1)$. Then, the distribution of the (time indexed) process $(S(t),t\in\{0,\cdots,k\})$, recording the multiset of occupied CC cardinalities, defined by $S(t):=\l\{\l\{\l|\bO{t}_m\r|, 0\leq m\leq \bN_j-1\r\}\r\}$  is the same for all valid \ADDM.
	\end{theo}
	
	\begin{rem}[Reduction to the Right Diffusion with Constant Speed and excursion sizes]
		The same phenomenon occurs here. The law of $(\bO{k},\bF{k})$, for a fixed $k$ is the same for all valid \ADDM, and then, can be computed for the simplest model.
	\end{rem}

	\subsection{Distribution of the number of CC $\bN_k^{(n)}$.}
	\label{sec:DNC}
	
	The distribution of $\bN_k^{(n)}$ is more complicated than in the continuous case, when  it was just   binomial. 
	Consider $\bu_0,\cdots,\bu_{k-1}$ some random variables i.i.d. uniform on ${\cal C}_n$ and consider the set
	\[ Z_k:=\{\bu_0,\cdots,\bu_{k-1}\}\]
	of arrival places.
	Note that when all the masses $m_0,\cdots,m_{k-1}$ are equal to zero, $\bO{k}=Z_k$.  
	The support of the random variable  $|Z_k|$ is $\{1,\cdots, \max\{k,n\}\}$ and the distribution of the set $Z_k$ is computable, since for $z$ a subset of ${\cal C}_n$ with cardinality $|z|>0$ 
	\[\P(Z_k=z) ={S(k,|z|)}{n^{-k}}\]
	where  $S(a,b) := \sum_{j=0}^b (-1)^{b-j}\binom{b}{j}j^a$ is the number of surjections from $\{1,\ldots,a\}$ to $\{1,\ldots,b\}$.
	Hence, the law of $|Z_k|$ is also explicit
	\[\P(|Z_k|=m)= \binom{n}{m}{S(k,m)}{n^{-k}},\] and conditional on $\{|Z_k|=m\}$ (for $m \in \{1,\cdots,\max\{k,n\}$), $Z_k$ is a uniform subset of ${\cal C}_n$ with cardinality $m$.
	
	We use Theorem \ref{thm:_freespaces_D}, in order to describe $\bN_k^{(n)}$ when the masses $m_0,\cdots,m_{k-1}$ have been treated with a valid \ADDM: it suffices to diffuse the mass $W(m[k])$ from $\bu_0$ and keep the other elements of $Z_k$ unchanged (and observe how many of them are not covered). The diffusion from $\bu_0$ will reach $\bu_0$ and the $M:=nW(m[k])$ next points of ${\cal C}_n$ on its right, so that
	\[\bN_k^{(n)}=1 + | Z_k \backslash \fleche{[\bu_0,\bu_0+W(m[k])]}|.\]
	and then, a random subset $V$ of $Z_k\setminus \{\bu_0\}$ with length
	$c$ remains free (not covered) with probability 
	\beq\P(\bN_k^{(n)}=1+c)= \sum_{m=1+c}^{k} \P(|Z_k|=m) \frac{\binom{n-1-M}{c}\binom{M}{m-1-c}}{\binom{n-1}{m-1}}\eq
	where, here, the classical factor $\frac{\binom{n-1-M}{c}\binom{M}{m-1-c}}{\binom{n-1}{m-1}}$ is the probability that, considering a subset of $m-1$ elements in a set of cardinality $n-1$, exactly $c$ of them fall into a prescribed subset of length $n-1-M$.\\

	We now consider the process $(\bN_k^{(n)}, k\geq 0)$. It is a Markov chain for the same reason as in the continuous case. Let us describe its transition matrix.
	Given that $\bN_k^{(n)}=b$, the occupied domain is constituted with $b$  components, with total mass $W(m[k])$ and then contains a proportion 
	\[q_b:=(nW(m[k])+b)/n=W(m[k])+b/n\] of the elements of ${\cal C}_n$; let $p_b:=1-q_b$ denote the proportion of ${\cal C}_n$ that is still free. 
	
	We now add a mass arrival event $(m_{k},\bu_k)$, and will proceed to the RDCS of this mass from $\bu_k$. Initially, $\bu_k$ belongs to $\bO{k}$ with probability $q_b$, and to $\bF{k}$ with probability $p_b$.
	Hence, just after the insertion of the point $\bu_k$, we have $ b+X$ free CC (possibly reduced to a single point), where $X$ is a ${\sf Bernoulli}(p_b)$ random variable. The free spaces,  up to a uniform rotation of indices (to disregard the influence of the labelling choice), after the insertion of the point $\bu_k$,  form a uniform partition of $R:=n(1-W(m[k]))-b$, each free spaces having length at least $1/n$  (and multiple of $1/n$). Hence, conditional on $X=x$, these free spaces form a ${\sf DDirichlet}(R,b+x)$-distributed vector.

	We let
	$(\bD_1^{\ell,R},\cdots,\bD_\ell^{\ell,R})$ be a ${\sf DDirichlet}(R, {\ell} )$ random variable. Denote again $\bS_{j}^{\ell,R}= \bD_1^{\ell,R}+\cdots+\bD_j^{\ell,R}$. We have
	\begin{equation}\P(\bN_{k+1}=b+1~|~\bN_{k}=b)= p_b \P( { nm_k< \bS_1^{b+1,R} })\end{equation}
	and 
	\begin{equation}\P(\bN_{k+1}=b-j~|~\bN_{k}=b)=p_b \P(\bS_{j+1}^{b+1,R}\leq nm_k <   \bS_{j+2}^{b+1,R})+q_b\P(\bS_j^{b,R}\leq nm_k< \bS_{j+1}^{b,R}),\end{equation}
	and the computation can be done since 
	\[\P(S_j^{\ell,R} \leq x<S_{j+1}^{\ell,R})= {\displaystyle\binom{x}{ {j}}\binom{R-1-x}{ {\ell-1 -j}}}~/~{\displaystyle\binom{R-1}{\ell-1}}\]

	\subsection{Examples of \ADDM{s}.}
	\label{sec:examples} 
	\begin{itemize}
		\item \textbf{Parking model (or hashing with linear probing):} when masses all have weight $1/n$, if the dispersion policy is  the RDCS  then the model is exactly the classical parking model (introduced in Konheim and Weiss \cite{Konheim1966AnOD} and studied for example by Chassaing and Louchard \cite{chassaing2002phase}; see also \cite{MR1814521}). We will come back to this model in Section \ref{sec:park}; the cost analysis of the hashing with linear probing (in terms of total car displacement from their chosen place to their final park place), due the Flajolet \& al. \cite{flajolet1997Analysis}, and also to \cite{chassaing2002phase} are discussed in Section \ref{sec:CPC}. 
		However, our analysis shows that the distribution of the occupied CC of the parking (at any time $k$) does not depend on the valid \ACDM used. Janson \cite{MR2190121} noticed the same property for three different policies of parking in which a car may eject an already parked car. When ejected, a car searches a place at its right. For these three policies, if a car arrives in a block $[a,b]$, after insertion, the new block becomes $[a,b+1/n]$ (possibly merged with the next block).
		\item \textbf{Particle dispersion}
		The masses $m_i$ have weights proportional to $1/n$, and are composed by particles of size $1/n$. The particles perform independent random walks successively  until they exit  the occupied interval in which they are. 
		
		\item \textbf{Discrete caravan type model:} The masses are proportional to $1/n$ and undergo the RDCS so that they progressively fill in the empty slots at the right of the arrival position. 
	\end{itemize}

	\section{Asymptotics}
	\label{sec:LOT}

	\subsection{Discussion regarding random masses results}
	\label{sec:RMR}
	
	Theorems \ref{theo:excha} and Theorem \ref{theo:length2}  hold for any valid \ACDM, for any fixed masses $(m_0,\cdots,m_{k-1})$, as long as $\sum m_i <1$.   
	Here, we assume that the masses $\bm[k]=({\bf m}_0,\cdots,{\bf m}_{k-1})$  are random and taken according to a  distribution $\mu^{(k)}$ on $[0,1)^k$, such that
	\beq\label{eq:condmasse} \P\big( W(\bm[k])  <1\big)=1.\eq 
	Again, take some i.i.d. arrival positions $({\bf u}_0,\cdots,{\bf u}_{k-1})$, independent from the masses. 
	
	Since conditional on the event $\{ {\bf m}_i=m_i, 0\leq i \leq k-1\}$, the results of Theorems \ref{theo:excha} hold, in particular, $\bN_k \eqd 1+ B$  where $\Law(B~|~ W(\bm[k])=W)={\sf Binomial}(k-1,1-W )$.
	The occupied and free block results of Theorem \ref{theo:dqgsrt} can be extended to random block length by integration (with respect to $\mu^{(k)}$).  
	However, this general statement is insufficient for understanding or computing the typical behaviour of the block lengths for a given distribution $\mu^{(k)}$ when $k$ becomes large. Important ``partially tractable'' examples correspond to the case for which the combinatorial terms can be understood. Other important examples correspond to the case where the right diffusion with constant speed representation allows to bypass exact representations.
	
	However, we will need additional  hypotheses on masses distribution to design limit theorems. We will come back on particular cases in Section \ref{sec:RWRL}.\medskip
	
	\color{black}
	\subsection{Discrete masses: comparison between continuous and discrete spaces}
	\label{sec:park}
	
	The aim of this section is to consider the cases where \DDM{s} and \CDM{s} meet: when the masses are multiples of $1/n$. In this case, it is possible to define the right diffusion with constant speed on ${\cal C}$ and ${\cal C}_n$ respectively, and compare the statistics of the space occupation in both cases. 
	As we will see, \DDM{} and \CDM{} defined using the same masses have some similarities, and some discrepancies.
	
	One interesting case is the continuous model of parking, that we will call the lazy bulldozer parking model.
	
	In the usual parking problem defined on the discrete cycle ${\cal C}_n$ (see \cite{chassaing2002phase} or Section \ref{sec:CPC}), some cars (with length $1/n$) arrive successively, and the $i$th car chooses a uniform place $c_i$ in ${\cal C}_n$. It then parks at the first available place among $c_i,c_i +1/n\mod 1, c_i+2/n \mod 1$, etc (the place where it parks is then not available for the subsequent cars). 
	In terms of \ADDM, it occupies the interval $\fleche{[\ell,\ell+1/n]}$ where $\ell$ is the first available place discovered.

	The \textbf{lazy bulldozer parking model} is defined on the continuous parking ${\cal C}$. The bulldozers are each $1/n$ in   length $1/n$ and choose uniform places $(\bu_i)$ on ${\cal C}$ (where the $\bu_i$ are i.i.d.). The $i$th bulldozer arrives when bulldozers 0 to $i-1$ are already parked. It observes the parking and locates $y$, the first free point at the right of $\bu_i$ (and $y$ may belong to a free space smaller than $1/n$, as on Figure \ref{fig:Bulldozer}). Bulldozer then parks on $\fleche{[y,y+1/n]}$, pushing to the right, if needed, the bulldozers already on $\fleche{[y,y+1/n]}$ if any, as on Figure \ref{fig:Bulldozer}. More precisely, if the  interval $\fleche{[y,y+1/n]}$ intersects an occupied block $O$, then, this component is translated on its right (and then $O$ will possibly  push the next occupied component, and this recursively, if needed), so that the length of $O$ is preserved, and after this action by the bulldozer, $O$ is at the right of $\fleche{[y,y+1/n]}$ and adjacent to it.
	
	In terms of occupied space, the lazy bulldozer parking model corresponds to the RDCS (see Section \ref{sec:RDCS}).
	
	\begin{figure}[htbp]
		\centerline{\includegraphics{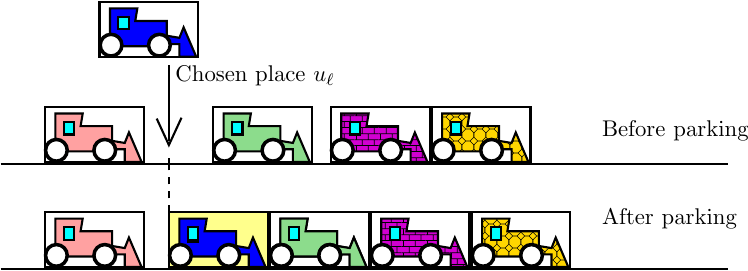}}
		\captionn{\label{fig:Bulldozer} The bulldozer parks in the first free space after the chosen location and pushes what is on the way to make the room it needs. If the chosen place $\bu_\ell$ is already occupied, the bulldozer searches for the first free point $y$ at the right of $\bu_\ell$ and parks on $\fleche{[y,y+1/n]}$, pushing what is needed to clear the place it wants. }
	\end{figure}
	\noindent-- Consider the number $\bN_k$ of blocks in the (continuous state) bulldozer model, and let $\bO{k},\bF{k}$ denote the state at time $k$ (with $k$ cars of length $1/n$ treated).\\ 
	-- Consider the number $\bN_k^{(n)}$ of blocks in the (discrete state) parking model, and denote by $\bO{n,k},\bF{n,k}$ the state at time $k$, with an additional exponent $n$ (with $k$ cars of length $1/n$ treated).
	
	By \eref{eq:Q} and \eref{eq:PilingD}, if all the $m_i$ are multiple of $1/n$, the discrete  and  continuous piling propensities are very similar:
	\[Q(m[n] )=\P\l(\bN_k = 1\r)= W(m[k])^{k-1} \textrm{~~and~~}Q^{(n)}(m[n] )=\P\l(\bN_k^{(n)} =1 \r)= \l(1/n+ W(m[k])\r)^{k-1}.\]
	This similarity, as well as the resemblance between the block lengths in Theorems \ref{theo:dqgsrt} and \ref{theo:full_dist-D}, may give the intuition that discrete and continuous diffusion processes should behave very similarly... However, as we will see, this is only partially true. 
	
	One first  discrepancy is that the  support of $\bN_k$ and $\bN_k^{(n)}$ are different for large $k$. Indeed, if $n-k$ cars (masses) have parked (in the parking problem),  the remaining space has Lebesgue measure $k/n$.
	Since every free CC has length at least $1/n$, there can be at most $k$ occupied CC in $\bO{n,n-k}$: hence the support of $\bN_{n-k}^{(n)}$ is $\cro{1, \min\{k,n-k\}}$ (the number of CC is also smaller than the number of cars $n-k$).  
	
	In the bulldozer case, the free CC may be arbitrary small, so there is a positive probability that every bulldozer is isolated,  so the support of $\bN_{n-k}$ is easily seen to coincide with $\cro{1, n-k}$.\medskip
	
	\noindent We will give more detailed comparison results in the remainder of this section. The two models exhibit significant similarities and discrepancies depending on the observed statistics.\\
	The similarities concern large occupied CC, while the discrepancies concern small occupied CC and the number of occupied CC.
	
	While comparing the two models, we focus on the phase transition that occurs at time  
	\[t_n(\lambda):=\floor{n-\lambda\sqrt{n}},\] for some $\lambda >0$, at which $t_n$ vehicles/masses have been treated (this phase transition has been identified in \cite{chassaing2002phase}).  For any $k\geq 1$, and any integer $t\in \{0,\cdots,n\}$, set
	\be
	L^{(n),\da}[t,k]&:=&(L_i^{(n)},1\leq i \leq k),\\
	L^{\da}[t,k]&:=&(L_i,1\leq i \leq k)
	\ee 
	respectively, the Lebesgue measures of the  $k$ largest occupied CC of $\bO{n,t}$ in the discrete parking and of  $\bO{t}$, in the continuous parking, both lists being sorted in decreasing order, and completed by zeroes, if there are less than $k$ CC.		
	\begin{pro}\label{pro:disc}
		In the discrete parking,
		\beq\label{eq:NnAs} 
		{\bN_{t_n}^{(n)}}/{\sqrt{n}} \proba \lambda (1-e^{-1})\eq
		while, in the continuous parking
		\beq\label{eq:NAs}  {\bN_{t_n} }/{\sqrt{n}} \proba \lambda.\eq
		For any $k$, the following convergence in distribution holds in $\R^k$,
		\ben \label{eq:grhte1}
		n^{-1}\,L^{(n), \da}[t_n(\lambda),k] &\dd&(\ell_i^{\da}(\lambda,\se),1\leq i \leq k)\\
		\label{eq:grhte2} n^{-1}\,L^{\da}[t_n(\lambda),k] &\dd&(\ell_i^{\da}(\lambda,\se),1\leq i \leq k)
		\een
		where $(\ell_i^{\da}(\lambda,\se),1\leq i \leq k)$ is the (random) vector formed by the $k$ largest excursion sizes of the process $t\mapsto (\se_t -\lambda_t)-\min_{s\leq t} (\se_s -\lambda s)$, sorted in descending order (where $\se$ is a normalised Brownian excursion).  
	\end{pro}
	
	\begin{rem} In Bertoin \cite[Lemma 5.11]{MR2253162}, the limiting fragment sizes sorted in decreasing order are described as the ranked atoms $(a_1,a_2,\cdots)$ of a Poisson point measure on $(0,+\infty)$ with intensity $\frac{\lambda}{\sqrt{2\pi a^3}} da$, conditioned by $\sum_{i\geq 1}a_i=1$. In Aldous \& Pitman \cite[Theo.4, Cor. 5]{MR1675063}, Chassaing \& Louchard \cite[Theo. 1.5]{chassaing2002phase}, the description of the law of the blocks when sorted according to a size bias order are given. 
	\end{rem}

	\begin{rem} Since there can only be a finite number of blocks of length at least $x n$ for any given $x$, the convergence stated in Proposition \ref{pro:disc} is equivalent to the vague convergence on $(0,1]$ of the measure $\sum_{O \in \bO{(k)}} \delta_{\Leb(O)/n}$ to the sum of the Dirac masses at the excursion lengths of $(\se_t -\lambda_t)-\min_{s\leq t} (\se_s -\lambda s)$. 
	\end{rem}
	\begin{rem}The discrepancy in the distribution of the occupied domains between discrete and continuous cases stems from the combined effects of the piling propensity and the ``spacing propensity'' which is a measure of the set of non-intersecting positioning of CC around the circle:\\
		\bls in the continuous case, the measure of the set of the possible distances between $b$ occupied blocks when the remaining space is $R$, is \[S(b,R) =R^{b-1}/(b-1),\] independently from the masses, since this is the measure of the simplex $\{(x_1,\cdots,x_b) \in [0,1]^b, \sum x_i=R\}$,\\
		\bls in the discrete case, the number of ways of sharing $k$ (unit) remaining spaces (each of them being of Lebesgue measure $1/n$), into $b$ non-empty (integer) blocks is (as explained in \eref{eq:qsdzer}) $S^{n}(b, R )=|\Comp(nR,b)|= \binom{nR-1}{b-1}$.  
		Of course, the two formulas of $S(b,R)$ and $S^n(b,R)$ have to be included in a larger picture -- this is the goal of Theorems \ref{theo:dqgsrt} and \ref{theo:full_dist-D} -- but the fact that these quantities do not have the same support, and are not (close to) ``proportional''  is the main message here.
	\end{rem}
	\begin{proof}[Proof of Proposition \ref{pro:disc}]
		Proof of \eref{eq:NAs}:
		We know that $\bN_{t_n(\lambda)}-1$ follows the binomial distribution with parameter $(t_n(\lambda)-1,1- t_n(\lambda)/n)=(t_n(\lambda)-1, \lambda/\sqrt{n})$. Hence  $\bN_{t_n(\lambda)}/\sqrt{n}$ has mean $\lambda+o(1)$ and variance $\lambda O(1/\sqrt{n})$, so that \eref{eq:NAs} follows by Bienaym\'e-Chebyshev inequality.

		Proof of \eref{eq:NnAs}: By Theorem \ref{thm:_freespaces_D}, the number of occupied CC would be the same if we would place a ``big'' mass equals to $t_n(\lambda) /n=(1-\lambda/\sqrt{n})$ (which by invariance under rotation can be placed wherever we want, say on $[0,t_n(\lambda) /n]$ , and $(t_n(\lambda)-1)$ masses  $0$. In this case $\bN_{t_n}^{(n)}$ would be 1 (big block) plus  $B_n$, the number of sites in ${\cal S}_n:=[t_n(\lambda)/n+1/n,\cdots,(n-1)/n]\cap {\cal C}_n$, that receive at least one mass.
		A given site in ${\cal S}_n$ receives a binomial $B(t_n(\lambda), 1/n)$ number of zero masses (this converges to a Poisson$(1)$ r.v., in distribution).From this, we easily   deduce  that since $|{\cal S}_n|\sim \lambda\sqrt{n}$, $\E( {B_n} )\sim \lambda \sqrt{{n}}(1-\exp(-1))$, and  $B_n$  is approximately a Poisson distributed r.v.. To obtain \eref{eq:NnAs}, a concentration argument relying on a second moment method suffices.
		
		We present just some elements on \eref{eq:grhte1}, \eref{eq:grhte2}, for which variants exist in the literature.  First, to prove the asymptotic proximity between both largest occupied CC, the idea is to construct, on the probability space on which the continuous space parking is defined, with the data $((1/n, \bU_i),0\leq i \leq t_n(\lambda))$ of mass arrival events, a discrete counterpart, by taking 
		\beq\label{eq:bU} \bU^{(n)}_i:= \ceil{n \bU_i}/n,\eq
		as location of the $i$th arrival. Hence $ \bU^{(n)}_i$ is a bit on the right of the continuous one $\bU_i$, but, on interval of the form $(k/n,(k+1)/n]$ for some $k$, the numbers of arrivals in the discrete and continuous parking coincide perfectly. We construct the collecting paths  (already seen in \eref{eq:fegee}) associated with both parking models.
		Set, for two discrete time $t$, 
		\ben\label{eq:qfgreht}
		S^{(n)}(t,x)&=& -x + \sum_{i\leq t} \1_{\bU^{(n)}_i \leq x} /n, \\
		S(t,x)&=& -x + \sum_{i\leq t} \1_{\bU_i \leq x} /n.
		\een 
		The term $-x$ represents the deposition on $[0,x]$ while  $\sum_{i\leq t} \1_{\bU_i \leq x} /n$, encodes the total mass that has arrived on the same interval. The erosion term contributes  even when there is no mass to erode. This method has the advantage of making the occupied components appear under the form of excursions above the current minimum.
		These collecting paths\footnote{The continuous version $S(t,x)$ has been represented on Figure \ref{fig:RDCS3}} encode the occupation of the parking (the idea of this coupling is borrowed to Chassaing \& Marckert \cite{MR1814521} and can be found also in Bertoin \& Miermont \cite[Section 5.]{bm2006}). 
		Now, there are essentially two methods to conclude:\\
		$\bullet$  			either we prove that $\sqrt{n}S^{(n)}(t_n(\lambda),~.~)$ converges in distribution in $C[0,1]$, up to a rotation that we define below, to $e_\lambda$ defined as $e_\lambda : t \mapsto e(t) - \lambda t$, where $e$ is a standard Brownian excursion; and we recover the lengths of the occupied CC by taking into account Figure \ref{fig:RDCS3}: in words, we need to define a rotated version of $S^{(n)}$ with respect to $a'_n=\min\argmin(t\mapsto S^{(n)}(t_n(\lambda),t) - S^{(n)}(t_n(\lambda),1) t )$ (which amounts in Figure \ref{fig:RDCS3} to put the new origin at $(a'_n,S^{(n)}(t_n(\lambda),a'_n))$ and consider in this new basis, the obtained length 1 trajectory. Technically, this corresponds to set $\bar{S}^{(n)}(t_n(\lambda),~.~)={\sf Rot}(a'_n,{S}^{(n)}(t_n(\lambda),~.~))$
		where the rotation is defined more generally, for all function $f$ and $a\in[0,1]$ by
		\[ 
		{\sf Rot}(a,f)(x) = \begin{cases} f(a+x)-(\min f) & \textrm{ for }x\in[0,1-a],\\
			f(1)-(\min f )+f(x-(1-a)) &\textrm{ for }x\in[1-a,1].
		\end{cases}\]
		
		Then to complete the proof by the Skorokhod representation theorem, it is possible to find a probability space, containing copies of  ${S}^{(n)}(t_n(\lambda),~.~)$ and $e_\lambda$ (still denoted by the same names in the sequel), such that  $\sqrt{n}\bar{S}^{(n)}\as e_\lambda$, and since $e_\lambda$ reaches its minimum a.s. once (say in $a'$) and one has also on this space $a'_n\as a'$, then $\sqrt{n}\bar{S}^{(n)}(t_n(\lambda),~.~)\dd {\sf Rot}(a',e_\lambda)$. Now,  
		the convergence of the excursion sizes above the current minimum toward those of the limit is due to Aldous \cite[Lemma 7]{Aldous1997}.
		
		$\bullet$ Or we may use the method of \cite{chassaing2002phase} which consists in working on the complete parking process conditioned on having the last vertex free when $n-1$ cars are parked. This has no incidence on the distribution of the occupied CC lengths at time $t_n(\lambda)$. In this case the collecting path $S^{(n)}(n-1,~.~)$ converges in distribution to the Brownian excursion $\se$ and,
		\beq
		\l(\sqrt{n}\,S^{(n)}(t_n(\lambda),x),0\leq x \leq 1\r)\dd (\se(x)-\lambda x,0\leq x \leq 1)
		\eq
		where the convergence holds in $C[0,1]$, and the convergence holds also jointly,  for a finite number of $\lambda_j$ (the proof is \cite{chassaing2002phase}, but other proofs can be given, e.g. using the argument in \cite{MR2386089}, and the convergence as a function of $(t,x)$ has also been given in \cite{MR3547745}, as recalled further in the paper, in \eref{eq:qget}). From here, again by Aldous \cite[Lemma 7]{Aldous1997}, we get \eref{eq:grhte1}.

		Now, it remains to explain why in continuous parking version, \eref{eq:grhte2} holds, that is, why the results are the same in discrete and continuous parking, regarding large excursion sizes.
		We have, for a fixed $\lambda$,
		\beq\label{eq:joincv} \sqrt{n}\l(S^{(n)}(t_n(\lambda),~.~),S(t_n(\lambda),~.~)\r) \dd (b_\lambda,b_\lambda) \eq
		for the Skorokhod topology on $D([0,1],\R^2)$. 
		We have already explained the convergence of the first marginal, so it suffices to explain why 
		the uniform distance between the two processes goes to zero in probability. This has already been proven in Chassaing \& Marckert \cite{MR1814521} and is simple: the two processes $(S(t,x),x\in[0,1])$ and $(S^{(n)}(t,x),x\in[0,1])$ coincide at the points $x=k/n$ (with $k$ integer). Let $R(k,t)= \{i: i \leq t, U_i\in (k/n,(k+1)/n]\}$ be the binomial distributed number of cars arrived in $(k/n,(k+1)/n]$ at time $t$, and clearly
		\[\sup_{x\in[0,1]} \l| S^{(n)}(t,x)-S(t,x)\r|\leq \max_{0\leq k\leq n-1} R(k,t)/n\leq  \max_{0\leq k\leq n-1} R(k,n)/n,\] 
		which is then the corresponding value in the completely full parking.
		Now, it is well known that there exists a constant $C>0$ such that 
		\[\P\l(\max_{0\leq k\leq n-1} R(k,n) \geq C\log n\r)\leq n\P\l(  R(1,n) \geq C\log n\r)\leq 1/n\]
		for $C$ large enough.

		Again,  Aldous \cite[Lemma 7]{Aldous1997} allows us to deduce the convergence of the second marginal in \eref{eq:joincv}.
		
	\end{proof}
	
	\subsubsection{Extension to more general discrete masses model}
	
	The general message here, is that the collecting paths are the main tools to study the behaviour of any dispersion models, for any model of masses $(\bm_i,0\leq i \leq t_n(\lambda))$, where the $\bm_i$ are multiple of $1/n$ (the $\bm_i$ being random or not), taken at some time $t_n(\lambda)$, random or not\footnote{when the masses are taken as i.i.d. random variables, stopping the diffusion process just before the space overflow, that is, at time $\tau-1$ with $\tau=\inf\{t: m_0+\cdots+m_{t-1}> 1\}$,  provides a random time as which it is natural to stop the diffusion process.}.% Again   the collecting paths $S^{(n)}(t,~.~)$ and $S(t,~.~)$ associated with the discrete and continuous models as defined in \eref{eq:qfgreht} are some tool to consider first.

	For all models for which the range of $S^{(n)}$ is typically of order larger than $1/n$ (globally), the large excursions of $S^{(n)}(t,~.~)$ and $S(t,~.~)$ above their current minima will be asymptotically identical. This can be proven using the arguments in the proof of Proposition \ref{pro:disc} together with Aldous \cite[Lemma 7]{Aldous1997} and Bertoin \cite[Lemma4]{MR1825153} (see also Bertoin \& Miermont \cite[Proof of Prop. 1]{bm2006}):
	\begin{theo}\label{theo:genen} Consider a sequence of \ADDM (or $\ACDM$) such that, for some sequence of times $(t_n,n\geq 0)$,   
		\[
		\l(\alpha_n S^{(n)}(t_n,x),x\in [0,1]\r)\dd \l(S^{\infty}(x),x \in [0,1]\r) 
		\]
		where $\alpha_n/n \to 0$ and $\alpha_n\to+\infty$, with $S^\infty(1)<0$, and where the convergence holds in $C[0,1]$ (or in $D[0,1]$).  
		If $S^\infty$ reaches its minimum a.s. once on $[0,1]$ at some  $a\in[0,1]$, and if moreover $S$ has a.s.  the property that its excursions above the current minimum are strict (in the sense that there are no strict non-trivial sub-excursions of any excursion at the same level), then for $S={\sf Rot}(a,S^\infty)$, the sequence formed by the $k$ largest clusters   converges  in distribution to the length of the excursions of $S$ above its current minimum 
		\be n^{-1}\,L^{(n), \da}[t_n,k] &\dd&(\ell_i^{\da}(S),1\leq i \leq k)\\
		n^{-1}\,L^{\da}[t_n,k] &\dd&(\ell_i^{\da}(S),1\leq i \leq k)
		\ee
	\end{theo} 
	\noindent$\bullet$ We refer to Bertoin \cite{MR1825153} for the right way of defining excursions in the case where $S$ has some jumps,\\		
	$\bullet$ We could have stated convergence for the vague topology instead (on $(0,1]$).

	Both trajectories $S^{(n)}(t,~.~)$ and $S(t,~.~)$ coincides at the discrete points $(k/n,0\leq k\leq n)$. Moreover the evolution of the $-x$ term is $-1/n$ on such an interval.
	A small figure is sufficient to see that, if $\alpha_n S^{(n)}(t,~.~)$ converges in distribution to a limit in $D[0,1]$, then, if $\alpha_n =o(n)$, the largest excursion lengths of $S^{(n)}$ and of $S$ above their current minimum should be close asymptotically. The condition  $\alpha_n =o(n)$ is necessary, because if $S^{(n)}$ has order $1/n$, then the block lengths are expected to be very small (of order $1/n$), and the asymptotic behaviour of $\alpha_n S^{(n)}(t,~.~)$ would not be an appropriate tool for studying them.
	As exemplified in the parking case, statistics concerning small occupied CC are expected to differ  in \ADDM and \ACDM. 
	
	\subsection{Random models (and caravans)}
	\label{sec:RWRL}
	By Theorems \ref{theo:excha}, \ref{theo:full_dist}, \ref{theo:proce}, \ref{theo:excha-D}, \ref{theo:full_dist-D} and \ref{theo:proce-D}, the study of occupied CC of all valid continuous or discrete diffusion processes reduces to that of right diffusion with constant speed, which are called ``caravan'' in Bertoin \& Miermont \cite{bm2006} (up to the open/closed representation of the occupied space, and the processing of zero masses, see Remark \ref{rem:caravan}). 
	
	\subsubsection{Bertoin \& Miermont results about caravans}\label{subsect:caravans}
	To study the asymptotic behaviour of valid \ACDM or \ADDM when the number of masses goes to $+\infty$, we need to make some hypotheses about the masses. The case in which all masses are the same size corresponds to parking \footnote{even if the size of the cars is $`e>0$ and is not of the form of $1/n$, one may take a discretization of $\floor{1/`e}$ parking places with an additional ``incomplete place'' which can clearly be ignored in the asymptotic range of Section \ref{sec:park}} and was discussed in Section \ref{sec:park}.
	
	When the number $t_n$ of masses goes to $+\infty$ (with sum in $[0,1)$) and the masses are random, the collecting path $W(t,~.~)$ will have a limiting distribution under regularity assumptions on the masses (for example, the fact that they are i.i.d., have some moments, or a regular tail).
	
	Bertoin \& Miermont \cite{bm2006} consider masses $(m_i,i\geq 0)$ that come from some normalised i.i.d. random variables $(l_i,i\geq 0)$ with finite expectation $\E(\ell)=\mu_1<+\infty$ and either finite second moment $(\mu_2=\E(\ell^2)$, in this case, $\ell$ is said to be in ${\cal D}_2$), or have a regular tail:
	\[ \P(\ell>x)\sous{\sim}{x\to+\infty} cx^{-\alpha}\] for some $c\in(0,+\infty)$ and $\alpha \in (1,2)$. In these cases, $\ell$ is in the domain of attraction of spectrally positive stable distribution with index $\alpha$, and said to be in ${\cal D}_\alpha$).
	Then, they consider for (small) $`e>0$,
	\[ \tau_{1/`e}=\inf\l\{i: \ell_0+\cdots+\ell_i>1/`e\r\}\]
	and they work on a circle of length $1/`e$, with these masses, but in terms of the present paper, this amounts to taking $k= 1+ \tau_{1/`e}$ masses of size
	\[m_i=`e\ell_i, \textrm{ for }i<\tau_{1/`e},\]
	and with the last mass taken to complete ``1'', that is $m_{\tau_{1/`e}}=1 -W(m[\tau_{1/`e}])$.
	They then prove the convergence of the associated collecting path  $((1/`e)^{1/\alpha} S_t,0\leq t \leq 1)$  toward an analogue of the Brownian bridge with a linear drift for stable processes, the standard stable loop. See the definition in Bertoin \& Miermont \cite{bm2006}, in Formula $(2)$ and the surrounding text. From this, they deduce the finite-dimensional convergence in distribution, of the lengths of the largest occupied CC in their Theorem 1, in a critical time window defined in terms of $\alpha$.
	
	Additional discussion concerning discrete masses arriving on the discrete circle $\Z/n\Z$ are provided in their Section 5, notably a coupling equivalent to \eref{eq:qfgreht} is introduced for describing  the limiting discrete parking in terms of the continuous one.
	
	\subsubsection{Other conditioned tractable models}
	\label{sec:oth}
	
	We can imagine two other classes of models of mass that lead to asymptotic tractable models:
	
	\noindent(1). i.i.d. random masses $(m_i)$ conditioned by their sums. Their numbers $K_n$ can be random or not.\\
	(2). Prescribed deterministic masses $(m_i)$.

	In case (2), in order to obtain convergence of the collecting paths, some regularities are required. For example, one can fix in the $n$th model, the proportion $p_j(n)$ of masses equal to $w_j(n)$, and then demand some regularities for $p_j(n)$ and $w_j(n)$ (for example, require that $p_j(n)\to p_j$ for each $j$, and $nw_j(n)\to w_j$, and $\sum p_j=1$, $\sum p_j w_j \leq 1$). In other words, we prescribe the proportion of masses of each type, and require some convergence (in general additional assumption are needed to get a limiting behaviour, if the number of possible masses types ``$j$''\ goes to $+\infty$), see e.g. \cite{MR3188597}.

	To illustrate the case (1), we can for example, construct the masses using i.i.d. random variables $(X_i)$ taking their values in $\mathbb{N}$ (this simplifies the condition), with common distribution $\mu$, mean $m=1$ and variance $\sigma^2\in[0,+\infty)$, and take as mass 
	\[m_i= X_i/n.\] 
	With this choice, the space is naturally totally occupied at time $K_n$ close to $n$ (with  $\sqrt{n}$ fluctuations), and we may then require that, for example at time $K_n=n-a\sqrt{n}$ the free space is $b/\sqrt{n}$.
	
	Instead, we propose requiring that it be completely full at time $n$, which allows us to let time go backwards without getting lost in the change of time or variable and without altering the nature of the phenomenon at play.
	We will then condition by the event
	\[{\cal E}_n:=\biggl\{\sum_{i=0}^{n-1}m_i=1\biggr\}=\biggl\{\sum_{i=0}^{n-1}X_i=n\biggr\}.\]
	
	When one writes the  final collecting path 
	\[S_n^{(n)}(x)= -x+\sum_{i=0}^{n-1} \l(X_i/n\r)\1_{\bu_i\leq x},\]
	one sees that there is a binomial number $B(n,x)$ of masses that fall in $[0,x]$ (and if one is interested in the multidimensional distributions $S(x_i)-S(x_{i-1})$ where $x_0=0$ and $x_k=1$, then one notices that 
	\ben
	\l(\l|\l\{i~: \bu_i\in(x_{j-1},x_{j}]\r\}\r|,1\leq j\leq k\r)&=:&\left(\Delta_j(n),1\leq j\leq k\right)\\
	&\sim&{\sf Multinomial}\l(n, (x_j-x_{j-1},1\leq j \leq k)\r).\een
	This means that, in distribution, since the $X_i$ are exchangeable (still conditional on $\mathcal{E}_n$), if one defines
	\[ W_n(k)= \sum_{j=1}^k X_j/n,~~~k\in\{0,\cdots,n\},\]
	one has 
	\ben\nonumber\sqrt{n}\l(S_n(x_j),1\leq j \leq k\r)&\eqd& \sqrt{n}\l(-x_j + W_n({\Delta_1(n)+\cdots+\Delta_j(n)}) , 1\leq j \leq k\r)\\
	&=&\sqrt{n}\l(-x_j+\frac{\Delta_1(n)+\cdots+\Delta_j(n)}{n} , 1\leq j \leq k\r)\\
	\label{eq:qdgte}&& + \sqrt{n}\l(W_n({\Delta_1(n)+\cdots+\Delta_j(n)})-\frac{\Delta_1(n)+\cdots+\Delta_j(n)}{n} , 1\leq j \leq k\r).
	\een 
	We have \\ 
	-- If $\sigma^2=0$, then $X_i=1$ a.s. (this is the parking case), $W_n(k)=k/n$, so that the term in \eref{eq:qdgte} vanishes, and  $\sqrt{n}S_n(~.~)$ converges to the Brownian bridge in distribution.\\
	-- If $\sigma^2>0$, under the  hypotheses we have on $(X_i)$, \[\big(w_n(t),t\in[0,1]\big):=\big(\sqrt{n}(W_n({nt})-t) ,t\in[0,1]\big)\dd \big(\sigma\, {b}_t,t\in[0,1]\big)\] in $C[0,1]$, where $ {b}$ is a  Brownian bridge (we assume here that $W_n$ is interpolated linearly between integer positions).\footnote{The proof of this fact is an exercise, several methods being available: the convergence of the finite dimensional distributions (FDD) can be proven by proving a local limit theorem for them, and this latter being a consequence of the central local limit theorem. The tightness can be proven to be a consequence of the tightness of the same family of processes without the eventual conditioning by $S_n=n-1$ (see Proof of Lemma 1 in \cite{MR2167644}). An alternative proof (of the weak convergence) consists of using Kaigh's result \cite{MR415706} about the weak convergence in $C[0,1]$, of random walks conditioned to hit  -1 for the first time at the end, to Brownian excursion, and then use the fact that the random walk only conditioned by $S_n=-1$ are obtained by a random re-rooting of those conditioned to hit -1 at the end).}  
	Since $x^{(n)}_j:=\frac{\Delta_1(n)+\cdots+\Delta_j(n)}{n}\proba x_j$, the rhs of \eref{eq:qdgte}, which coincides with $\big( {w}_n(x^{(n)}_j),1\leq j\leq k\big)$ satisfies
	\[\big( {w}_n(x^{(n)}_j),1\leq j\leq k\big)
	\dd \sigma (b_{x_1},\cdots,b_{x_k}),\]
	for  $\max_{1\leq j\leq k}|x^{(n)}_j-x_j|\proba 0$ together with the tightness of $(\tilde{w_n})$ implies that $\big({w}_n(x^{(n)}_j)- {w}_n(x_j),1\leq j\leq k\big)\proba\big(0, 1\leq j\leq k\big)$ (and $w_n\dd  {b}$ implies $\big(w_n(x_j),1\leq j \leq k\big)\dd  \sigma \big(b_{x_j},1\leq j \leq k)$.
	
	On the other hand, by standard limit theorem on multinomial random variables
	\[\sqrt{n}\l(-x_j + {\frac{\Delta_1(n)+\cdots+\Delta_j(n)}{n}} , 1\leq j \leq k\r)\dd \l(\tilde{b}_{x_j},1\leq j \leq k\r)\] 
	where $b$ and $\tilde{b}$ are two independent Brownian bridges (and where the convergence to the two Brownian bridges holds jointly).  
	Hence,   in $C[0,1]$ \footnote{ 
		The Brownian bridge is a centered Gaussian process with covariance function $\cov(b_s,b_t)=s(1-t)$ for $0\leq s\leq t\leq 1$. Hence, the sum $\bar{b}:=c_1b +c_2\tilde{b}$  for two independent Brownian bridges $b$ and $\tilde{b}$, and two constants $c_1,c_2$ gives as a covariance function $\cov(\bar{b}_s,\bar{b}_t)=(c_1^2 +c_2^2)(s(1-t))$ }, 
	\[\sqrt{n}S_n(~.~)\dd \sqrt{1+\sigma^2} b.\]
	
	\begin{pro}\label{pro:dqgsr} For all compact interval $\Lambda\subset[0,+\infty)$
		\beq\label{eq:cv}
		\left(\sqrt{n}S_{n-\lambda \sqrt{n}}(x)-\sqrt{n}S_{n}(x),0\leq x \leq 1,\lambda\in \Lambda\right)\dd (-\lambda x, 0\leq x \leq 1, \lambda \in \Lambda)\eq
		where the convergence holds in $C([0,1]\times \Lambda,\R)$, 
		and
		\[ \l(\sqrt{n}S_{n-\lambda \sqrt{n}}(x) ,0\leq x \leq 1, \lambda \in \Lambda\r)\dd (\sqrt{1+\sigma^2}b_x-\lambda x, 0\leq x \leq 1, \lambda \in \Lambda)\]
		for the same topology.
	\end{pro}
	\begin{proof}
		The proof of \eref{eq:cv} here is simple because the convergence of the FDD is a consequence of the concentration of a sum of random variables. The tightness comes from the fact that $(\lambda,x)\mapsto \sqrt{n}S_{n-\lambda \sqrt{n}}(x)-\sqrt{n}S_{n}(x),0\leq x \leq 1,\lambda\in \Lambda$ is monotonous in both $\lambda$ and in $x$, and the limit is continuous. In the deterministic case, pointwise convergence toward a continuous function is uniform (on compact sets) as soon as the functions are monotonous. More details are given in \cite{MR2386089}, for ``a similar correction phenomenon'', and additional elements on topology are given in the Appendix there.
	\end{proof}  
	In Aldous \cite[p.168-169]{Aldous_exchangeability}, general criteria for the tightness of ``auto-normalised'' processes with exchangeable increments (such as those in the two previous examples)
	are given under the hypothesis that the sum of the increments are zero and sum of their square is $1$. This hypothesis must be slightly  adapted  here, since we want the limiting trajectories to end at a negative height (while regarding this height as the rescaled limit of the number of occupied components).
	
	In terms of excursion length above the current minimum, the limit is then the same as for the parking process. 
	
	\section{Energy dissipation and other cost associated with dispersion model}
	\label{sec:ED}
	This section {addresses} a natural question related to dispersion models: Assume that a mass $m_k$ arrives at some point $\bu_k$ within a partially occupied space ${\cal C}$. In a physical system, or in an abstract data structure (as are parking models, or hash tables), the mass -- which may be composed of smaller particles --  while it undergoes its dispersion process (until occupying a new space $\bO{k+1}\backslash \bO{k}$), have(s) moved or flowed, or maybe has lost some potential energy, or been pushed along a given distance, or been submitted to some other physical process during a particular time. In many situations, the main quantities of interest in a diffusion process are related to these additional considerations.
	
	Assume that  a ``unitary cost'' (a real number) can be associated with the dissipation of $(m_k,\bu_k)$. The ``global cost'' at time $t$ will then be defined as the sum of the unitary costs of the first $t$ masses. For example, in the RDCS, a mass $m$ can be seen as composed of tiny masses $\d m$, and if a mass $m$ arrives at a point $u$ inside $\bO{k}$, if one assumes that transporting this mass $\d m$ at distance $x$ has a cost $x\d m$, then the resulting cost of the displacement is
	\[\int_{\bO{k+1}\backslash\bO{k}} \flecheu{d}(u,y) \d y\]
	(recall notation given in Section \ref{sec:VCDM}).
	
	The problem we face is that the arrival of a single mass,  on a CC $O$ of $\bO{k}$ can result in the coalescence of many occupied CC. In general it is quite intricate to study, and even to define such a cost model, since we allow \ACDM to depend on the current occupied interval during the dispersion. 
	
	\textbf{The simplest models} are the
	\ADDM{} in which the masses only have size $1/n$ on ${\cal C}_n$.
	These models are simpler because the dispersion of the $k$th mass stops when it gets out of the CC of $\bO{k}$ in which it arrived (on its left or on its right), so that the global cost is a function of the collections of arrival block lengths $(|B_j|,j\leq t)$. In this case, the cost is entirely determined by the ``unitary cost model'' which specifies  the unitary cost distribution, for a single car arriving  uniformly in an interval of length $\ell$, for all $\ell$. This type of models has been studied for ``simple cost models'', for analysis of algorithm purpose mainly.\medskip
	
	We first discuss at length this case, and will come back for partial results, in the general case, in Section \ref{sec:qegrhtyu}.
	
	\subsection{Cost of parking construction  for general unitary cost function}
	\label{sec:CPC}
	Before discussing general costs, let us review standard parking cost functions. Again, this is the simplest case in terms of masses (that are all equal to $1/n$), but the cost model is general.

	To align with the standard representation, we will work with the set $\Z/n\Z$ and will come back to ${\cal C}_n$ later on.
	So, let us denote by $\bc_0,\cdots,\bc_{k-1}$ the arrival places of the cars, that are now i.i.d. uniform on $\Z/n\Z$ (with the letter ``c'' for the choice). In the standard parking problem, the car $i$ parks at the first available place among $\bc_i,\bc_i+1\mod n, \cdots$, say at place $\bc_i+d_i\mod n$ (with $d_i\in\{0,\cdots,n-1\}$).  Therefore, $d_i$ is the distance car $i$ has to do from its choice $c_i$ to its final parking spot.

	Define
	\beq {\sf Cost}(k)=\sum_{i=0}^{k-1} d_i.\eq
	\begin{rem}In computer science, the hashing with linear probing is an algorithm devoted to store some data $(x_i,i\geq 0)$, taken in a set ${\cal D}$ from more or less any type, in an array whose entries are labelled by $\Z/n\Z$. A hash function $h:{\cal D}\to \Z/n\Z$ is applied, and data $x_i$ is stored in the first available place among $h(x_i)+k \mod \Z/n\Z$, for $k$ ranging from 0 to $n-1$. Up to the vocabulary, this is the parking model under the hypothesis that the hash function produces (almost) uniform random variables. Under this hypothesis, ${\sf Cost}(k)$ corresponds to the number of cell availability tests in the array needed to store the first $k$ data $x_i$, and it is therefore {a crucial parameter in the analyzing the cost of this algorithm}.  \end{rem}
	Let us consider again the phase transition that occurs at time
	\[t_n(\lambda):=\floor{n-\lambda\sqrt{n}},\]
	\begin{pro}\label{pro:hash}
		We have the following convergences in distribution, [Flajolet \& Poblete \& Viola \cite{flajolet1997Analysis}, Janson \cite{MR1871562},
		Chassaing \& Marckert \cite{MR1814521}], 
		\[{\sf Cost}(n) / n^{3/2}\dd \int_0^1 e(t)\d t\]
		and [Chassaing \& Louchard\footnote{In \cite{chassaing2002phase}, \eref{eq:cl} is not stated, but it is a consequence of their Lemma 3.7, in which they prove the convergence in distribution of the \LL path encoding the parking at time $n-\lambda\sqrt{n}$, and normalised by $n^{1/2}$ and suitably change of origin, to  $(e_\lambda(s)-\underline{e_\lambda}(s), s\in [0,1]$). The cost being an integral of the \LL path, the convergence stated in our Proposition \ref{pro:hash} follows (this tool is used in a lot of places in the literature, including  \cite{MR1814521}, \cite{MR3547745}, \cite{MW2019}).}\cite{chassaing2002phase}]
		\beq\label{eq:cl}
		{\sf Cost}(t_n(\lambda)) / n^{3/2}\dd \int_{0}^1 \se_\lambda(s)-\underline{\se_\lambda}(s) \d s.\eq
		where $\se$ is again, the Brownian excursion, and
		$\se_\lambda(s)= \se(s)-\lambda s,~\textrm{and }~ \underline{\se_\lambda}(s)= \min_{t \leq s} ~\se_\lambda(t).$
	\end{pro}
	
	\subsubsection*{The  ``cluster size insertion sequence''.}

	Let us call an occupied cluster, a successive set of occupied places (as discussed above, an occupied CC $[a,b]$ on $\R/n\Z$, corresponds to the occupied cluster $\{a,a+1\mod n,\cdots,b-1\mod n\}$ in the standard parking terminology).  
	
	When the $k$th car arrives in the parking there are two possibilities. Either:\\
	-- the choice $\bc_k$ is an element of an occupied cluster $C=\{a,a+1 \mod n,\cdots a+s-1 \mod n\}$ of length $\bs_k:=s\geq 1$, and this occurs with probability $s/n$ for an occupied cluster of length $s$; in this case conditional on the fact that $\bc_k \in C$, then $\bc_k$ is uniform in $C$,\\
	-- or it arrives on a free site $f$ and parks there; its arrival cluster is then considered to have length $\bs_k=0$. Each free place at time $k$ has probability $1/n$ to be the choice of car $\bc_k$.\medskip
	
	Let  $\UCost^{(s)}$ denote the unitary (random) cost corresponding to the cost arrival in a block of length $s$ (knowing that  the arrival occurs in such a block, the position is uniform within it).  
	In general, we have
	\beq\label{eq:CST} \Cost_n(t) = \sum_{k\leq t}   \UCost^{(\bs_k)}_k\eq
	where again, $\bs_k$ is the length of the occupied cluster in which arrives the $k$th car, and the index $k$ at  $\UCost^{(\bs_k)}_k$ indicates that we take different and independent copies of $\UCost^{(s)}$ if several blocks of the same length $s$ are considered. 		Two main factors influence  the global cost $\Cost_n$: the sequence of distributions of the unitary cost  random variables $(\UCost^{(s)},s\geq 0)$, and  $(\bs_k)$ the  cluster length insertion sequence.
	
	\modd{\begin{note}\label{note:note1}  In the sequel, we assume that a complete bi-dimensional array of independent random variables  ${\sf Cost Array}:=\l(\UCost^{(s)}_i,i\geq 0, s\geq 0\r)$ is defined on our probability space, where the $\UCost^{(s)}_i$ are independent copies of $\UCost^{(s)})$, and this array is independent of the choices sequence $(\bc_k,k\geq 0)$. The choices sequence determines the complete history of the parking, and so the cost array is also independent of the cluster length insertion sequence $(\bs_k)$. %However, in general, the unitary costs used (that is, the sequence $(\UCost^{(\bs_k)}_k,k\geq 0)$) is not independent of $(\bc_k,k\geq 0)$, nor independent of  $(\bs_k)$.
			%{\color{cyan} Proposition alternative pour la dernière phrase :
				In general, the unitary costs used (that is, the random variables in the sequence $(\UCost^{(\bs_k)}_k,k\geq 0)$) are not mutually independent, not independent of $(\bc_k,k\geq 0)$ and nor independent of  $(\bs_k,k\geq 0)$. However, conditional on $(\bs_k,k\geq 0)$, the $(\UCost^{(\bs_k)}_k,k\geq 0)$ are indeed mutually independent.
		\end{note}}
		
		\subsubsection*{Asymptotic behaviour of the cluster size insertion sequence.}
		
		We already have some information on the cluster size sequence (see Proposition \ref{pro:disc}), and additional can be found in Pittel \cite{pittel1987linear}, Chassaing \& Louchard \cite{chassaing2002phase}, and in other cited references. At time $t_n(\lambda)$, only $\lambda\sqrt{n}$ cars are lacking, and the sequence $(\bs_k,k\leq t_n(\lambda))$ is composed by $t_n(\lambda)$ elements, most of them being very small (by Pittel, for $c\in[0,1)$ fixed, and some $C>0$, at time $cn$, all the clusters have size $\leq C \log n$ with probability converging to 1). Hence, say, $(1-\epsilon)n$ clusters have a logarithmic size in the insertion sequence. However, a small fraction of blocks have a large size, and they will play the main role in the sequel. We will make precise their number and sizes, by giving a limit theorem for the empirical measure associated with the ``large blocks'' in the cluster size insertion sequence.
		
		\begin{figure}[h!]
			\centerline{\includegraphics{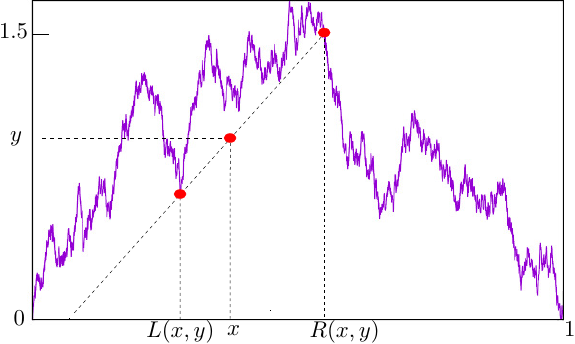}}
			\captionn{\label{fig:theline} Illustration of a Brownian excursion, the pair $(L(x,y),R(x,y))$ associated with a point $(x,y)\in {\sf under}(\se)$. The slanted dotted line has slope $S(x,y)$.}
		\end{figure} 
		Let ${\cal L}(\se)$ be the set of abscissa of local minima of a Brownian excursion $\se$ and let
		${\sf under}(\se)=\{(x,y):0<x<1,~ 0\leq y\leq \se_x \}$ be the set of points under the Brownian excursion curve.
		We say that $\se$ is above a linear map $\ell:[0,1]\to \R$ on $[a,b]$, if, for any $x\in[a,b]$, $\se_x\geq \ell(x)$ (see Figure \ref{fig:theline}). For any $(x,y)\in {\sf under}(\se)$ and $u \in{\cal L}(\se)$ such that $u<x$, define the linear map $\ell_{u,x}:[0,1]\to \R$ of the line passing by $(u,\se_u)$ and $(x,y)$ (that is $\ell_{u,x}(u)=\se_u$ and $\ell_{u,x}(x)=y$).
		Now, let $L(x,y)$ be the largest $u \in{\cal L}(\se)$ such that $u< x$ and such that $\se$ is above $\ell_{u,x}$ on $[0,x]$. Then, set $R(x,y)$ be the abscissa of the first intersection of $\ell_{u,x}$ with the graph of $\se$, on $[x,1]$.
		Let $S(x,y)$ be the slope of $\ell_{x}:=\ell_{u,x}$.
		
		It is easy to see that the map $(x,y)\to (L(x,y),R(x,y),S(x,y))$ is  {indeed} measurable.
		\begin{lem} \label{lemma:calcul_ps_Mlambda}
			For any continuous function $f:[0,1]\to \R$ such that $|f(x)|\leq Cx$ for some constant $C$, any $\lambda\geq 0$, set  
			\beq\label{eq:segry} \langle f, M_{\lambda}(\se) \rangle = \int_0^1  \int_{0}^{\se_x}    \frac{2}{R(x,y)-L(x,y)}  f\Big( R(x,y)-L(x,y) \Big)\1_{S(x,y)\geq \lambda} \d y \d x.\eq
			\begin{itemize}
				\item[(i)] Almost surely, $\langle f, M_{\lambda}(\se) \rangle$ is well defined and finite,
				\item[(ii)] $M_{\lambda}(\se)$ is a positive (random) measure on $[0,1]$ (equipped with the Borel $\sigma$-field).
				\item[(iii)] The measure $M_{\lambda}(\se)$ has a.s. infinite total mass.
			\end{itemize}
		\end{lem}
		In the sequel we will often write $M_\lambda$ instead of $M_{\lambda}(\se)$.
		\begin{proof}
			\modd{$(i)$
				First observe that for $g:x\mapsto x$, 
				\beq\label{eq:efette} \langle g, M_{\lambda}(\se) \rangle = 2 \int_0^1  \int_{0}^{\se_x}  \1_{S(x,y)\geq \lambda} \d y \d x 	=2\int_{0}^1 \se_\lambda(x)-\underbar{\se}_\lambda(x)\, \d x<+\infty\textrm{~a.s.}\eq
				and this last property can be seen by a simple picture (the first integral is the Lebesgue measure of the points above some chords in $\se$, and, in the graphical representation of $\se_\lambda -\underbar{\se}_\lambda$, the corresponding chords, correspond to the intervals that support the excursions above zero).\\	 
				If $f$ satisfies the hypothesis then 
				\beq\label{eq:efette1} \langle f, M_{\lambda}(\se) \rangle \leq C \langle g, M_{\lambda}(\se) \rangle\eq
				and this implies that $(i)$ holds. 
				$(ii)$ is clear.
				%Notice that for $f := x\mapsto x$ (we will write sometimes $f=\Id$)
				%	\beq\label{eq:efette} \langle f, M_{\lambda}(\se) \rangle = 2\int_0^1  \int_{0}^{\se_x}  \1_{S(x,y)\geq \lambda} \d y \d x =2\int_{0}^1 \se_\lambda(x)-\underbar{\se}_\lambda(x)\, \d x<+\infty\textrm{~a.s.}.\eq
				The point $(iii$) is proven just after the proof of Theorem \ref{theo:formule} (since it uses this theorem).}
		\end{proof}
		\modd{Next theorem applies to discrete parking models (on $\ZnZ$, where cars have size 1, arrive uniformly, and search a place according to some general policy, corresponding up to the natural scaling to \ADDM). By Theorem \ref{theo:excha-D}, the distribution of the multiset of occupied block lengths (at some time $k$) does not depend on the policy (and this is true also for the cluster size insertion sequence). However, the global cost of the parking construction depends on the unitary cost function:}
		\begin{theo}\label{theo:dqgreht}
			Suppose that for each $s$, $\UCost^{(s)}$ is in $L^2$, and set
			\[\Psi(s)= \E\l(\UCost^{(s)}\r),  ~~ V(s)= \Var\l(\UCost^{(s)}\r).\]
			Assume that there exists a sequence $(\alpha_n)$ with $\alpha_n \to +\infty$ such that the following conditions hold:\\
			$(i)$ for all $t>0$, $\Psi(\floor{nt})/\alpha_n \to \Psi_\infty(t)$ uniformly on $[a,1]$ for all $a>0$, where $\Psi_\infty$ is continuous on $[0,1]$ and satisfies $|\Psi_\infty(t)|\leq Ct$ for some constant $C\geq 0$, in a neighbourhood of zero. \\
			$(ii)$% For $`e>0$  
			\beq\label{eq:cond233} \modd{\lim_{\epsilon\to 0 }} \limsup_n \frac{1}{\alpha_n n^{1/2}} \l[ \frac n 2 \Psi(0) +4n \sum_{k=1}^{n\varepsilon} \frac{\Psi(k)}{k^{3/2}}\r] = \modd{0}.\eq 
			%o(`e). \eq 
			$(iii)$% For $`e>0$, 
			\beq\label{eq:cond234} \modd{\lim_{\epsilon\to 0 }}\limsup_n \frac{1}{\alpha_n^2 }  \l[ \frac{V(0)}2  
			+4 \sum_{k=1}^{n\varepsilon} \frac{V(k)}{k^{3/2}}\r] =\modd{0}. \eq %o(`e). \eq 
			$(iv)$ %For all $`e>0$
			\beq\modd{\lim_{\epsilon\to 0 }} \lim_n \sup_{x\geq `e} \frac{1}{\sqrt{n}\alpha_n^2} V(\floor{nx}) = \modd{0}.\eq%=0.\eq
			Under these four conditions, for all $\lambda \geq0$,
			\beq
			\frac{\Cost_n(\floor{n-\lambda \sqrt{n}})}{\sqrt{n} \alpha_n} \dd \langle \Psi_\infty, M_\lambda \rangle.
			\eq
		\end{theo}
		\begin{rem}\label{rem:fqqf} Some conditions, in the theorem, are stronger than needed. \\
			\bls For example, the condition on $\Psi_\infty$ near zero is sufficient (but probably not necessary) to have
			$\l\langle 1_{[0,\varepsilon]}|\Psi_{\infty}|, M_{\lambda}\r\rangle \to 0$ in probability as $\varepsilon$ goes to 0, because
			\beq
			\langle \1_{[0,\varepsilon]} |\Psi_{\infty}|,M_{\lambda}\rangle \leq C\,\int_0^1\int_{0}^{\se_x}  \1_{R(x,y)-L(x,y)\leq \varepsilon} \d y\d x\xrightarrow[`e\to 0]{(a.s.)} 0.\eq 
			\modd{\bls The fact that the cost function is positive can be withdrawn also if needed. In this case, the hypothesis have to apply to both the negative and positive parts of the cost.}  
		\end{rem}

		\subsubsection{Applications of Theorem \ref{theo:dqgreht}}
		
		In each of the following case, the expectation of the limit can be computed thanks to the formulas of Section \ref{sec:CD}. 
		
		\paragraph{Standard parking} $\UCost^{(k)}$ is $1+\bU_k$ where $\bU_k$ is uniform on $\{0,\cdots,k-1\}$ (in general, the cost is defined as the set of ``tries'' a car needs to park, but this $+1$ can be dropped, if one prefers, it does not change the asymptotic results). We then have 
		\[\Psi(k)= (k+1)/2, ~~V_k= (k^2-1)/12,  \textrm{ for }k\geq 1\]
		and $\Psi(0)=1$, $V(0)=0$.\par
		It follows the hypothesis of Theorem \ref{theo:dqgreht} holds, since :\\
		-- Condition $(i)$ holds  with $\alpha_n=n$, $\Psi_\infty(t)=t/2$, \\
		-- Condition $(ii)$ holds since $\frac{1}{  \alpha_n n^{1/2}} \l(\frac{n}{2} \Psi(0)+4n\sum_{k=1}^{n\varepsilon} \frac{\Psi(k)}{k^{3/2}}\r)=\frac{1}{\alpha_n n^{1/2}}4nO((n`e)^{1/2})=\modd{O(`e^{1/2})}$.  \\
		-- For condition $(iii)$, write $\frac{1}{\alpha_n^2}\l(\frac{V(0)}{2}+4\sum_{k=1}^{n`e} \frac{V(k)}{k^{3/2}}\r)= \frac{1}{n^2} O( (n`e)^{3/2})=o(`e)/\sqrt{n}$.\\
		-- Condition $(iv)$ immediate. 
		
		Notice that $(ii)$ is tight in this case. We then have 
		\beq 
		\frac{\Cost_n(\floor{n-\lambda \sqrt{n}})}{n^{3/2}} \dd \langle \Psi_\infty, M_\lambda \rangle=\int_{0}^1 e_\lambda(s)-\underline{e_\lambda}(s)  ds
		\eq
		and we recover by \eref{eq:efette}, the Chassaing \& Louchard results recalled in Proposition \ref{pro:hash}.

		\paragraph{Parking with random direction}
		If upon arrival the drivers choose to go to the right with probability $p$, and to the left with probability $1-p$, the unitary cost has the same distribution as for the previous point: the same asymptotic result holds.
		
		\paragraph{Choosing the closest place strategy}
		
		This strategy provides a cost $1+\min\{\bU_k, k-\bU_k\}$ for $\bU_k$ a uniform random variable on $\{0,\cdots,k\}$. It follows that $\Psi(k)/k \to 1/4$ and $V(k)/k^2\to 1/48$, since   $(1+\min\{\bU_k, k-\bU_k\})/k$ is a sequence of uniformly bounded random variable converging in distribution to the \modd{uniform distribution on $[0,1/2]$. Its mean and variance converges to those of this distribution. Besides $V(k)=O( k^2)$ (the maximum variance for a random variable with support $\cro{1,k+1}$. Hence, the standard parking analysis above applies, with $\Psi_{\infty}(t)=t/4$ instead (so, asymptotically, a factor $(1/4)/(1/2) =1/2$ of the usual parking cost).}
		
		\paragraph{Making a $p$-random walk to find a place, with $p>1/2$}
		
		$\UCost^{(k)}$ is the exit time of an interval of length $k$ by a $p$-random walk (the steps are iid, a step $+1$ occurs with probability $p$, a step $-1$ with probability $1-p$), starting from a uniform point. {We skip the details since the asymptotic are the same for two costs that differs from a constant}.
		In this case $\UCost^{(k)}/k$ converges in distribution to $U/(2p-1)$ (the asymptotic speed of a $p$ random walk is $2p-1$), where $U$ is uniform on $[0,1]$. 
		The convergence of $\Psi(k)/k$ and $V(k)/k^2$ to $\E(U)/(2p-1)=1/(2(2p-1))$ and to ${\sf V}(U)/((2p-1)^2)=1/(12(2p-1)^2) $ can be proven using the Hoeffding inequality which ensures exponentially rare deviation for $\UCost^{(k)}/k$ from its expected values (at a smaller scale). Hence, 
		\beq \label{eq:asymp1}
		{\sf Cost}(n-\lambda \sqrt{n}) / n^{3/2}\dd \int_{0}^1 \frac{e_\lambda(s)-\underline{e_\lambda}(s)}{2p-1} ds
		\eq
		
		\paragraph{Making a $p$-random walk to find a place, with $p= 1/2$}
		
		Here $\UCost^{(k)}$ is the time $T^{(k)}$ needed for a centered random walk to exit an interval of time $k$, starting from a uniform position. 
		It can be proven that 
		\[\Psi(k)/k^2 \to 1/6, V(k)/k^4 \to 19/180.\]
		There are two main methods to prove this result. The first one consists in noticing that $\UCost^{(k)}/{k^2}\dd T_U$
		where $T_U$ is the needed time needed for a Brownian motion $B^{(U)}$ starting from $U$ uniform on $[0,1]$ to exit the interval $[0,1]$, which is a consequence of Donsker. Since by L\'epingle \cite{MR622596}, for a Brownian motion $B^{(u)}$ starting at a fixed point $u\in[0,1]$, the exit time $T_{u}$   from $[0,1]$, satisfies
		\[\E\l[\exp\l(\alpha^2T_u/2\r)\r]= \frac{\cos(\alpha(1-2u)/2)}{\cos(\alpha /2)}\] for $0\leq \alpha< \pi $. Follows  that,
		\beq \label{eq:qgr}\E\l[\exp\l(\alpha^2T_U/2\r)\r]= \frac{2\sin( \alpha/2)}{\alpha \cos(\alpha/2)},\eq
		and then $T_U$ admits all its moments, and $\E(T_U)=1/6$ as well as $\E(T_U^2)=2/15$ can be obtained by expansion of $\E(\exp(x T_U))$ (set $\alpha^2/2=x$ in \eref{eq:qgr}). Now, to prove that the moments of $T^{(k)}/k^2$ converges to those of $T_U$ some uniform integrability argument can be used, for example, Komlos, Major and Tusn{\'a}dy \cite{MR402883} strong convergence theorem (with speed), or again, the second part of L\'epingle \cite{MR622596} paper, who computed the Laplace transform of the exit time of a random walk from an interval, which ensures the convergence of the Laplace transform of $T^{(k)}/k^2$ to that of $T_U$ (for positive arguments, in a neighbourhood of zero), from which convergence of moments are immediate.  
		A second proof, totally different, would rely on martingale argument: it is well known that the hitting time $T_{a,b}$ of $\{a,b\}$ for a random walk starting at 0, were $a<0<b$, satisfies $\E(T_{a,b})=-ab$, and this can be proven by the martingale stopping time theorem, applied to the martingale $(W_k^2-k,k\geq 0)$ where $W_k$ is the simple  symmetric random walk. The other moments of $T_{a,b}$ can be computed using  the family of martingales,indexed by $c$, $Z_n^{(c)}=e^{cW_n}(\cosh c)^{-n}$ (see also L\'epingle \cite{MR622596}, and Chung \cite[Sec. II.3]{MR1371379} for Brownian motion techniques that can be adapted to random walks).  
		Now take $\alpha_n=n^2$. Since $\Psi(k)/k^2 \to a:=1/6$ and $V(k)/k^4\to b:=19/180$
		\[\Psi(nt)/n^2  \to a t^2, \limsup_n \sup_{x\geq `e} \frac{V(\floor{nx})}{\sqrt{n}\alpha_n^2}\to 0.\]  
		Since $\Psi(k)/k^2 \leq C$ and $V(k)/k^4 \leq C$ by convergence, for a well-chosen constant $C$,
		one can check that $(ii)$, $(iii)$ and $(iv)$ hold. (The convergence in $(i)$ follows the fact that $\E(T_{a,b})=-ab$ from which the results follow). We then have
		\beq \label{eq:asymp2}{\sf Cost}(n-\lambda \sqrt{n},p) / n^{5/2} \dd \langle g, M_\lambda \rangle=\frac{1}{3}\int_0^1\int_0^{\se_x} (R(x,y)-L(x,y))\1_{S(x,y)\geq \lambda} \d y\d x\eq
		where $g(x)= x^2/6$.

		\subsection{Proof of Theorem \ref{theo:dqgreht}}
		
		This section   consists of three subsections. First, we discuss about some properties of $M_{\lambda}$, and provide an alternative representation of this measure. Next, we introduce the measure $\Theta_n^{(\lambda)}$ that encodes the insertion length sequence $(\bs_k,0\leq k\leq t_n(\lambda))$. We then establish its convergence to $M_\lambda$ after normalisation, and conclude with the proof of Theorem \ref{theo:dqgreht}.

		\subsubsection{Some distribution identities concerning the measures $M_\lambda$} 
		
		Take $\lambda>0$, and a  continuous function $f:[0,1]\to \R$ such that $|f(x)|\leq C.x$, for some constant $C$ and all $x\in[0,1]$. 
		Consider the null set   $Z_\lambda=\{x: \se_\lambda(x)-\underbar{\se}_\lambda(x)=0\}$.
		The complement of $Z_\lambda$ in $[0,1]$ is an open set $L_\lambda$, whose CC are countably many open intervals $L_\lambda=\cup_{i\geq 1} L_\lambda(i)$. The null set $Z_\lambda$ has Lebesgue measure zero, a.s., so that a.s., for $ \ell_\lambda(k)= \Leb(L_\lambda(k))$, we have
		\[\sum_{k\geq 1} \ell_\lambda(k)=1,\]
		and then we let 
		\begin{equation}
			\rho_\lambda=\sum_{k\geq 1} \delta_{{\ell_\lambda(k)}},
		\end{equation} and we define 
		\[\langle \rho_\lambda,f\rangle =\sum_{k\geq 1} f( {\ell_\lambda(k)}).\]
		It is well-defined, by Fubini, and finite. In particular, this is the case for $f(x)=x^\beta$ for $\beta \geq 1$.  
		
		A.s. the map $\lambda \mapsto \langle \rho_\lambda,f\rangle$ is Lebesgue measurable, and Lebesgue integrable on any compact of $\R^+$.
		To prove this, we take $f$ non-negative (which is sufficient), and a time interval $[t,t']\subset[0,+\infty)$. For $`e>0$ there exists $N$ such that the $N$ largest intervals $\l(\widehat{\ell}_{t'}(1),\cdots,\widehat{\ell}_{t'}(N)\r)$, sorted in decreasing order, satisfy $ \sum_{k=1}^N\widehat{\ell}_{t'}(k)\geq 1-`e$. Since the intervals in $L_{t'}$ are obtained from those of $L_{t}$ by fragmentation,
		for any $t\leq \lambda \leq t'$,  $ \sum_{k=1}^N \widehat{\ell}_\lambda(k)\geq 1-`e$. Since $|f(x)|\leq C.x$, up to $`e>0$, $\langle \rho_\lambda,f\rangle$ is determined by the $N$ largest intervals of $L_\lambda$ (because $\sum_{k>N} f(\widehat{\ell}_\lambda(k))\leq C\sum_{k>N} \widehat{\ell}_\lambda(k)\leq C`e$).  
		
		Let $\bar{L}_\lambda(k)$ be the interval, at time $\lambda$, that contains  the $k$th largest interval $L_{t'}(n_k)$ at time $t' \geq \lambda$ (whose length is $\widehat{\ell}_{t'}(k)$). It can happen that  $\bar{L}_\lambda(k)=\bar{L}_\lambda(j)$ for some $k\neq j$. If it is the case, then keep only the ones with smaller index and keep $( \bar{L}_\lambda(k), k \in I(\lambda))$ where $I(\lambda)$ is the corresponding set of indices.  
		
		For the corresponding interval lengths $\l(\bar{\ell}_\lambda(k), k\in I(\lambda)\r)$, we also have\linebreak $\l|\sum_{k\in I(\lambda)} f(\bar{\ell}_\lambda(k))- \langle \rho_\lambda,f\rangle\r|<{C}`e$. The maps $\lambda\mapsto \bar{\ell}_\lambda(k)$ as $\lambda$ goes from $t'$ to $t$, are non decreasing (and may possibly vanish in case of coalescence with another interval of this same family with lower index). It follows that, for fixed $k$, $\lambda\mapsto f(\bar{L}_\lambda(k))$ is measurable (by taking the value zero in case of disappearance). The integrability is clear since $\langle \rho_\lambda,f\rangle\leq C$. 
		Thus
		\begin{equation}
			\int_t^{t'} \langle \rho_\lambda,f\rangle \d \lambda= \int_t^{t'}\sum_k f(\ell_\lambda(k)) \d\lambda,
		\end{equation}
		is well-defined and finite.	This well-defined ness can be extended to Lebesgue measurable functions $f$ satisfying $\sup\{|f(x)/x|, x\in[0,1]\}<+\infty$, by density.
		\begin{theo}\label{theo:formule} For all bounded and Lebesgue measurable function $g:[0,1]\to \R^+$,
			\beq\langle g, M_{\lambda}(\se) \rangle  = \int_{ \lambda}^{+\infty} \sum_{k\geq 0} g( {\ell_t(k)})\, {\ell_t(k)} \,   \d t, ~~~\textrm{ a.s.}   \label{eqn:scal_g_Mlambda}  \eq
			
			In particular, for $f = \Id/2$, we have: 
			\ben\langle f, M_{\lambda}(\se) \rangle 
			&=&\int_{0}^1 \se_\lambda(s)-\underbar{\se}_\lambda(s)\, \d s
			= \frac12\int_{ \lambda}^{+\infty} \sum_{k=1}^{+\infty}  {\ell_t(k)}^{2} \d t.       \een
			
		\end{theo}
		The first statement of the theorem can be rephrased as: for $G:=x\mapsto xg(x)$, $\langle g, M_{t}(\se) \rangle  = \int_{ t}^{+\infty}  \langle \rho_\lambda,G\rangle d\lambda$.\\
		
		The first equality is discussed above, see \eqref{eq:efette}. See Section \ref{sec:CD} for closed formulas for   $\E\l(\langle x\mapsto x^k, M_{\lambda} \rangle \r)$.
		
		\begin{proof}
			\textbf{Decomposition of $\langle f, M_{\lambda}(\se) \rangle $ as a sum over the abscissa of local minima of $\se$ for general $f$ }. \\ 
			We may partition $\under(\se)=\cup_{m \in{\cal L}(\se)} {\cal P}(m)$ where 
			\[{\cal P}(m)=\{(x,y)\in \under({\se})~:~ L(x,y)=m\}.\]
			The sets $({\cal P}(m), m \in{\cal L}(\se))$ form a partition of $\under(\se)$, and then any integral on the set $\under(\se)$ can be written as a sum over ${\cal L}(\se)$ (under usual conditions, as Fubini's theorem).  
			We will show that :
			\ben \langle g, M_{\lambda}(\se) \rangle 
			&=& \int_{ \lambda}^{+\infty} \sum_{m\in\mathcal{L}(\se)} \ell(m,{t}) g(\ell(m,{t})) \d t,
			\een
			where $\ell(m,t)$ is the size of the segment with slope $t$, starting at $(m,\se(m))$ and stopped when it passes above the graph of $\se$, or, in other words, the length of the excursion of $x\to\se_{t}(x)-\underline{\se}_t(x)$ starting at $m$.

			For any Borel function $f:[0,1]\to \R^+$, 
			\be \langle f, M_{\lambda}(\se) \rangle 
			&=& \int_0^1  \int_{0}^{\se_x}    \frac{2 f\Big( R(x,y)-L(x,y) \Big)}{R(x,y)-L(x,y)}  \1_{S(x,y)\geq \lambda} \d y \d x\\
			&=& \sum_{m\in \mathcal{L}(e)} \int_0^1  \int_{0}^{\se_x}    \frac{2 f\Big( R(x,y)-m \Big)}{R(x,y)-m}  \1_{S(x,y)\geq \lambda} \1_{(x,y)\in{\cal P}(m)}\d y \d x.
			\ee
			For any $(x,y)\in{\cal P}(m)$, by letting $t$ be the slope $t=S(x,y)$, one has  $y = \se(m) + t(x-m)$. We may then proceed to a change of variable $(x,y)\to (x,t)$.
			For this choice, $\1_{S(x,y)\geq \lambda}$ will be transferred to the integration domain, $\{t: t\geq \lambda\}$.
			We obtain  
			\ben\nonumber\langle f, M_{\lambda}(\se) \rangle  &=& \sum_{m\in \mathcal{L}(e)} \int_0^1 \int_{ \lambda}^{+\infty}    \frac{2f\Big( R(x,\se(m) + t(x-m))-m \Big)}{R(x,\se(m) + t(x-m))-m}   (x-m)\1_{(x,\se(m)+t(x-m))\in{\cal P}(m)} \d t \d x~~\\
			\label{eq:g297et}		&=&  \sum_{m\in \mathcal{L}(e)} \int_{ \lambda}^{+\infty}    \frac{2 f\Big( \ell(m,t)\Big)}{\ell(m,t)}   \int_{m}^{m+\ell(m,t)}(x-m)\d x \d t \\ 
			\label{eq:qsfdq}	&=& \int_{ \lambda}^{+\infty}  \sum_{m\in\mathcal{L}(\se)} \ell(m,t) f(\ell(m,t)) \d t.
			\een  
		\end{proof}
		
		\begin{proof}[Proof of Lemma \ref{lemma:calcul_ps_Mlambda}$(iii)$, that is the fact that $M_\lambda$ has infinite total mass] The first formula in Theorem \ref{theo:formule} applied to the constant function $g$ equal to 1, gives the total mass of $M_{\lambda}(e)$: Hence this total mass is $\int_{ \lambda}^{+\infty} \sum_{k\geq 0}   {\ell_t(k)} \,   \d t$. As already said, a.s.,  for all $t$, $\sum_{k\geq 0}   {\ell_t(k)}=1$, so that the total mass is infinite.
		\end{proof} 
		
		\subsubsection{ A measure encoding the cluster size insertion sequence}
		\label{sec:MEAI}
		We store in a measure ${\sf Mes}_n^{(  t)}$ the elements of the cluster size insertion sequence $(\bs_k,k\geq 0)$ viewed before time $t$ (with multiplicity)~: 
		\beq\label{eq:mes1} {\sf Mes}_n^{(  t)} = \sum_{j\leq t} \delta_{\bs_j}.\eq
		A consequence of \eref{eq:CST}, is that when some unitary cost laws (the laws of $\UCost^{s}$) are fixed, then the total cost can be expressed as 
		\beq\label{eq:mes2} {\sf Cost}_n(t) = \sum_k \sum_{j=1}^{{\sf Mes}_n^{(t)}(\{k\})} {\sf UCost}_j^{(k)}. \eq
		
		\noindent{\bf Foreword.} In view of \eref{eq:mes1} and \eref{eq:mes2}, we may expect that under some regularity assumptions on the unitary cost laws, the knowledge of the limiting behaviour for ${\sf Mes}_n^{(t)}$, after rescaling, when $n\to +\infty$, and for ${\sf UCost}^{(k)}$ should lead to the convergence result  for the global cost. While this reasoning is somewhat valid,   to capture the asymptotic behavior of ${\sf Mes}_n^{(t)}$, we  need to apply a normalization that causes  clusters with sub-linear sizes vanish at the limit. Those with linear sizes provide a limit random measure with infinite total mass.  The difficulty that follows is that, for some cost functions, small clusters can play a role, while for others, only large ones matter. This is apparent in Theorem \ref{theo:dqgreht}, where our hypotheses  imply that only large clusters will characterize the limiting behavior of the cost, which is not true in general. 
		
		~\\\textbf{Several normalisations are possible} for ${\sf Mes}_n^{(  t)}$, depending on the values of $t$   to be studied, and the unitary cost model. The normalization is mainly threefold, and it is natural to consider
		\[\frac{1}{\alpha_n} {\sf Mes}_n^{( t_n)}( \,.\,/ \beta_n).\]
		Here,$\alpha_n$ addresses the total mass normalisation, while $\beta_n$ is the block size normalization, and $t_n$ is the time at which we stop the construction.
		
		\noindent\underbar{For example,} if ${\sf UCost}^{(k)}=1$ for all $k\geq 0$, the total cost is ${\sf Cost}_n(t) = t$, it is not needed to search for a limit of ${\sf Mes}_n^{(t)}$ to see that! Nevertheless when searching for a limit of ${\sf Mes}_n^{( t_n)}$, for the linear time  $t=t_n:=an$ (and $a\in(0,1)$, the maximum $\bs_j$ for $j\leq an$ has order $O(\log n)$ in probability by \cite{pittel1987linear}). It means that, taking a limit over $n$, 
		\[\frac{1}{n}{\sf Mes}_n^{( an)}(~.~ / n)\dd a\delta_0, \] so that the block size information essentially vanishes.
		This exemplifies a more general phenomenon. It might seem natural to take $\alpha_n=t_n$ to maintain a measure with total mass one. However, since most of the blocks are very small (at least for $t_n$ close to $n$) this leads to the (degenerated) Dirac mass at the limit. Thus, we must then change the normalisation in order to analyze much of the cost model. 
		
		In order to describe the large blocks, $\alpha_n$ must be different from the total mass:  the number of  	clusters that have a linear size has order $\sqrt{n}$ (for the critical time $t_n(\lambda)=\floor{n-\lambda \sqrt{n}}$).
		We then consider the random Borel  measure $\Theta_n^{(\lambda)}$ characterised, for all interval $A\subset [0,1]$, by 
		\ben\label{eq:mes3}
		\Theta_n^{(\lambda)}(A)&:=& \frac{1}{\sqrt{n}} {\sf Mes}_n^{\l(t_n(\lambda)\r)}(n A )\\
		&=&\sum_{k: k \in nA} \frac{1}{\sqrt{n}} {\sf Mes}_n^{\l(t_n(\lambda)\r)}(\{k\} )\\
		&=&\int_0^n \frac{1}{\sqrt{n}} {\sf Mes}_n^{\l( t_n(\lambda)\r)}(\floor{x} )\1_{ {\floor{x}}\in nA} \d x\\
		&=&n\int_0^1 \frac{1}{\sqrt{n}} {\sf Mes}_n^{\l( t_n(\lambda)\r)}(\floor{ny} )\1_{\floor{ny}/n\in A} \d y.
		\een
		These four formulas help explain the applied normalisation: integration of a continuous and bounded function against $\Theta_n^{(\lambda)}$ coincides then with
		\beq\label{eq:qsUIdq} \langle f, \Theta_n^{(\lambda)}\rangle = n\int_0^1 \frac{1}{\sqrt{n}} {\sf Mes}_n^{\l( t_n(\lambda)\r)}(\floor{nx} ) f(x) \d x.\eq  
		The measure  $\Theta_n^{(\lambda)}$  has total mass $\sqrt{n}-\lambda\to+\infty$ with, much of this mass close to the point zero. 
		Let us consider the vague topology on the set of Borel measures on ${\cal M}(0,1)$ (which, prosaically, allows us to ignore this degeneracy at zero, since the test functions for these topologies are those with compact support strictly included in $(0,1)$).

		\begin{theo}\label{theo:vague} 
			For any $\lambda>0$, 
			$ \Theta_n^{(\lambda)} \dd M_\lambda$ for the vague convergence topology on ${\cal M}(0,1)$ (and the joint convergence  $(  \Theta_n^{(\lambda_i)}, 1\leq i \leq \kappa) \dd (M_{\lambda_i},1\leq i \leq \kappa)$ holds too, for any times $0<\lambda_1<\cdots<\lambda_\kappa$ and any $\kappa$).
			
		\end{theo}
		
		\begin{proof} We borrow to Chassaing \& Louchard  \cite{chassaing2002phase} some elements of their analysis.
			
			Consider a random walk $S[n]:=(S_0=0,S_1,S_2,\cdots)$ with i.i.d. increments $(X_i-1,i\geq 1)$, where the  $X_i$ are Poisson(1) random variable. Denote by $\tau_{-1}(S)=\inf\{j: S_j=-1\}$.
			
			The collecting path (recall Section \ref{sec:collecting_paths}) of the classical parking problem on a size $n$ parking, stopped when $n-1$ cars are parked, and conditional on the fact that the last place is empty, is distributed as $S[n]$ conditioned by $\tau_{-1}(S)=n$, where $\tau_{-1}(S)=\inf\{j: S_j=-1\}$. Moreover, the vector $(Y_1,\cdots,Y_n)$ of cars that have chosen place $1$ to $n$, satisfies
			$
			\Law(Y_1,\cdots,Y_n) = \Law\l(  (X_1,\cdots,X_n)~|~ \tau_{-1}(S)=n\r).
			$
			The corresponding state of this random walk at time $t_n(\lambda)$, that is when $\lambda\sqrt{n}$ cars are lacking, is denoted $S^{(t_n(\lambda))}$. 
			Let 
			\[\l(s^{n,(\lambda)}_t,t \in[0,1]\r):= \l( n^{-1/2}\,S^{(t_n(\lambda))}_{nt},t \in[0,1] \r)
			\]
			where $S^{(t_n(\lambda))}$ is seen as a continuous process, interpolated linearly between integer points. 
			In Broutin \& Marckert \cite[Theo.8] {MR3547745} is has been shown that
			\beq\label{eq:qget} \l(s^{n,(\lambda)}_t, t\in [0,1], \lambda\geq 0\r) \dd \l((\se_t -\lambda t) - \min_{s\leq t} (\se_s-\lambda s), t\in [0,1], \lambda\geq 0\r)\eq  in $D(\R^+,(C[0,1],\R))$, while the convergence of FDD (that is, for a finite number of distinct times $\lambda_i$) was already present in \cite{chassaing2002phase} (we see $s^{n,(\lambda)}$ as a process indexed by $\lambda$, taking, for each $\lambda$ its values in $C([0,1],\R)$). 
			
			By the Skorokhod representation theorem, there exists a probability space on which there exists some copies $\bigl(\widetilde{s}_t^{n,(\lambda)},\lambda\geq 0,t\in [0,1]\bigr)$ of $\bigl(s^{n,(\lambda)}_t,\lambda\geq 0,t\in [0,1]\bigr)$ (for each $n$), and a copy $\widetilde{\se}$ of $\se$, such that 
			\[\l(\widetilde{s}_t^{n,(\lambda)}-\l[(\widetilde{\se}_t -\lambda t) - \min_{s\leq t} (\widetilde{\se}_s-\lambda s)\r],t\in [0,1],\lambda\geq 0\r)\as  0\] for the same topology. \par
			Consider ${\cal E}$ the two following conjunction of events:\\
			$(i)$ ${\cal L}(\se)$ is countable and each  $x\in (0,1)\cap{\cal L}(\se)$ is a global minimum on a non-empty open interval,\\
			$(ii)$ if for $j$ from 1 to 3, the $p_j:=(x_j,\se_{x_j})$ are points on the curve of $\se$, such that  the $x_i$ are distinct elements of ${\cal L}(\se)$, then the $p_j$ are not aligned.~\\
			The measure $\Theta_n^{(\lambda)}$ records only the insertion block sizes, but can be enriched to record also the position of these cluster (on $\widetilde{s}_t^{n,(\lambda)}$): for $B$ a Borel set of $\R^2$, set
			\[ \Xi_{n}^{(\lambda)}(B)=\frac{1}{\sqrt{n}}\sum_{k\leq t_n(\lambda)} \delta_{[a_k/n,b_k/n]\in B}\] where $[a_k,b_k]$ is the cluster, which size $\bs_k$, which contains the arrival position $\bu_k$ (and which was then an excursion  position in the $k$th collecting path $S^{(k)}$).
			Now, recall the computation made in the proof of Theorem \ref{theo:formule}, and in particular the decomposition of the measure $M_\lambda$ as a sum over the abscissa of local minima of $\se$, that is $M_{\lambda}=\sum_{m\in{\cal L}(e)} M_\lambda^{(m)}$ and \[\langle f, M_\lambda^{(m)}\rangle=\int_0^1  \int_{0}^{\se_x}    \frac{2 f\Big( R(x,y)-m \Big)}{R(x,y)-m}  \1_{S(x,y)\geq \lambda} \1_{(x,y)\in{\cal P}(m)}\d y \d x,\]
			then, this decomposition $M_{\lambda}$ can also be seen as a measure on $\R^2$ of the limiting block positions, and in $M_\lambda^{(m)}(A)$, the number $m$ is the left-hand side and $A$ measures the distance from the right-hand-side to $m$. Further, for $\lambda'>\lambda$,
			\[  M_{\lambda}^{(m)}(A):=\langle \ind{A}, M_\lambda^{(m)}\rangle=\int_0^1  \int_{0}^{\se_x}    \frac{2 }{R(x,y)-m}  \1_{S(x,y)\geq \lambda} \1_{(x,y)\in{\cal P}(m)}\1_{R({x,y})\in m+A}\d y \d x.\]
			By the same computation as in \eref{eq:g297et}, for $\lambda'\geq \lambda$,
			\[
			\langle f, M_{\lambda}^{(m)}(\se)-M_{\lambda'}^{(m)}(\se) \rangle=\int_{ \lambda}^{\lambda'}   \ell(m,t) f(\ell(m,t)) \d t\] so that for $f=\1_{[a,b]}$ ,
			\[
			\langle f, M_{\lambda}^{(m)}(\se)-M_{\lambda'}^{(m)}(\se) \rangle=\int_{ \lambda}^{\lambda'}   \ell(m,t) \1_{\ell(m,t)\in[a,b]} \d t.\] 
			Let us see that this gives the limit number of insertions in clusters located ``approximately'' in $[a+m,b+m]$, normalised by $\sqrt{n}$, between the times $t_n(\lambda')$ and $t_n(\lambda)$.
			We already know that  $\widetilde{s}_t^{n,(\lambda)}\to e_t-\lambda t $ uniformly for $(\lambda,t)\in [\Lambda_1,\Lambda_2]\times [0,1]$ (because the limit is a.s. continuous), and that there exists an interval $[m_a,m_b]$ containing $m$ in its interior on which $m$ is a global minimum (a.s.). Let $m_n$ be the value of $t$ at which $\widetilde{s}_t^{n,(\lambda)}$ takes its minimum on $[m_a,m_b]$ for the last time.
			By defining the analogues  $R_n(x,y)$, $S_n(x,y)$ and $\ell_n(m_n,t)$ of $R(x,y)$, $S(x,y)$ and $\ell(m,t)$ one sees that
			\[\langle f, M_{\lambda}^m(\se)-M_{\lambda'}^m(\se) \rangle=\lim_{n\to\infty}\int_{ \lambda}^{\lambda'}   \ell_n(m_n,t) \1_{\ell_n(m_n,t)\in[a,b]} \d t,\]
			for almost all $a,b,\lambda,\lambda'$, a.s..
			It remains to be seen that this quantity gives also the limit number of insertions during the time interval $[t_n(\lambda'),t_n(\lambda)]$, in an occupied interval starting in $nm_n$ and  of length belonging to $n[a,b]$, divided by $\sqrt{n}$ (call this value $\xi_n(\lambda,\lambda',m,a,b)$, and this is a sort of equivalent of $\Theta_n^{(\lambda)}-\Theta_n^{(\lambda')}$, enriched, to take into account the minimum $m$, and the start and length of insertion intervals). 
			For some time $k\in[t_n(\lambda'),t_n(\lambda)]$, the probability that $\bu_k$  arrives in such an interval, conditional on $S^{(k)}$ is either 0 if there is no excursion starting in $m_n$ and such that $\ell_n(m_n,k)\in n[a,b]$,   or $\ell_n(m_n,k)/n$ if such an excursion exists. Now, by a concentration argument for martingales with bounded increments (as Azuma's inequality), 
			we have
			\[\xi_n(\lambda,\lambda',m,a,b)-\int_{ \lambda}^{\lambda'}   \ell_n(m_n,t) \1_{\ell_n(m_n,t)\in[a,b]} \d t,\] goes to zero in probability.
			This result implies the announced result.
		\end{proof}
		
		\begin{lem} 
			\bir \itr For all $\lambda \geq 0$,
			\[X_{n}^{(\lambda)}:=\max_{j \leq t_n(\lambda)} \bs_j/n \dd X^{(\lambda)}\]
			where $X^{(\lambda)}$ is a.s. in $(0,1)$.
			\itr We have {$X^{(\lambda)} \probabis0$} when $\lambda\to+\infty$.
			\itr For all $\varepsilon>0$, there exists $\Lambda$ such that for all $\lambda\geq \Lambda$, for all $n$ large enough 
			\[ \P\l(X_n^{(\lambda)}\geq \varepsilon\r)\leq \varepsilon.\]
			\eir
		\end{lem}
		\begin{proof} \bir \itr This is a consequence of  \cite[Section 3.3]{chassaing2002phase} (which itself relies on Aldous \cite[Lemma 7]{Aldous1997}), which allows us to deduce the convergence of the excursion size above the current minima of a process, which converges in distribution, to those of the limit under some natural assumptions.

			\itr We borrow this argument to \cite{varin2024golf} (many other proofs are possible but this one does not depend on the statistical properties of the process $e_t-\lambda_t$). Consider any continuous function $f:[0,1]\to \R$ such that $f(0)=f(1)=0$, and consider the maximal length $L_{\lambda}$ of an excursion of $\Psi_\lambda(f):x\mapsto f(x)-\lambda x$ above its current minimum.
			As a matter of fact the map $\Psi_\lambda(f)$ lies between the two lines with equations $y=-\lambda_x+ \max f$ (above) and $y=-\lambda_x +\min f$ (below). A simple picture shows that the excursions above the current minima of a function in this band is $O(1/\lambda)$.
			
			\itr Observe that $\lambda\mapsto X_n^{(\lambda)}$ is non-increasing. Take $\Lambda$ large enough such that $\P(X^{(\lambda)}\geq \varepsilon/2)\leq \varepsilon$; by $(ii)$, for $n$ large enough, $|\P(X^{(\lambda)}\geq \varepsilon/2)-\P(X^{(\lambda)}_n\geq \varepsilon/2)|\leq \varepsilon/2$ which allows concluding (we also used the fact that $X^{(\lambda)}$ has an atomless distribution).
			\eir
		\end{proof}
		
		\subsubsection{Proof of Theorem \ref{theo:dqgreht}}
		
		Theorem \ref{theo:dqgreht} is a consequence of the three following Lemmas, and of Remark \ref{rem:fqqf} which explains that for all $\lambda>0$ fixed, all $\delta,\delta'>0$, for all small enough $`e>0$,
		\[\P\l(\l|\l\langle 1_{[0,\varepsilon]}\Psi_{\infty}, M_{\lambda}\r\rangle\r| \geq \delta\r)\leq \delta'.\]
		Recall that $t_n(\lambda) := \floor{n-\lambda\sqrt{n}}$, and from \eref{eq:mes2} (or \eref{eq:CST}),
		\[ C_n(\lambda):={\sf Cost}_n(t_n(\lambda)) = \sum_k \sum_{j=1}^{{\sf Mes}_n^{(t_n(\lambda))}(\{k\})} {\sf UCost}_j^{(k)}.\]
		A strategy to analyse this sum consists in using the convergence of $\Theta_n^{(\lambda)}$, and then working conditionally on ${\sf Mes}_n^{(t_n(\lambda))}(\{k\})$, somehow. The convergence of $\Theta_n^{(\lambda)}$ captures only the blocks of linear size. By decomposing
		\[{\sf UCost}_j^{(k)}=\Psi(k)+ \l({\sf UCost}_j^{(k)}-\Psi(k)\r)\] along its mean and its fluctuation around its mean, we naturally write $C_n(\lambda)$ as a sum of 4 terms:
		
		For $\varepsilon\in(0,1)$, $\lambda\geq 0$ fixed, set
		\[ C_n(\lambda):= L_n(\lambda,`e)+  S_n(\lambda,`e)+  LF_n(\lambda,`e)+  SF_n(\lambda,`e)\]
		where  
		\ben\label{eq:quantities}
		\bpar{ccl}L_n(\lambda,`e)&=& \dis\sum_{ k\geq n\varepsilon} \Psi(k) \Mes_n^{(t_n(\lambda))}(\{k\}),\\
		S_n(\lambda,`e)&=& \dis\sum_{k < n\varepsilon}\Psi(k) \Mes_n^{(t_n(\lambda))}(\{k\}),\\
		SF_n(\lambda,`e)&=& \dis\sum_{k < n\varepsilon} \sum_{j=1}^{\Mes_n^{(t_n(\lambda))}}\l({\sf UCost}_j^{(k)}-\Psi(k)\r),\\
		LF_n(\lambda,`e)&=& \dis\sum_{ k\geq n\varepsilon}\sum_{j=1}^{\Mes_n^{(t_n(\lambda))}}\l({\sf UCost}_j^{(k)}-\Psi(k)\r)\epar
		\een
		(symbols $L$ and $S$ are chosen to represent ``large'' and ``small'' blocks contributions, and the $F$ for the ``fluctuation terms'', conditional on $\Mes_n^{(t_n(\lambda))}$).

		Using Theorem \ref{theo:vague}, by the Skorokhod representation theorem, there exists a probability space $(\bar{\Omega},{\cal A},\P)$, on which are defined some copies  $\l(\bar{\Theta}_n^{(\lambda)},n\geq 0\r)$ of $\l(\Theta_n^{(\lambda)},n\geq 0\r)$ that converges a.s. to a copy ${\bar{M}}_\lambda$ of ${M}_\lambda$ (for the vague topology).
		
		\begin{lem}[Analysis of $L_n(\lambda,`e)$] Assuming the hypothesis of Theorem \ref{theo:dqgreht}, for $`e>0$ and $\lambda\geq 0$ fixed, on the probability space $(\bar{\Omega},{\cal A},\P)$,
			\[\frac{L_n(\lambda,`e)}{\alpha_n n^{1/2}}\proba \langle \1_{[\varepsilon,1]}\Psi_\infty, \bar{M}_\lambda\rangle.\]
		\end{lem}
		\begin{proof}
			Observe that since $\Psi_{\infty}$ is continuous on $(0,1]$,
			\[\|\Psi_\infty\|_{\infty,`e} :=\max_{x\in[`e,1]}|\Psi_\infty(x)|<+\infty.\] Therefore,
			\ben\label{eq:cosdq}
			\langle 1_{[\varepsilon,1]}|\Psi_\infty|,M_{\lambda}\rangle&=&\int_0^1\int_0^{e_x} \frac{1_{R(x,y)-L(x,y)\geq `e }}{R(x,y)-L(x,y)}\bigl|\Psi_{\infty}(R(x,y)-L(x,y))\bigr| \d y\d x\\
			&\leq & \|\Psi_\infty\|_{\infty,`e} \int_0^1\int_0^{e_x} (1/`e)  \d y\d x\\
			\label{eq:bound12} &\leq & \|\Psi_\infty\|_{\infty,`e} \int_{0}^1 \se(x)/`e \d x<+\infty \textrm{~a.s.}.
			\een
			With the same argument (taking somehow $\Psi_\infty=1$), one sees that we have also
			\beq\label{eq:finitNDSJN} \langle 1_{[\varepsilon,1]},M_{\lambda}\rangle<+\infty \textrm{~a.s.}. \eq
			Now, {using that $\Mes_n^{(t_n(\lambda))}( {nx})$ is zero when $ {nx}$ is not an integer,} we have (by \eref{eq:qsUIdq}),
			\ben\label{eq:qgefa} \frac{L_n(\lambda,`e)}{n^{1/2} \alpha_n}&=& \sum_{k\geq n\varepsilon}  \frac{\Psi(  k)}{\alpha_n} \frac{1}{\sqrt{n}}\Mes_n^{(t_n(\lambda))}( k)\\
			&=& \left\langle 1_{[`e,1]}\Psi_{\infty}, \Theta_n^{(\lambda)}\right\rangle+\left\langle 1_{[`e,1]} \l( \frac{\Psi( \floor{nx})}{\alpha_n}-\Psi_\infty(x)\r), \Theta_n^{(\lambda)}\right\rangle.
			\een
			By the a.s.  convergence of $ \Theta_n^{(\lambda)}$ to $M_{\lambda}$ (for the vague topology on $(0,1)$), the fact that $\Psi_{\infty}$ is continuous on $(0,1]$ (and that $ M_{\lambda}$ has a.s. no atom at $\varepsilon$, and by \eref{eq:bound12}),
			\[ \left\langle 1_{[`e,1]}\Psi_{\infty}, \Theta_n^{(\lambda)}\right\rangle\as \left\langle 1_{[`e,1]}\Psi_{\infty}, \bar{M}_\lambda\right\rangle.\]
			Since $\frac{\Psi(\floor{nx})}{\alpha_n}\to \Psi_{\infty}(x)$ uniformly on $[\varepsilon,1]$ and by \eref{eq:finitNDSJN}, 
			\[\left\langle 1_{[`e,1]} \l( \frac{\Psi( \floor{nx})}{\alpha_n}-\Psi_\infty(x)\r), \Theta_n^{(\lambda)}\right\rangle \proba 0.\]
		\end{proof}

		\begin{lem}[Analysis of $S_n(\lambda,`e)$]\label{lem:as}
			Let $\lambda \geq 0$ and  $\eta, \delta>0$. If $`e$ is small enough, for all $n$ large enough,
			\[\P\l(\frac{S_n(\lambda,`e)}{\alpha_n n^{1/2}}\geq \eta\r)\leq \delta.\]
		\end{lem}
		\begin{proof} It suffices to prove the Lemma for $\lambda=0$ (since $\lambda\mapsto S_n(\lambda,`e)$ is non-increasing). 
			Recall that $\bs_{m}$ is the size of cluster receiving the $m+1$th car: it is of size 0 if the chosen place $\bc_{m+1}$ is free. The length of the cluster is observed before the insertion of this $m+1$th car. We have
			\ben\label{eq:sdq1}\P(\bs_m=k)&=& \1_{k=0}\frac{n-m}{n}+ \1_{0<k\leq m}{{m}\choose{k}} n  (k+1)^{k-1} (n-k-1)^{m-k}\frac{n-1-m}{n-k-1}\frac{1}{n^{m}} \frac{k}{n}\\ 
			&=& \label{eq:sdq3}\1_{k=0}\frac{n-m}{n}+ \1_{0<k\leq m}\P\l( {\sf Bin}\l(m,\frac{k+1}{n}\r)=k\r) \frac{n-1-m}{n-k-1} \frac{k}{k+1}
			\een
			where ${\sf Bin}(m,\frac{k+1}{n})$ is a binomial random variable with parameters $m$ and $(k+1)/n$.\\
			The proof of the first formula is mainly combinatorial (while the passage to the other is algebraic), the case $k=0$ being clear, let us see the case $k>0$. We will add the total weights of the parking histories leading to the event $\{\bs_{m}=k\}$.\\
			-- The identities of the $k$ cars of the cluster of size $k$ is any subset of $\{1,\cdots,m\}$ of size $k$: this the factor ${{m}\choose{k}}$.\\
			-- Now, let us place this cluster somewhere on $\Z/n\Z$ by fixing the empty place $e$ preceding it: $n$ choices.\\
			-- The $k$ cars we just talk about, must have fallen between $(e+1) \mod n,\cdots (e+k)\mod n$, and all of them must be finally parked on the same places: the number of such parking schemes is known to be $(k+1)^{k-1}$.\\
			-- The remaining $(m-k)$  cars must have arrived in $e+k+2\mod n,\cdots, e\mod n$ (in total $n-(k+1)$ places), and left the last place free. To count the number of compatible positions of arrivals of cars, we use the (already used) trick of working with the set  $\Z/(n-k-1)\Z$ instead and demanding position zero to be empty: \modd{$(n-k-1)^{m-k}$ is the number of ways to put $m-k$ cars on $n-k-1$ places}, and $\frac{n-1-m}{n-k-1}$ is the proportion of those letting zero free (by invariance under rotation). \\
			Finally, we introduce the weight: $\frac{1}{n^{m}}$ is the probability of each history of the first $m$ cars, and $k/n$ is the probability that the $m+1$th car arrives in the block of size $k$, we were talking about all along.
			
			We have, by bounding $k/(k+1)\leq 1$, and $1/(n-1-k)\leq 1/(n(1-`e))$,
			\ben \E(S_n(1,\varepsilon))&=& \sum_{k < n\varepsilon}\sum_{m=1}^n \Psi(k) \P(\bs_m=k)\\
			\label{eq:ogff}&\leq& \sum_{m=1}^n \frac{n-m}{n}\Psi(0)+\frac{1}{n(1-`e)}\sum_{k=1}^{n\varepsilon}\Psi(k) \sum_{m=1}^n \P\l({\sf Bin}\big(m,\frac{k+1}{n}\big)=k\r)(n-m)
			\een
			Observe that the infinite sum $\sum_{m\geq 0} \P\l( {\sf Bin}\big(m,\frac{k+1}{n}\big)=k\r)=1+n/(k+1)$ is equal to the mean time, a random walk $(S_j^{(k)},j\geq 0)$ with Bernoulli increments with parameter $(k+1)/n$ passes at level $k$ (and the number of passages is a geometric random variables, plus one).
			Roughly speaking, the first time $\tau^{(k)}_k$ that such a random walk reaches $k$ is around time $m$ such that $m(k+1)/n$ is close to $k$, so, when $m$ is around $n$, the factor $n-m-1$ in the rhs of \eref{eq:ogff} cannot be bounded by $n$, otherwise, the upper bound would be much too large, and useless.  Rewrite
			\ben
			\sum_{m=1}^n \P\l({\sf Bin}\big(m,\frac{k+1}{n}\big)=k\r)(n-m)&=& \E\l(\sum_{m=1}^n \1_{S_m^{(k)}=k} (n-m)\r)\\
			&\leq & \E( \sum_{m=1}^n \1_{S_m^{(k)}=k} (n-\tau^{(k)}_k)_+)\\
			&\leq & \E( (\tau^{(k)}_{k+1}-\tau^{(k)}_{k})  (n-\tau^{(k)}_k)_+)\\
			&\leq & \frac{n}{k+1} \E ( (n-\tau^{(k)}_k)_+)\\
			\label{eq:IJZIdqd}&\leq &  \frac{n}{k+1} \E ( |n-\tau^{(k)}_k|), 
			\een
			where we define more generally $\tau^{(k)}_j$ as the first time at which the random walk (still with Bernoulli increments with parameter $(k+1)/n$) reaches $j$. The penultimate equality comes from the strong Markov property.
			Since $\tau^{(k)}_k\eqd \sum_{j=1}^k G_j^{(k)}$ where the $G_j^{(k)}$ are geometric with parameter $(k+1)/n$; this is a negative binomial random variable with mean  $\frac{kn}{k+1} $ and variance $k (1-(k+1)/n)/((k+1)/n)^2=kn(n-k-1)/(k+1)^2$.
			Write, using $\E\l[|X|\r]\leq \sqrt{\E\l[X^2\r]}$,
			\be
			\E \l[ |n-\tau^{(k)}_k|\r]&\leq& \frac{n}{k+1}+ \E \l[ |n\frac{k}{k+1}-\tau^{(k)}_k|\r]\\
			&\leq& \frac{n}{k+1}+ {\E \l[ \l(n\frac{k}{k+1}-\tau^{(k)}_k\r)^2\r]^{1/2}}\\
			&\leq& \frac{n}{k+1}+ \l(\frac{kn(n-k-1)}{(k+1)^2}\r)^{1/2}\\
			&\leq& \frac{n}{k+1}+ \frac{n^{1/2}(n-k-1)^{1/2}}{\sqrt{(k+1)}}\\
			&\leq& n/k +  n /\sqrt{k} ~\leq~ 2n/\sqrt{k}.
			\ee 
			Now, we conclude by \eref{eq:IJZIdqd}: for $\varepsilon\in(0,1/2)$, 
			\ben \E(S_n(1,\varepsilon)) &\leq& \frac{n}2 \Psi(0)
			+\frac{1}{n(1-`e)}\sum_{k=1}^{n\varepsilon} \Psi(k) \frac{n}{k} \frac{2n}{\sqrt{k}}\\
			&\leq& \frac{n}2 \Psi(0)
			+4n\sum_{k=1}^{n\varepsilon} \Psi(k) / k^{3/2}
			\een
			Condition \eref{eq:cond233} is designed so that, for $n$ large enough, this is $\lim_n\E(S_n/(\alpha_n n^{1/2}))$ is as small as wanted, up to take a small $`e>0$; and then the Markov inequality allows us to conclude the proof. \end{proof}

		\paragraph{Fluctuation analysis.}
		
		\modd{Recall Note \ref{note:note1}: the re-centered cost array $({\sf UCost}_j^{(k)}-\Psi(k),j\geq 0, k\geq 0)$ have centered and independent entries: by conditioning on the block sizes, we deduce that, for all $t$, }% are all centered, mutually independent, and independent of $\Mes_n^{(t)}$ for all $t$, we have
		\[\E(SF_n(\lambda,`e))=\E(LF_n(\lambda,`e))=0.\]
		Let us control the variance of $SF_n(\lambda,`e)$ and of $LF_n(\lambda,`e)$. We have
		$SF_n(\lambda,`e)= \sum_{j=1}^{t_n(\lambda)}X_{j}^{(\bs_j)}\1_{\bs_j\leq n\varepsilon}$ where the $(X_\ell^{(b)},b\geq 0,\ell\geq0)$ are independent random variables, and for all $\ell$ and $b$,
		\[X_\ell^{(b)}\eqd{\sf UCost}_1^{(b)}-\Psi(b);\] we have
		\[
		\Var(SF_n(\lambda,`e)) = \E\l[ \l(\sum_{k=1}^{t_n(\lambda)} X_{k}^{(\bs_k)}\1_{\bs_k\leq n`e} \r)^2\r]\]
		and by conditioning by $\Mes_n^{(t_n(\lambda))}$,
		using that $\E\l[X_{\ell}^{(b)}\r]=0$,$\Var(X_\ell^{(b)})=\E\l[(X_{\ell}^{(b)})^2\r]=V(b)$,  for all $b$, and by independence of the $X_i^{(j)}$, we obtain
		\ben\label{eq:varcontrol}\bpar{ccl}
		\Var(SF_n(\lambda,`e)) &=& \E\Big[ \sum_{k< n`e} \Mes_n^{(t_n(\lambda))}(\{k\}) V(k) \Big],\\
		\Var(LF_n(\lambda,`e)) &=& \E\Big[ \sum_{k\geq n`e} \Mes_n^{(t_n(\lambda))}(\{k\}) V(k) \Big].
		\epar\een
		\begin{lem} Under the hypothesis of Theorem \ref{theo:dqgreht}, for all $\eta,\delta>0$, if $`e>0$ is small enough, for $n$ large enough,
			\beq\label{eq:grdqs} \P\l(\l|\frac{SF_n(\lambda,`e)}{ \alpha_n n^{1/2}}\r|\geq \eta\r)\leq \delta\eq
			and \beq  \frac{LF_n(\lambda,`e)}{ \alpha_n n^{1/2}}\proba 0.\eq
		\end{lem}
		\begin{proof}
			\noindent \textbf{Control of $SF_n(\lambda ,\varepsilon)$.}
			Since these variables are centered, it suffices to control properly the variance.
			The variance $\Var(SF_n(\lambda,`e))$ given in \eref{eq:varcontrol} has exactly the same form as $\E(SF_n(\lambda,`e))$ as studied in the proof of Lemma \ref{lem:as}, with $V(k)$ instead of $\Psi(k)$.
			By the same analysis to that of $\E(S_n(1, `e))$ we get
			\[\Var(SF_n(1,`e))\leq \frac{n}{2} V(0)+ 4n   \sum_{k=1}^{n\varepsilon} \frac{V(k)}{k^{3/2}}.\]  
			Condition \eref{eq:cond234} is designed so that $\limsup_n \Var({SF}_n(\lambda,`e)/(\alpha_n n^{1/2}))=o(`e)$, and the Bienaym\'e-Chebyshev inequality allows to deduce \eref{eq:grdqs}.  
			
			\noindent  \textbf{Control of $LF_n(\lambda,`e)$.}
			Consider the variable
			\ben
			Y_{n, `e}:=\E\l[LF_n(\lambda,`e)^2~\middle| \Mes_n^{(t_n(\lambda))}\r]&=& \E\l[ \l(\sum_{k=1}^{t_n(\lambda)} X_{k}^{(\bs_k)}\1_{\bs_k\geq n`e} \r)^2~\middle|~\Mes_n^{(t_n(\lambda))}\r]\\
			&=& \sum_{k\geq n`e} \Mes_n^{(t_n(\lambda))}(\{k\}) V(k);
			\een
			this variable is the variance of $LF_n(\lambda,`e)$ conditional on $\Mes_n^{(t_n(\lambda))}(\{k\})$.
			We then have
			\beq
			Y_{n, `e}\leq Y_{n,`e}':= n\int_{x\geq `e} \Mes_n^{(t_n(\lambda))}(\{\floor{nx}\}) V(\floor{nx})\d x,
			\eq
			where the introduction of $Y_{n,`e}'$ is here only to treat the boundary effect near $`e$, so that,
			\beq\frac{Y_{n, `e}'}{n\alpha_n^2} = \l\langle\frac{1}{\sqrt{n}\alpha_n^2} V(\floor{n.}),  \Theta_n^{(\lambda)}\r\rangle.\eq
			By hypothesis $\lim_n\sup_{x\geq `e}\frac{1}{\sqrt{n}\alpha_n^2} V(\floor{nx})=0$, and by Theorem \ref{theo:vague}, $\Theta_n^{(\lambda)} \dd M_\lambda$ for the vague topology. This implies that $\frac{Y_{n, `e}}{n\alpha_n^2}$ converges to zero in probability, from what the conclusion follows.\end{proof}

		\subsection{Discussion: Cost for general diffusion processes}
		\label{sec:qegrhtyu}
		
		Theorem \ref{theo:dqgreht}, which is valid uniquely in the parking case (where the diffusion stops when the car leaves the initial occupied CC containing the arrival points), shows that even in this simplest case, the analysis is not trivial.	In general, defining of the ``total cost'' of a general diffusion process begins with the question of defining the unitary cost model, corresponding to the insertion and dispersion of a single mass. An important issue arises: the dispersion of a single mass may result in the coalescence of many occupied CC. Hence, there are mainly two kinds of ``unitary cost functions'':\\
		$(a)$ those that depend uniquely of the size $\bs_k$ of the occupied CC that contains the arrival position $\bu_k$, \\
		$(b)$ those that depend of all the occupied CC involved in the dispersion of each mass (for example, by using all the information contained in interval relaxation process $\IR_k$).
		
		Cost models that falls in the category $(a)$ can be studied as for the parking model (under the same conditions), if their collecting paths converge to the Brownian bridge $b_\lambda$ (as stated in Proposition \ref{pro:dqgsr}); the lemmas and theorems we used for the parking model would need to be adapted to this new model\footnote{For example, $\Theta_n^{(\lambda)}$ that encodes the insertion sequence $(\bs_k, 0\leq k\leq t_n(\lambda))$ should converge toward $M_{\lambda}$   possibly up to a constant, since the mean of the masses play a role in the time normalisation.}.
		The same method could be applied to the subset of models of type $(b)$ for which, at the first order, the unitary cost depends on $\bs_k$ (for example, a cost that would depend on the size $\bs_k'$ of the occupied CC that contains $u_k$, but after dispersion, should behave as if the size was taken before insertion, for many cost models, since very few coalescence events concern several linear size occupied CC).

		Many models of type $(b)$ depend on small CC, and should not be accessible to this kind of analysis. For example, a unitary cost that  depends on the product of each of the occupied CC that merged during the $k$th insertion would bring important additional problems.

		\section*{Conclusion}
		
		The universal properties of the valid \CDMs described here are one--dimensional results. In this case, the geometry of the CC are well encoded by their lengths.
		
		As previously discussed, our results regarding  the universality of the distribution of $(\bF{k},\bO{k})$ for a fixed $k$, for all valid \CDM, can be extended to other diffusion models. For example, this can be applied to models that treat all masses simultaneously, as long as all along the diffusion, the connected components grow independently from the others. \par
		When there is an infinite number of masses with total mass $<1$, some of our results still apply, and a.s. the number of connected components of $\bO{+\infty}$ is infinite too.  The analysis of the connected component lengths, or of the fractal dimensions of the final occupied set should be accessible for some special distribution of masses.

		Finally, in 2D (on a torus $(\R/\Z)^2$, for example), there are no analogue of Theorem \ref{theo:excha}$(i)$ (with the same degree of generality) which states that in 1D, the distribution of $(\bO{k},\bF{k})$ depends only on the masses, and not on the details of the valid \CDM considered: the coalescence induced by a \CDM producing Euclidean balls (say, making components grow on their boundary, with the same speed everywhere) is nothing compared to \CDM which would produce very thin hairs visiting a large part of  $(\R/\Z)^2$. However, we can still define, beyond the piling propensity, the cluster distribution  as the distribution of the dispersion of some masses conditioned to form a single CC. The distribution of the configuration at time $k$ can still be written as a kind of product of the clusters probabilities, weighted by the Lebesgue measure of the translation space (that measures the ``amount'' of position at which we can place these clusters while avoiding intersection).
		This could maybe be used to study the statistical properties of some \CDM.

		\begin{appendices}
			
			%\section{Appendix}
			
			\section{Table of notation}\label{sect:table_of_notations}
			
			{\modifs
				
				This table is obviously non-exhaustive, but contains the major notation appearing throughout the paper.

				\begin{center}
					\begin{tabular}{|m{3.15cm}| m{7.5cm}| m{2.4cm}|} 
						\hline
						Notation & Meaning & Definition in the paper \\ [0.5ex] 
						\hline\hline
						CC & Connected Component  &  Section \ref{debut} \\ 
						\hline
						CDM & Continuous Dispersion Model &  Section \ref{sec:VCDM}  \\ 
						\hline
						RDCS & Right Diffusion at Constant Speed &  Section \ref{sec:RDCS} \\ 
						\hline
						${\cal C}$ &  (continuous) unit circle $\R/\Z$  & Section \ref{sec:VCDM} \\
						\hline 
						$\fleche{[a,b]}$ & Circular interval from $a$ to $b$ in $\R\backslash\Z$  & Section \ref{sec:VCDM}  \\ 
						\hline
						$\flecheu{d}(a,b)$& canonical positive direction from $a$ to $b$ & Section \ref{sec:VCDM}\\
						\hline 
						$W(m[k])$ & total weight of the $k$ first masses $\sum_{i=0}^{k-1} m_i$ & Equation \eqref{eqn:defWmk}\\
						\hline
						$\bO{k}$ & Occupied set after the first $k$ masses have been dispersed & Section \ref{sec:VCDM} \\
						\hline
						$\bF{k}$ & Free set after the first $k$ masses have been dispersed & Section \ref{sec:VCDM} \\
						\hline
						$ \bN_k$ & number of free CC when $k$ masses have been dispersed & Equation \eqref{eqn:def_Nk}\\
						\hline
						$\bO{k}_j$ (resp. $\bF{k}_j$)& $j$th occupied (resp. free) CC in $\bO{k}$ (resp. $\bF{k}$) & Equation (\ref{eqn:def_ordre_OetF}) \\
						\hline
						$(\IR_k(t),k\leq t \leq k+1)$ & Interval Relaxation process (the circular interval that encodes the growing CC between time $k$ and $k+1$, during the dispersion of mass $m_k$) & Section \ref{sec:VCDM} \\
						\hline
						$\UCost^{(s)}_.$& Unitary cost associated to the arrival of a in a block of size $s$& Page \pageref{eq:CST}\\
						\hline
						$\Psi(s)$; $V(s)$& Expectation and variance of $\UCost^{(s)}$ & Theorem \ref{theo:dqgreht} \\ \hline
						$\Cost_n(t)$& Global cost of the insertion of the first $t$ cars in a parking of size $n$ & Equation \eqref{eq:CST}\\
						\hline
						$\mathcal{L}(\se)$; $ \under({\se})$; $R(x,y)$; $L(x,y)$; $S(x,y)$& parameters of the Brownian excursion $\se$ & Figure \ref{fig:theline} and paragraph above\\
						\hline
						$M_{\lambda}(\se) $; $M_{\lambda} $& defined by setting $\langle f, M_{\lambda}(\se) \rangle$ , for any $f$& Lemma \ref{lemma:calcul_ps_Mlambda} \\
						\hline
						${\sf Mes}_n^{(  t)}$&(random) measure encoding the cluster size insertion sequence  & Equation \eqref{eq:mes1} \\
						\hline
						$\Theta_n^{(\lambda)}$ & (random) measure corresponding to some normalisation of ${\sf Mes}_n^{(t)}$ : $\Theta_n^{(\lambda)}(.)= \frac{1}{\sqrt{n}} {\sf Mes}_n^{\l(t_n(\lambda)\r)}(n . )$& Equation \eqref{eq:mes3}  \\
						\hline
						$L_n(\lambda,`e)$; $S_n(\lambda,`e)$; $LF_n(\lambda,`e)$; $SF_n(\lambda,`e)$& contributions to the cost, decomposed between the contribution from large and small blocks, both from the mean cost and the fluctuation around the mean  & Equation \eqref{eq:quantities} \\ \hline
					\end{tabular}
				\end{center}

			}

			\section{Justification of fluid approximation of Section \ref{sec:IPLD}}\label{sec:annex_FA}
			In order to justify the approximation of the interval evolution by this fluid limit, several methods are possible. We use the results on P\'olya urns. Since the model is identical to a P\'olya urn starting from an urn composition $(\ba,\bb)$ (number of balls red/blue, with replacement matrix $\big[\begin{smallmatrix} 0 & 1\\ 1 & 0 \end{smallmatrix}\big]$); this model is called in the literature the adverse-campaign model (Friedman) (see e.g.\ Section 2.2 in Flajolet \& al.\ \cite{MR2509623}) and it is accessible to exact computations, and in fact, it may be used as an exercise in martingale lectures. Letting $[\ba(0),\bb(0)]$ be given, at time $t$,  $L(t):=\bb(t)+\ba(t)=t+(\bb(0)+\ba(0))$, and then knowing $\bb(t)$ is sufficient to characterise the configuration at time $t$. Since conditional on $\bb(t)$, $\bb(t+1)=\bb(t)+\bX_t$ where $\bX_t$ is a Bernoulli random variable with parameter ${(L(t)-\bb(t))}/{L(t)}$  we get that 
			\beq\label{eq:fsgsds} \E\l(\bb(t+1)~|\bb(t)\r)=1+ \bb(t)(1-{1}/{L(t)}).\eq  
			From here, setting $m(t):=\E(\bb(t))$, by taking the expectation in \eref{eq:fsgsds}, we have 
			\[L(t)m(t+1)=L(t-1)m(t)+L(t)=\cdots =L(0)m(1)+\sum_{j=1}^{t}L(j)=L(0)m(1)+tL(0)+t(t+1)/2,\] 
			so that
			\[m(t+1)=\frac{L(0)m(1)+tL(0)+t(t+1)/2}{t+L(0)},\]
			with $m(1)=\bb(0)+\ba(0)/L(0)$.
			Now, the concentration around the mean is a standard consequence of exponential bounds for martingales with bounded increments (for example, Azuma's inequality).  
			Then, assume that the balls have mass $1/M$ with $M$ large, and that $a(0) = M \alpha_0$, $b(0)= M \beta_0$ corresponds to some large urns content, where $\alpha_0, \beta_0>0$. At time $t M$ for some $t>0$, $m(tM+1) /M$ is the mean position of the right extremity of the segment in this rescaled process,
			\[\frac{m(tM+1)}{M} =\frac{1}{M} \frac{(\alpha_0+\beta_0)M \l(\beta_0M+ \alpha_0/(\alpha_0+\beta_0)\r)+tM(\alpha_0+\beta_0)M+(tM)(tM+1)/2}{tM+(\alpha_0+\beta_0)M}\]
			and we get  
			\[\frac{m(tM+1)}{M} \xrightarrow[M\to+\infty]{}  \frac{(\alpha_0+\beta_0) \beta_0+t(\alpha_0+\beta_0)+t^2/2}{t+\alpha_0+\beta_0}\]
			which can be seen to be identical to the formula giving $b(t)$ in \eref{eq:a-b} (in identifying $(\alpha_0,\beta_0)$ with $(a(0),b(0))$).

			\section{Justification of \eref{eq:TsetMeasure} }
			\label{sec:JTset}

			Consider the set $S:=\{ (u,s_0,\cdots,s_{n-2}), u \in [0,M_0+s_0], 0\leq s_0 \leq s_1\leq \cdots s_{n-2}\leq 1-W(m[k])\}$ (the variable $s_j$ cumulates the free length spaces between the $\bO{k}_0$ and $\bO{k}_{j+1}$); the map
			\[\app{\Psi}{S}{{\cal C}^n}{(u,s_0,\cdots,s_{n-2})}{\l(-u, -u+s_0+ M_0,\cdots, -u+s_{n-2}+(M_0+\cdots+M_{n-2})\r)}\] is linear. Its Jacobian determinant is 1 (the linear map has determinant -1), and sends $S$ onto the set $(a_0,\cdots,a_{n-1})$ described in Theorem \ref{theo:dqgsrt}.  
			{We then have  $ {\sf TL}(M[n])=\Leb(S)$, and then, as announced
				\ben
				{\sf TL}(M[n])&=& \int_{0<y_0<\cdots<y_{n-2}<1-W(m[k])} (M_0+ y_0) \d y_0\cdots \d y_{n-2}\\
				&=&M_0 \frac{(1-W(m[k]))^{n-1}}{(n-1)!}+\frac{(1-W(m[k]))^{n}}{(n)!}.\een}
			
			\section{Computation of  $\E\l(\langle x\mapsto x^k ,M_{\lambda}\rangle\r)$}
			\label{sec:CD} 
			We consider first the  case when $f=\Id$, that is  $\E\l(\langle x\mapsto x^1 ,M_{\lambda}\rangle\r)$; we recall that Aldous CRT usual convention is to consider that the CRT is the continuous tree whose contour process is $2\se$.
			
			According to Theorem \ref{theo:formule}, in this case, 
			\beq \langle \Id, M_{\lambda} (\se) \rangle 
			=  \int_{ \lambda}^{+\infty} \sum_{m\in\mathcal{L}(\se)} \left(\ell(m,t )\right)^2  \d t
			\eq
			and we know thanks to \eqref{eq:efette} that the right hand side is equal to $\int_{0}^1 2\se_\lambda(x)-2\underbar{\se}_\lambda(x)\, \d x$. 
			
			We first discuss the case $\lambda = 0$. On the one hand, we know that 
			\begin{equation}
				\int_{0}^1 2\se_0(x) \d x =  \int_{0}^1 2\se(x) \d x= \E_\se\left[2\se(U)\right] \eqd \E_\se\left[d_{{2\se}}(U_1, U_2)\right]
			\end{equation}
			where $U$ is uniform on $[0,1]$ (independent of $\se$), and $U_1$ and $U_2$ are two uniform points taken in the CRT $T_{2\se}$ (in other words, using $2\se$ as a contour process, the points $U_1$ and $U_2$ corresponds by the canonical projection $[0,1]_{\sim 2\se}$, to the image of independent uniform random variables $u_1,u_2$ on $[0,1]$). 
			The second equality stems from the fact that the RHS can be seen as the average height of $2e$ on $[0,1]$. Moreover, in a continuous random tree characterised by $2\se$, $2\se(U)$ is also the distance between a uniform point of the tree and the root. Thus, up to a uniform re-rooting of the tree (operation that  preserves the continuum random tree), it is also the distance between two uniform points. 
			
			On the other hand, at some time $\lambda$, and for $t\geq \lambda$, $\sum_{m\in\mathcal{L}(2\se)} \left(\ell(m,t)\right)^2=\sum_{m\in\mathcal{L}(\se)} \left(\ell(m,t)\right)^2 $ corresponds to the probability that two points chosen uniformly at random in the CRT $T_{2\se}$ are in the same CC at time $t$ in the fragmentation (corresponding to the time reversal of the additive coalescent). As proven by Aldous \& Pitman \cite{MR1675063}, in order to construct this fragmentation process, one may equip the CRT $T_{2\se}$ with a Poisson point process with intensity 1 on the product $Skel(T_{2\se})\times \R^+$, where $Skel(T_{2\se})$ is set of vertices of degree 2 of the tree (the skeleton can be seen as the spanning tree of a countable number of points taken uniformly at random in the CRT, and of its root). A point $\xi=(u,t)$ of this Poisson point process corresponds to a fragmentation event: the node $u$ is removed at time $t$. 
			
			Take two uniform points $U_1$ and $U_2$ in the CRT $T_{2\se}$: the probability that they are still in the same component at time $t$ is, conditional on their distance $D=d_{2\se}(U_1,U_2)$ in the CRT, $\P({\sf Expo}(D)\geq t~|~D) =e^{-Dt}$ (for $t\geq 0$), where ${\sf Expo}(D)$ is an exponential random variable with intensity $D$.  
			Hence, for any $t\geq 0$,
			\beq\label{eq:qget2} \sum_{m\in\mathcal{L}(\se)}  \ell(m,t)^2 \eqd \exp(-t d_{2\se}(U_1,U_2))\eq
			since both sides give the probability that two uniform points are still in the same component of the fragmented CRT at time $t$, during the fragmentation process. 
			As a process in $t$, the two sides of \eref{eq:qget2} are not identical in distribution, since the right hand side is a deterministic function of its value at any time $t=1$. But the relation \eref{eq:qget2} is sufficient to compute the mean:
			\beq\label{eq:rze} \E\Bigl[ \int_{\lambda}^{+\infty} \sum_{m\in\mathcal{L}(\se)}  \ell(m,t)  ^2\d t \Bigr] = \E\Bigl[ \int_{\lambda}^{+\infty} \exp(-t d_{2\se}(U_1,U_2)) \d t\Bigr]=\E\l[\frac{\exp(- \lambda d_{2\se}(U_1,U_2))}{d_{2\se}(U_1,U_2)}\r].
			\eq
			This can be computed, because $d_{2\se}(U_1,U_2)$ is  Rayleigh distributed (see e.g. Aldous \cite[Lemma 21]{MR1207226}, \cite[Section 2.1]{MR1675063}) which is the distribution with density $ x e^{-x^2/2}\1_{x\geq 0}$.
			The RHS in \eref{eq:rze} is
			\[\int_0^{+\infty}  \frac{\exp(- \lambda x/2)}{x/2} x e^{-x^2/2} \d x =  
			\sqrt {\pi/2}{\rm e}^{{ {{\lambda}^{2}}/{2}}}\l({1}-{\rm erf} \left( {\lambda\,/\sqrt {2}}\right)\r). \]
			\textbf{Computation of $\E\l(\langle x\mapsto x^m, M_{\lambda} (\se) \rangle \r)$ for $m\geq 2$} (where $m$ is an integer). We have, by Theorem \ref{theo:formule},
			\[\E\l(\langle x\mapsto x^m, M_{\lambda} (\se) \rangle \r)=\E\Bigl(\int_{ \lambda}^{+\infty} \sum_{k\geq 0} \ell_k^{m+1}(t) \,   \d t\Bigr)\]
			and by the interpretation done before, $\sum_{k\geq 0} \ell_k^{m+1}(t) \,   \d t$ is the probability, given the process $\se$, that $m+1$ uniform random points $U_1,\cdots,U_{m+1}$ are still in the same component of the Poisson fragmentation of the CRT at time $t$.
			When one takes $m+1\geq 2$ random points in the CRT, the total lengths $L_{m+1}$ of the tree spanned by these points satisfies
			\beq\label{eq:for2} \P(L_{m+1}\geq y) = \P(N(y^2/2)<m)=\sum_{j=0}^{m-1}\exp(-y^2/2)(y^2/2)^j/j!,\eq where $N(a)$ is a Poisson random variable with parameter $a$ (by \cite[Theo. 8]{MR1675063}). At time $t$, knowing $L_{m+1}$, the $U_1,\cdots,U_{m+1}$ are still in the same component with probability $\exp(-t L_{m+1})$.
			Hence 
			\[\E\l(\langle x\mapsto x^m, M_{\lambda} (\se) \rangle \r)= \E\l(\exp(-\lambda L_{m+1})\r).\]
			which can be computed thanks to \eref{eq:for2}, and gives the explicit expression,
			\[\E\l(\langle x\mapsto x^m, M_{\lambda} (\se) \rangle \r)= \int_0^{\infty} \exp(-\lambda y) \frac{\d}{\d y} (1- \P(L_{m+1}\geq y))\d y.\]

			\section{Acknowledgements}
			
			We are grateful to the referee for their comments, suggestions and remarks, that helped us to improve the quality of the paper.

		\end{appendices}
		
	\bibliographystyle{abbrv}%\nocite{*}
		\bibliography{biblio}

	\end{document}